\pdfoutput=1
\documentclass{amsart}
\usepackage{calc}
\usepackage[nomessages]{fp}
\usepackage{tikz,amsmath,amssymb}
\usepackage{tkz-euclide} 
\usetikzlibrary{calc,arrows,decorations.markings, patterns}
\usepackage{quiver}
\usepackage[english]{babel}
\usepackage[hidelinks]{hyperref}
\usepackage{amsmath,amsfonts,amssymb}
\usepackage{amsthm}
\usepackage[all]{xy}
\usepackage{ stmaryrd }
\usepackage{ mathrsfs }
\theoremstyle{plain} 

\usepackage{epsfig}
\usepackage{lmodern}
\usepackage{eurosym}
\usepackage{graphicx}
\usepackage{enumitem}
\usepackage{amsbsy}
\usepackage{rotating}
\usepackage[noabbrev,capitalise]{cleveref}
\usepackage{float}
\usepackage{comment}
\usepackage{stackengine,scalerel}
\stackMath
\def\Tleq{\ThisStyle{\mathrel{%
  \stackinset{r}{.7pt+.15\LMpt}{t}{.1\LMpt}{\rule{.4pt}{1.1\LMex+.2ex}}{\SavedStyle\leqslant}%
}}}

\def\TleqE{\ThisStyle{\mathrel{%
  \stackinset{r}{6pt+.15\LMpt}{t}{.1\LMpt}{\rule{.4pt}{1.1\LMex+.2ex}}{\SavedStyle\leqslant_\mathcal{E}}%
}}}

\def\TleqEunder{\ThisStyle{\mathrel{%
  \stackinset{r}{5.4pt+.15\LMpt}{t}{.1\LMpt}{\rule{.4pt}{1.1\LMex+.2ex}}{\SavedStyle\leqslant_\mathcal{E}}%
}}}

\def\TldotE{\ThisStyle{\mathrel{%
  \stackinset{r}{6pt+.15\LMpt}{t}{.1\LMpt}{\rule{.4pt}{1.1\LMex+.2ex}}{\SavedStyle\lessdot_\mathcal{E}}%
}}}

\def\TldotEunder{\ThisStyle{\mathrel{%
  \stackinset{r}{5.4pt+.15\LMpt}{t}{.1\LMpt}{\rule{.4pt}{1.1\LMex+.2ex}}{\SavedStyle\lessdot_\mathcal{E}}%
}}}

\definecolor{midnightblue}{rgb}{0.1, 0.1, 0.44}
\definecolor{plum}{rgb}{0.56, 0.27, 0.52}
\definecolor{Plum}{rgb}{0.56, 0.27, 0.52}
\definecolor{patriarch}{rgb}{0.5, 0.0, 0.5}
\definecolor{darkgreen}{rgb}{0.0, 0.2, 0.13}
\definecolor{darkcerulean}{rgb}{0.03, 0.27, 0.49}
\definecolor{jade}{rgb}{0.0, 0.66, 0.42}
\definecolor{bleudefrance}{rgb}{0.19, 0.55, 0.91}
\definecolor{lava}{rgb}{0.81, 0.06, 0.13}

\hypersetup{
	colorlinks=true,
	linkcolor=darkcerulean,
	citecolor=jade,
	filecolor=magenta,      
	urlcolor=cyan,
	pdfpagemode=FullScreen,
}

\newcommand{\rep}{\operatorname{rep}}
\newcommand{\Cat}{\operatorname{Cat}}
\newcommand{\Int}{\operatorname{Int}}
\newcommand{\GenJF}{\operatorname{GenJF}}
\newcommand{\GenRep}{\operatorname{GenRep}}
\newcommand{\JF}{\operatorname{JF}}
\newcommand{\vJF}{\pmb{\operatorname{JF}}}

\newcommand{\add}{\operatorname{add}}
\newcommand{\Ker}{\operatorname{Ker}}
\newcommand{\Coker}{\operatorname{Coker}}
\newcommand{\Hom}{\operatorname{Hom}}
\newcommand{\Id}{\operatorname{Id}}

\newcommand{\ind}{\operatorname{ind}}

\newcommand{\NEnd}{\operatorname{\mathsf{N}End}}

\newcommand{\End}{\operatorname{End}}
\newcommand{\Ext}{\operatorname{Ext}}

\newcommand{\vdim}{\operatorname{\pmb{\dim}}}

\newcommand{\Ind}{\operatorname{Ind}}

\newcommand{\A}{\mathbf A}
\newcommand{\B}{\mathbf B}
\newcommand{\E}{\mathbf E}

\newcommand{\Tilt}{\operatorname{Tilt}}

\newcommand{\Mat}{\operatorname{Mat}}
\newcommand{\GL}{\operatorname{GL}}
\newcommand{\vGL}{\pmb{\operatorname{GL}}}

\newcommand{\proj}{\operatorname{proj}}
\newcommand{\inj}{\operatorname{inj}}
\newcommand{\Gen}{\operatorname{Gen}}
\newcommand{\Rad}{\operatorname{Rad}}
\newcommand{\Sub}{\operatorname{Sub}}
\newcommand{\GS}{\operatorname{GS}}

\newcommand{\opC}{\operatorname{\pmb{\mathscr{C}}}}
\newcommand{\opJ}{\operatorname{\pmb{\mathscr{J}}}}

\newcommand{\epip}{ \xymatrix{   \ar@{-->>}[r] &  \\} }
\newcommand{\epi}{ \xymatrix{   \ar@{>>}[r] &  \\} }
\newcommand{\p}{ \xymatrix{   \ar@{-->}[r] &  \\} }

\AtEndEnvironment{ex}{}
\newcommand{\CK}{\mathrm{CK}}

\usepackage{stackengine,scalerel}
\stackMath
\def\domleq{\ThisStyle{\mathrel{%
  \stackinset{r}{.75pt+.15\LMpt}{t}{.1\LMpt}{\rule{.3pt}{1.1\LMex+.2ex}}{\SavedStyle\leqslant}%
}}}

\author[B.~Dequêne]{Benjamin Dequêne}
\address[B.~Dequêne]{School of Mathematics, University of Leeds}
\email{b.d.dequene@leeds.ac.uk}

\author[S.~Roy]{Sunny Roy}
\address[S.~Roy]{Département de mathématiques, Université de Sherbrooke}
\email{sunny.roy@usherbrooke.ca}

\date{\today}

\title[Exact struc. and CJR categories in type $A$]{Exact structures and maximal canonically Jordan recoverable subcategories for modules over type $A$ algebras}

\usepackage{thmtools}
\declaretheorem[numberwithin=section,name=Theorem,
refname={Theorem,Theorems},
Refname={Theorem,Theorems}]{theorem}
\declaretheorem[numberlike=theorem,name=Lemma,
refname={Lemma,Lemmas},
Refname={Lemma,Lemmas}]{lemma}
\declaretheorem[numberlike=theorem,name=Proposition,
refname={Proposition,Propositions},
Refname={Proposition,Propositions}]{prop}
\declaretheorem[numberlike=theorem,name=Corollary,
refname={Corollary,Corollaries},
Refname={Corollary,Corollaries}]{cor}
\declaretheorem[numberlike=theorem,name=Conjecture,
refname={Conjecture,Conjectures},
Refname={Conjecture,Conjectures}]{conj}

\declaretheorem[style=definition,numberlike=theorem,name=Definition,
refname={Definition,Definitions},
Refname={Definition,Definitions}]{definition}
\declaretheorem[style=definition,numberlike=theorem,name=Convention,
refname={Convention,Conventions},
Refname={Convention,Conventions}]{conv}

\declaretheorem[style=definition,numberlike=theorem,name=Example,
refname={Example,Examples},
Refname={Example,Examples}]{example}
\declaretheorem[style=definition,numberlike=theorem,name=Example,
refname={Example,Examples},
Refname={Example,Examples}]{ex}

\declaretheorem[style=remark,numberlike=theorem,name=Remark,
refname={Remark,Remarks},
Refname={Remark,Remarks}]{remark}

\usepackage{todonotes}

\newcommand{\new}[1]{\textit{\textbf{\color{patriarch}{#1}}}}
\newcommand{\llrr}[1]{\llbracket #1 \rrbracket}

\begin{document}

\begin{abstract}
On one hand, exact structures were introduced by D. Quillen in the '70s. They can be defined as collections of short exact sequences in a fixed abelian category satisfying additional properties. 

On the other hand, in a recent work, A. Garver, R. Patrias, and H. Thomas introduced Jordan recoverability. Given a bounded quiver $(Q,R)$, a full additive subcategory of $\rep(Q,R)$ is said to be Jordan recoverable if any $X \in \mathscr{C}$ can be recovered, up to isomorphism, from the Jordan form of its generic nilpotent endomorphisms. Such a subcategory $\mathscr{C}$ is said to be canonically Jordan recoverable if, moreover, there exists a precise algebraic procedure that allows one to get back $X \in \mathscr{C}$ from that same Jordan form data.

We introduce a new family of operators, called Gen-Sub operators $\GS_\mathcal{E}$, parametrized by the exact structures $\mathcal{E}$ of abelian categories. After showing some properties of those operators in hereditary abelian categories, by focusing on the setting of modules over path algebras of type $A$ quivers endowed with the diamond exact structure $\mathcal{E}_\diamond$, we establish that the maximal canonically Jordan recoverable subcategories are precisely of the form $\operatorname{GS}_{\mathcal{E}_\diamond}(T)$ for some tilting objects $T$.
\end{abstract}

\maketitle

\tableofcontents

\section{Introduction}
\label{sec:Intro}
\pagestyle{plain}

\subsection{Exact structures on hereditary abelian categories}
\label{ss:ExtStrHAbCat}

Fix $\mathscr{A}$ an abelian category. An \new{exact structure} on an abelian category \(\mathscr{A}\) is a collection \(\mathcal{E}\) of short exact sequences of the form
\[\begin{tikzcd}
	0 & A & B  & C &  0
	\arrow[from=1-1, to=1-2]
	\arrow["f",tail, from=1-2, to=1-3]
	\arrow["g",two heads,from=1-3, to=1-4]
	\arrow[from=1-4, to=1-5]
\end{tikzcd}\]
where \(f\) is an \(\mathcal{E}\)-monomorphism, \(g\) is an \(\mathcal{E}\)-epimorphism, and the sequence adheres to axioms ensuring closure under specific categorical operations (see \cref{def:exact} for details). Originally introduced by D. Quillen in \cite{Q73}, they provide an axiomatic framework that extends the utility of homological algebra techniques to a context constrained by a predefined class of short exact sequences. This framework is notably versatile, allowing exact structures to be characterized in multiple ways: as subfunctors of the bifunctor \(\text{Ext}^1_{\mathscr{A}}\), through projective objects relative to the chosen structure, or via the inclusion of specific Auslander-Reiten sequences. 

A category \(\mathscr{A}\) is \new{hereditary} if it satisfies \(\text{Ext}^i_{\mathscr{A}}(X, Y) = 0\) for all objects \(X, Y \in \mathscr{A}\) and all \(i \geq 2\). Vanishing of higher extensions simplifies homological analysis, as all extensions are captured by \(\text{Ext}^1\). In this paper, we focus on hereditary abelian categories, and, more specifically, on the category \(\text{rep}(Q)\) of finite-dimensional representations of some $ADE$ Dynkin type quiver \(Q\). In this setting, exact structures enable us to control which sequences are deemed ``admissible", thereby isolating subcategories with distinct algebraic and combinatorial features. 

While the standard maximal exact structure on an abelian category encompasses all short exact sequences, alternative structures, such as the \new{diamond exact structure} \(\mathcal{E}_\diamond\) on $\rep(Q)$, introduced in \cite{BHRS24} (see \cref{def:Ediamond}), allows us to refine our focus to sequences that reflect the combinatorial properties of type $A$ quivers.

A pivotal concept in this work is the \new{Gen-Sub operator} \(\text{GS}_\mathcal{E}\), which, given an exact structure \(\mathcal{E}\) and a subcategory \(\mathscr{C} \subseteq \mathscr{A}\), constructs the smallest subcategory containing \(\mathscr{C}\) that is closed under operations such as taking \(\mathcal{E}\)-subobjects and \(\mathcal{E}\)-quotients (see \cref{ss:GenSuboperators} for a formal definition). In the context of hereditary abelian categories, \(\text{GS}_\mathcal{E}\) iteratively builds subcategories by incorporating terms from \(\mathcal{E}\)-exact sequences, offering a systematic method to generate subcategories with stability properties. 

A subcategory $\mathscr{C} \subseteq \mathscr{A}$ will be called
\new{$\mathcal{E}$-adapted} if it behaves well with respect to the exact
structure $\mathcal{E}$, allowing us to relate classical constructions to
their relative counterparts. More precisely, $\mathscr{C}$ is
$\mathcal{E}$-adapted if every short exact sequence in $\mathscr{C}$ is
$\mathcal{E}$-exact, that is, it belongs to the exact structure $\mathcal{E}$.
A formal definition will be given later (see Definition \ref{def:Eadpted}).

\subsection{Canonical Jordan recoverability}
\label{ss:CJRintro}

A. Garver, R. Patrias, and H. Thomas \cite{GPT19} introduced the notion of Jordan recoverability. 

Fix $\mathbb{K}$ an algebraically closed field. Consider $(Q,R)$ a finite connected bounded quiver. Let $E$ be a finite-dimensional representation of $Q$ over $\mathbb{K}$. We study the set $\NEnd(E)$ of nilpotent endomorphisms of $E$. Any $N \in \NEnd(E)$ induces, at each vertex $q$ of $Q$, a nilpotent endomorphism $N_q$ of $E_q$. We can extract from $N$ a sequence of integer partitions $\lambda_q \vdash \dim(E_q)$ from the Jordan block sizes of the Jordan form of each $N_q$. Write $\vJF(N) = (\lambda_q)_{q \in Q_0}$. There exists an open dense set (for the Zariski topology) of $\NEnd(E)$ on which $\vJF$ is constant. Denote by $\GenJF(E)$ this constant, and we refer to it as the \new{generic Jordan form} of $E$.

The map $\GenJF$ is an invariant on the representations of $(Q,R)$, but not a complete one. A full subcategory $\mathscr{C}$, closed under sums and summands, is said to be \new{Jordan recoverable} whenever $\GenJF$ is complete on representations in $\mathscr{C}$.

Usually, checking whether a subcategory is Jordan recoverable is not an easy task. Under some assumptions, we can define an algebraic inverse to $\GenJF$ restricted to a fixed subcategory. We say that a subcategory is \new{canonically Jordan recoverable} if such an algebraic inverse exists. See \cref{ss:CJR} for more details.

In the case where $Q$ is an $A_n$ type quiver, and $R = \varnothing$, we have the following characterization of canonically Jordan recoverable subcategories.

\begin{theorem}[{\cite{D23}}] \label{thm:CJRAlg}
    Let $Q$ be an $A_n$ type quiver for some $n \in \mathbb{N}^*$ and $\mathscr{C}  \subseteq \rep(Q)$ subcategory closed under sums and summands. $\mathscr{C}$ is canonically Jordan recoverable if and only if for any nonsplit short exact sequence \[\begin{tikzcd}
	0 & E & F  & G &  0,
	\arrow[from=1-1, to=1-2]
	\arrow[tail, from=1-2, to=1-3]
	\arrow[two heads,from=1-3, to=1-4]
	\arrow[from=1-4, to=1-5]
\end{tikzcd}\] whenever $E,G\in \mathscr{C}$, the representation $F$ must be decomposable.
\end{theorem}

There exists a combinatorial characterization of such a subcategory $\mathscr{C}$ based on the behavior of the intervals of $\{1,\ldots,n\}$ corresponding bijectively to the indecomposable representations in $\mathscr{C}$. We even have a precise description of the maximal ones for inclusion. We refer to \cref{ss:CJR} for more details.

\subsection{Main results}
\label{ss:Main}

First, we study the Gen-Sub operators applied to the tilting objects of $\mathscr{A}$, and we establish the following result.

\begin{theorem}
\label{thm:maximalmain}
 Let $T \in \mathscr{A}$ be a tilting object. Then $\GS_\mathcal{E}(T)$ is the unique maximal $\mathcal{E}$-adapted subcategory of $\mathscr{A}$ closed under extensions which contains $T$.
\end{theorem}

Then, by considering $Q$ an $ADE$ Dynkin type quiver, and by fixing an exact structure $\mathcal{E}$ on $\rep(Q)$, we defined $\mathcal{E}$-mutations on tilting objects (see \cref{ss:GSandTiltmut}), and an equivalence relation $\approx_\mathcal{E}$ on tilting objects of $\rep(Q)$ as follows. Given two tilting objects $T,T' \in \rep(Q)$, we write $T \approx_\mathcal{E} T'$ whenever $T'$ can be obtained from $T$ by applying a finite sequence of $\mathcal{E}$-mutations. We write $[T]_{\approx_{\mathcal{E}}}$ for the equivalence class of $T$ for $\approx_\mathcal{E}$. We show that, for any tilting object $T \in \rep(Q)$, we can describe $\GS_\mathcal{E}(T)$ as the category additively generated by the tilting objects in $[T]_{\approx_{\mathcal{E}}}$ (see \cref{th:Treaching}).

Reducing ourselves to type $A$ quiver representations, we combine the exact structure point of view with the different descriptions known for the maximal canonically Jordan recoverable subcategories to prove conjectures stated in \cite{D23}, and hence, strongly link those categories to tilting theory. 

\begin{theorem} \label{thm:main1}
    Let $Q$ be an $A_n$ type quiver for some $n \in \mathbb{N}^*$. A subcategory $\mathscr{C} \subseteq \rep(Q)$ is a maximal canonically Jordan recoverable subcategory if and only if there exists a tilting representation $T \in \rep(Q)$ such that $\mathscr{C} = \GS_{\mathcal{E}_\diamond}(T)$. Moreover, for such a tilting representation $T$, we have \[\mathscr{C} = \add \left(\bigoplus_{T' \in [T]_{\approx_{\mathcal{E}_\diamond}}} T'\right).\]
\end{theorem}

\subsection{Outline of the paper}
\label{ss:Outline} We begin, in \cref{sec:Quiver}, by laying crucial recalls to set our background framework: quiver representations and, specifically, type $A$ quiver representations. Those recalls enable us to incorporate relevant examples throughout the article, which will gradually lead us to our main results.

In \cref{sec:ExStr}, given an abelian category $\mathscr{A}$ endowed with an exact structure $\mathcal{E}$, we introduce the Gen-Sub operators which allow one to construct nicely behaved subcategories relative to $\mathcal{E}$ from any subcategory of $\mathscr{A}$. In particular, by assuming that $\mathscr{A}$ is hereditary and Krull--Schmidt, we show that they preserve both $\mathcal{E}$-adaptation and extension-closure properties. We also prove \cref{thm:maximalmain}.

Afterwards, in \cref{sec:JR}, we recall the notions of Jordan recoverability and canonical Jordan recoverability for subcategories of representations over bounded quivers, and we recall the main results from previous works on type $A$ quivers, namely, a combinatorial and an algebraic characterization of canonically Jordan recoverable subcategories of type $A$ quiver representations.

This brings us to \cref{sec:TypeA} where we gather results from previous sections to prove \cref{thm:main1}. We first study the link between $\mathcal{E}$-mutations and $\GS_\mathcal{E}$, for any given exact structure $\mathcal{E}$, on tilting representations of any $ADE$ Dynkin type quiver, before focusing on type $A$ quivers for using results on canonically Jordan recoverable subcategories. Finally, in \cref{sec:further}, we discuss directions of research that we will likely pursue in the near future.

\section{Quiver representations}
\label{sec:Quiver}
\pagestyle{plain}

In this section, we recall the essential definitions that will serve as our main framework throughout this article. 

\subsection{Vocabulary}
\label{ss:Vocab}

A \new{quiver} is a quadruple $Q = (Q_0,Q_1,s,t)$ where:
\begin{enumerate}[label=$\bullet$, itemsep=1mm]
    \item $Q_0$ is a set called \emph{vertex set};
    \item $Q_1$ is a set called \emph{arrow set}; and,
    \item $s,t : Q_0 \longrightarrow Q_1$ are functions called \emph{source} and \emph{target functions}.
\end{enumerate}
Such a quiver $Q$ is said to be \new{finite} whenever $Q_0$ and $Q_1$ are finite sets. In the following, we assume that all our quivers are finite.

A \new{path} of $Q$ is either a formal variable $e_q$ for some $q \in Q_0$ called \emph{lazy path at $q$}, or a finite nonempty word $p = \alpha_k \cdots \alpha_1$ on $Q_0$ such that, for all $i \in \{1,\ldots,k-1\}$, we have $s(\alpha_{i+1}) = t(\alpha_i)$. Given such a path $p$, we define:
\begin{enumerate}[label=$\bullet$, itemsep=1mm]
    \item its \emph{source} $s(p)$ to be $s(\alpha_1)$ (we set $s(e_q) = q$);
    \item its \emph{target} $t(p)$ to be $t(\alpha_k)$ (we set $t(e_q) = q$); and,
    \item its \textit{length} $\ell(p)$ to be $k$ (we set $\ell(e_q) = 0$).
\end{enumerate}
For any arrow $\alpha \in Q_1$, we consider its formal inverse $\alpha^{-1}$ such that $s(\alpha^{-1}) = t(\alpha)$ and $t(\alpha^{-1}) = s(\alpha)$. We denote by $\overline{Q_1}$ the set of formal inverses of arrows in $Q_1$. A \new{walk} of $Q$ is either a lazy path or a finite nonempty word $w = \beta_k \cdots \beta_1$ on $Q_1 \sqcup \overline{Q_1}$ such that, for all $i \in \{1,\ldots,k-1\}$, we have $s(\beta_{i+1}) = t(\beta_i)$. We extend the source, target, and length functions to walks of $Q$ similarly to the case of paths of $Q$. A quiver $Q$ is said to be \new{connected} whenever for any pair $(a,b) \in Q_0$, there exists a walk $w$ of $Q$ such that $s(w) = a$ and $t(w)=b$. From now on, we assume that all our quivers are connected.

\subsection{Representations}
\label{ss:Reps}

Let $\mathbb{K}$ be a field. For geometric purposes, we assume that $\mathbb{K}$ is algebraically closed. The \new{path algebra} of $Q$, denoted by $\mathbb{K}Q$, is the $\mathbb{K}$-vector space generated as a basis by the set of all the paths of $Q$, together with a multiplication which acts as concatenation on the paths: meaning \[p_2 \cdot p_1 = \begin{cases}
    p_2p_1 & \text{if } s(p_2) = t(p_1) \text{; and,} \\
    0 & \text{otherwise.}
\end{cases}\] For any $m \in \mathbb{N}$, we denote by $\mathbb{K}Q_{\geqslant m}$ the (bi-sided) ideal of $\mathbb{K}Q$ generated by all the paths of length greater than or equal to $m$. An ideal $I \subseteq \mathbb{K}Q$ is said to be \new{admissible} whenever there exists an integer $m \geqslant 2$ such that $\mathbb{K}Q_{\geqslant m} \subseteq  I \subseteq \mathbb{K} Q_{\geqslant 2}$. A \new{relation} on $Q$ is a (finite) $\mathbb{K}$-linear combination of paths of length at least two, and having a common source and a common sink. For any subset of relations $R$ on $Q$, we denote by $\langle R \rangle$ the ideal of $\mathbb{K}Q$ generated by $R$. Call \new{bounded quiver} any pair $(Q,R)$ where $Q$ is a quiver and $R$ is a set of relations on $Q$.

Given a bounded quiver $(Q,R)$, we define a \new{representation} of $Q$ (over $\mathbb{K}$) as a pair $E=((E_q)_{q \in Q_0}, (E_\alpha)_{\alpha \in Q_1})$ such that:
\begin{enumerate}[label=$\bullet$, itemsep=1mm]
    \item $E_q$ is a $\mathbb{K}$-vector space, for all $q \in Q_0$;
    \item  $E_\alpha : E_{s(\alpha)} \longrightarrow E_{t(\alpha)}$ is a $\mathbb{K}$-linear ma, for any $\alpha \in Q_1$; and,
    \item by setting $E_p = E_{\alpha_k} \circ \cdots \circ E_{\alpha_1}$ for any path $p = \alpha_k \ldots \alpha_1$ of $Q$, whenever we have a relation $k_1 p_1 + \ldots + k_s p_s$ on $Q$ in $R$, then $k_1 E_{p_1} + \ldots  + k_s E_{p_s} = 0$.
\end{enumerate}
This representation $E$ is \new{finite-dimensional} if, for all $q \in Q_0$, $E_q$ is finite-dimensional. Its \new{dimension vector} $\vdim(E)$ is the vector $(\dim(E_q))_{q \in Q_0} \in \mathbb{N}^{\#Q_0}$. Now, we assume that all the representations we consider are finite-dimensional. Given $E$ and $F$ two representations of $Q$, we define a \new{morphism} $\varphi : E \longrightarrow F$ as a collection $\varphi = (\varphi_q)_{q \in Q_0}$ for linear maps $\varphi_q : E_q \longrightarrow E_q$ such that, for all $\alpha \in Q_1$, we have $\varphi_{t(\alpha)} E_\alpha = F_\alpha \varphi_{s(\alpha)}$. We say that $E$ and $F$ are \new{isomorphic} whenever there exists a bijective morphism $\varphi : E \longrightarrow F$, or equivalently, if there exists a morphism $\varphi : E \longrightarrow F$ such that, for any $q \in Q_0$, $\varphi_q$ is an isomorphism of $\mathbb{K}$-vector spaces. In such a case, we write $E\cong F$. We obtain the structure of a category by endowing the collection of representations of $Q$ with the morphisms between them. In the following, we denote this category by $\rep(Q,R)$. In the case where $R = \varnothing$, we denote this category by $\rep(Q)$. Note that $\rep(Q, R)$ depends on the field $\mathbb{K}$.

 Any (basic) finite-dimensional $\mathbb{K}$-algebra is isomorphic to the quotient algebra $\mathbb{K}Q/\langle R \rangle$ for some bounded quiver $(Q,R)$ where $Q$ is finite, and $\langle R \rangle$ is admissible. Moreover, the category $\rep(Q,R)$ is equivalent to the category of finitely generated left module over $\mathbb{K}Q/\langle R \rangle$. In particular, it implies that $\rep(Q,R)$ is an abelian Krull--Schmidt category with enough projective objects. We refer the reader to \cite{ASS06} for more details. 
 
 Write $E \oplus F$ for the direct sum of $E$ and $F$ in $\rep(Q,R)$. We say that a nonzero representation $X \in \rep(Q,R)$ is \new{indecomposable} if whenever we have $X \cong E \oplus F$ with some $E,F \in \rep(Q,R)$, then $E \cong 0$ or $F \cong 0$. Write $\ind(Q,R)$ for the collection of isomorphism classes of indecomposable representations of $(Q,R)$. Whenever $R = \varnothing$, we denote this collection by $\ind(Q)$. A \new{representation-finite type} bounded quiver is a bounded quiver that has finitely many isomorphism classes of indecomposable representations. 

 For any $\pmb{d} = (d_q)_{q \in Q_0} \in \mathbb{N}^{\#Q_0}$, we consider the following affine space
 \[\rep((Q,R),\pmb{d}) = \prod_{\alpha \in Q_1} \Hom \left(\mathbb{K}^{d_{s(\alpha)}}, \mathbb{K}^{d_{t(\alpha)}} \right) = \prod_{\alpha \in Q_1}  \Mat_{d_{t(\alpha)} \times d_{s(\alpha)}} (\mathbb{K}),\] and refer to it as the \new{representation space} of $(Q,R)$ with dimension vector $\pmb{d}$. Indeed, by choosing a basis for each of the vector spaces, we can identify the representations of $(Q,R)$ with the points of $\rep((Q,R), \pmb{d})$. The isomorphism classes of representations with vector dimension $\pmb{d}$ correspond exactly to the orbits of the group action given by the algebraic group $\vGL_{\pmb{d}}(\mathbb{K}) = \prod_{q \in Q_0} \GL_{d_q}(\mathbb{K})$, which acts on $\rep((Q,R), \pmb{d})$ by changing bases at each vertex.

\subsection{ADE Dynkin type quivers}
\label{ss:Dynkin}

In this paper, we show examples of our results on representation-finite quivers. We shortly recall a complete characterization of them due to P. Gabriel \cite{G72}.

\begin{definition} \label{def:Dynkin}
    Let $n \in \mathbb{N}^*$. A quiver $Q$ is said to be:
    \begin{enumerate}[label=$\bullet$, itemsep=1mm]
        \item of \new{$A_n$ type} if the underlying graph of $Q$ is of the following shape.
         \begin{center}
		\begin{tikzpicture}[-,line width=0.5mm,>= angle 60,color=black,scale=0.8]
			\node (1) at (0,0){$1$};
			\node (2) at (2,0){$2$};
			\node (3) at (4,0){$\cdots$};
			\node (4) at (6,0){$n$};
			\draw (1) -- (2) -- (3) -- (4);
		\end{tikzpicture}
	\end{center}
        \item of \new{$D_n$ type}, for $n \geqslant 4$, if the underlying graph of $Q$ is of the following shape.
        \begin{center}
		\begin{tikzpicture}[-,line width=0.5mm,>= angle 60,color=black,scale=0.8]
			\node (1) at (0,0){$1$};
			\node (2) at (2,0){$2$};
			\node (3) at (4,0){$\cdots$};
			\node (4) at (6,0){$n-2$};
            \node (5) at (7.4142, 1.4142){$n-1$};
            \node (6) at (7.4142, -1.4142){$n$};
			\draw (1) -- (2) -- (3) -- (4) -- (5);
            \draw (4) -- (6);
		\end{tikzpicture}
	\end{center}
        \item of \new{$E_6$}, \new{$E_7$} or \new{$E_8$ type} if the shape of the underlying graph of $Q$ is one of the following.
        \begin{center}
        \scalebox{0.7}{\begin{tikzpicture}[-,line width=0.5mm,>= angle 60,color=black,scale=0.7]
 		\node (1) at (0,0){$1$};
 		\node (3b) at (4,2){$4$};
 		\node (2) at (2,0){$2$};
 		\node (3) at (4,0){$3$};
 		\node (4) at (6,0){$5$};
 		\node (5) at (8,0){$6$};
 		\draw (1) -- (2) -- (3) -- (4) -- (5);
 		\draw (3b) -- (3);
 	      \begin{scope}[xshift=10cm]
 		\node (1) at (0,0){$1$};
 		\node (3b) at (4,2){$4$};
 		\node (2) at (2,0){$2$};
 		\node (3) at (4,0){$3$};
 		\node (4) at (6,0){$5$};
 		\node (5) at (8,0){$6$};
 		\node (6) at (10,0){$7$};
 		\draw (1) -- (2) -- (3) -- (4) -- (5) -- (6);
 		\draw (3b) -- (3);
 	\end{scope}
            \begin{scope}[xshift = 5cm, yshift = -3cm]
 		\node (1) at (0,0){$1$};
 		\node (3b) at (4,2){$4$};
 		\node (2) at (2,0){$2$};
 		\node (3) at (4,0){$3$};
 		\node (4) at (6,0){$5$};
 		\node (5) at (8,0){$6$};
 		\node (6) at (10,0){$7$};
 		\node (7) at (12,0){$8$};
 		\draw (1) -- (2) -- (3) -- (4) -- (5) -- (6) -- (7);
 		\draw (3b) -- (3);
        \end{scope}
 	\end{tikzpicture}}
     \end{center}
    \end{enumerate}
We say that $Q$ is an \new{$ADE$ Dynkin type} quiver whenever $Q$ is one of the above types.
\end{definition}

We have the following main result.

\begin{theorem}[{\cite{G72}}] \label{thm:Gab} Let $Q$ be a quiver. Then $Q$ is a representation-finite type quiver if and only if $Q$ is an $ADE$ Dynkin type quiver.
\end{theorem}

\begin{remark}
    This result holds among bounded quivers $(Q,R)$ with $R = \varnothing$.
\end{remark}

In the following, we recall the precise description of $\rep(Q)$ for $Q$ a type $A$ quiver. The reader can find those results in \cite{S14}.
 
Let $n \in \mathbb{N}^*$. The \new{intervals} of $\{1, \ldots, n\}$ are the sets $\llrr{i,j} := \{i, i+1, \ldots,j \}$ given by all $1 \leqslant i \leqslant j \leqslant n$. If $i=j$, we write $\llrr{i,i} = \llrr{i}$. Denote by $\mathcal{I}_n$ the set of intervals in $\{1, \ldots n\}$. For $K = \llrr{i,j} \in \mathcal{I}_n$, set $b(K) = i$ and $e(K) = j$. Call \new{interval set} any subset of $\mathcal{I}_n$.
	\begin{definition}\label{def:intervalbotandtop}  Let $Q$ be an $A_n$ type quiver. Consider $K,L \in \mathcal{I}_n$ such that $K \subseteq L$. We say that:
		\begin{enumerate}[label = $\bullet$]
			\item $K$ is \new{above} $L$ (\new{relative to} $Q$) if the following two assertions are satisfied:\begin{enumerate}[label = $\bullet$]
				\item $b(K) = b(L)$ or we have the arrow $b(K)-1 \longleftarrow b(K)$ in $Q$;
				\item $e(K) = e(L)$ or we have the arrow $e(K) \longrightarrow e(K)+1$ in $Q$.
			\end{enumerate}
			\item $K$ is \new{below} $L$ (\new{relative to} $Q$) if the following two assertions are satisfied:\begin{enumerate}[label = $\bullet$]
				\item $b(K) = b(L)$ or we have the arrow $b(K)-1 \longrightarrow b(K)$ in $Q$;
				\item $e(K) =e(L)$ or we have the arrow $e(K) \longleftarrow e(K)+1$ in $Q$.
			\end{enumerate}
		\end{enumerate}
	\end{definition}
	\begin{ex} Consider the following quiver.
		\begin{center}
			\begin{tikzpicture}[->,line width=0.5mm,>= angle 60,color=black,scale=0.8]
				\node (Q) at (-1,0){$Q =$};
				\node (1) at (0,0){$1$};
				\node (2) at (2,0){$2$};
				\node (3) at (4,0){$3$};
				\draw (1) -- (2);
				\draw (2) -- (3);
			\end{tikzpicture}
		\end{center}
		Then $\llrr{2}$ is above $\llrr{2,3}$ and below $\llrr{1,2}$.
	\end{ex}
	Note that any interval is above and below itself, relative to all $A_n$ type quivers.
	Let $Q$ be a quiver of $A_n$ type. To any interval $K \in  \mathcal{I}_n$, we consider $X_K$ the representation of $Q$ defined as it follows:
	\begin{enumerate}[label = $\bullet$]
		\item $(X_K)_q = \mathbb{K}$ if $q \in K$, $(X_K)_q = 0$ otherwise;
		\item $(X_K)_{\alpha} = \Id_\mathbb{K}$ if $\alpha$ is such that $\{s(\alpha), t(\alpha)\} \subseteq K$, $(X_K)_{\alpha} = 0$ otherwise;
	\end{enumerate}
	Note that $X_K$ is an indecomposable representation of $Q$ for all $K \in \mathcal{I}_n$.
	\begin{theorem}\label{thm:indecandmorphtypeA} Let $Q$ be an $A_n$ type quiver. \begin{enumerate}[label = $(\alph*)$]
			\item \label{indectypeA} The isomorphism classes of indecomposable representations of $Q$ are in bijection with $\mathcal{I}_n$; more precisely, they are described by indecomposable representations $X_K$ for $K \in \mathcal{I}_n$;
			\item \label{morphtypeA} The homomorphism space between two indecomposable representations of $Q$ is of dimension at most one; more precisely, $\Hom(X_K, X_L)$ is nonzero if and only if  there exists an interval $J$ such that $J$ is above $K$ and below $L$ relative to $Q$; if such an interval exists, it is unique and 
			$\Hom(X_K,X_L)$ consists of scalar multiples of the morphism $\phi = (\phi_q)_{q \in Q_0}$ such that $\phi_q = \Id_\mathbb{K}$ if $q \in J$ and $\phi_q = 0$ otherwise. 
		\end{enumerate}
	\end{theorem}
	In the light of the previous result : \begin{enumerate}[label = $\bullet$]
		\item for all interval sets $\mathscr{J} \subseteq \mathcal{I}_n$ and all quivers $Q$ of $A_n$ type, write $\Cat_Q(\mathscr{J})$ for the subcategory of $\rep(Q)$ additively generated by $X_K$ for $K \in \mathscr{J}$;
		\item for all quivers $Q$ of $A_n$ type and for all nonzero subcategories $\mathscr{C}$ of $\rep(Q)$, let $\Int(\mathscr{C})$ be the interval set of $\mathcal{I}_n$ consisting of intervals $K$ such that $X_K \in \mathscr{C}$.
	\end{enumerate}
	Under \cref{conv:addsubcat}, the subcategories we are considering are additively generated by $X_K$ for $K \in \mathscr{J}$ for some  $\mathscr{J} \subset \mathcal{I}_n$.
	
	Hence, for any $A_n$ type quiver $Q$, for all $\mathscr{J} \subseteq \mathcal{I}_n$ and for all subcategories $\mathscr{C} \subseteq \rep(Q)$, we have $\mathscr{J} = \Int(\Cat_Q(\mathscr{J})) \text{ and } \mathscr{C} = \Cat_Q(\Int(\mathscr{C})).$

    Now we recall the description of all the short exact sequences between indecomposable representations.

    \begin{theorem} \label{thm:exttypeA} Let $Q$ be an $A_n$ type quiver. Let $D,F \in \ind(Q)$ and consider a non-split exact sequence \[\begin{tikzcd}
	0 & D & E & F &  0.
	\arrow[from=1-1, to=1-2]
	\arrow[tail, from=1-2, to=1-3]
	\arrow[two heads,from=1-3, to=1-4]
	\arrow[from=1-4, to=1-5]
\end{tikzcd} \]
      Then exactly one of the following assertions holds:
      \begin{enumerate}[label=$(\alph*)$,itemsep=1mm]
          \item We have $E \in \ind(Q)$, and if we consider $K,L \in \mathcal{I}_n$ such that $D \cong X_K$ and $F \cong X_L$, then $K \cup L \in \mathcal{I}_n$, $K \cap L = \varnothing$ and $E \cong X_{K \cup L}$; or,
          \item We have $E \cong E_1 \oplus E_2$ with $E_1,E_2 \in \ind(Q)$, and there exist $K,L \in \mathcal{I}_n$ with $K \cap L \neq \varnothing$ such that both assertions hold:
          \begin{enumerate}[label=$\bullet$,itemsep=1mm]
          \item $K,L,K \cap L$ and $K \cup L$ are four different intervals in $\mathcal{I}_n$; and,
          \item either $\{D,F\} = \{X_K, X_L\}$ and $\{E_1,E_2\} = \{X_{K\cap L}, X_{K \cup L} \}$, or \\
          $\{D,F\} = \{X_{K \cap L}, X_{K \cup L}\}$ and $\{E_1,E_2\} = \{X_{K}, X_{L} \}$.
          \end{enumerate}
          In this case, we call it a \new{diamond exact sequence.} 
      \end{enumerate}
    \end{theorem}

\begin{definition}[Diamond exact sequence]
Let $Q$ be an $A_n$ type quiver. A non-split short exact sequence
\[
0 \longrightarrow D \longrightarrow E \longrightarrow F \longrightarrow 0
\]
in $\rep(Q)$ is called a \emph{diamond exact sequence} if
\begin{enumerate}[label=(\roman*), itemsep=1mm]
    \item $D,F \in \ind(Q)$;
    \item $E$ is decomposable and,
    \[
    E \cong E_1 \oplus E_2 \quad \text{with } E_1,E_2 \in \ind(Q).
    \]
\end{enumerate}
\end{definition}

\section{Gen-Sub operators on exact categories}
\label{sec:ExStr}
%\input{Exact.tex}
%modification en cours
\pagestyle{plain}

Let $\mathscr{A}$ be a hereditary abelian category with enough projectives. 

\begin{conv} \label{conv:addsubcat}
    Our subcategories are \emph{full}, \emph{additive} and closed under isomorphisms. By additive subcategories, we mean subcategories that are closed under sums and summands.
\end{conv}

In this section, after a few meaningful recalls on exact structures, we define the \emph{Gen-Sub operator}. By fixing $\mathcal{E}$ as an exact structure on $\mathscr{A}$, it allows one to construct a subcategory $\mathscr{D} \supseteq \mathscr{C}$ satisfying crucial properties relative to $\mathcal{E}$ from any subcategory $\mathscr{C}$. We illustrate the following definitions with examples within the framework of $\rep (Q)$, where $Q$ is a type $A$ quiver. 

\subsection{Exact structures}
\label{ss:genonexactstructure}

Let $\mathscr{A}$ be an abelian category. 
Recall that two short exact sequences in $\mathscr{A}$ \[\begin{tikzcd}
	\xi : & 0 & D & E & F &  0 \\
    \varsigma: & 0 & X & Y & Z &  0
	\arrow[from=1-2, to=1-3]
	\arrow["f",tail, from=1-3, to=1-4]
	\arrow["g",two heads,from=1-4, to=1-5]
	\arrow[from=1-5, to=1-6]
	\arrow[from=2-2, to=2-3]
	\arrow["i",tail, from=2-3, to=2-4]
	\arrow["j",two heads,from=2-4, to=2-5]
	\arrow[from=2-5, to=2-6]
\end{tikzcd} \] are said to be {\em isomorphic} whenever there exists a triplet of isomorphisms \[\left(\begin{tikzcd}
	D & X
	\arrow[Plum,"\phi",from=1-1, to=1-2]
\end{tikzcd}, \begin{tikzcd}
	E & Y
	\arrow[Plum,"\psi",from=1-1, to=1-2]
\end{tikzcd}, \begin{tikzcd}
	F & Z
	\arrow[Plum,"\mu",from=1-1, to=1-2]
\end{tikzcd} \right)\] such that every square in the following diagram commutes.
\[\begin{tikzcd}
	\xi : & 0 & D & E & F &  0 \\
    \varsigma: & 0 & X & Y & Z &  0
	\arrow[from=1-2, to=1-3]
	\arrow["f",tail, from=1-3, to=1-4]
        \arrow[Plum,"\phi",from=1-3, to=2-3]
	\arrow["g",two heads,from=1-4, to=1-5]
        \arrow[Plum,"\psi",from=1-4, to=2-4]
	\arrow[from=1-5, to=1-6]
        \arrow[Plum,"\mu",from=1-5, to=2-5]
	\arrow[from=2-2, to=2-3]
	\arrow["i",tail, from=2-3, to=2-4]
	\arrow["j",two heads,from=2-4, to=2-5]
	\arrow[from=2-5, to=2-6]
\end{tikzcd} \] 
From now on, by abusing notations, we identify a short exact sequence with its isomorphism class of short exact sequences. Fix a subcollection $\mathcal{E}$ of short exact sequences in $\mathscr{A}$. A short exact sequence \[\begin{tikzcd}
	\xi: & 0 & A & E & B &  0
	\arrow[from=1-2, to=1-3]
	\arrow["f",tail, from=1-3, to=1-4]
	\arrow["g",two heads,from=1-4, to=1-5]
	\arrow[from=1-5, to=1-6]
\end{tikzcd} \] is said to be \new{$\mathcal{E}$-exact} (or \new{$\mathcal{E}$-admissible}) if $\xi \in \mathcal{E}$. In such a case, we say that the map $f$ is said to be an \new{$\mathcal{E}$-monomorphism}, and the map $g$ is said to be an \new{$\mathcal{E}$-epimorphism}. Considering the short exact sequences up to isomorphism, the definitions of $\mathcal{E}$-exact sequences, $\mathcal{E}$-monomorphisms, and short $\mathcal{E}$-epimorphisms are interdependent, uniquely determining each other within this framework.

Quillen's original formulation of exact structures thus encapsulates these relationships, providing a robust foundation for applying homological constructs in settings that deviate from traditional abelian categories. Quillen’s
definition can be rephrased as follows.

\begin{definition}\label{def:exact}
A collection $\mathcal{E}$ of short exact sequence in $\mathscr{A}$ is an \emph{exact structure} on $\mathscr{A}$ whenever $\mathcal{E}$ satisfies the following axioms:
\begin{enumerate}[label=$(\mathsf{ES} \arabic*)$,itemsep=1mm]
    \item \label{ES1} All the split short exact sequences are in $\mathcal{E}$;
  
    \item \label{ES2} The collection of $\mathcal{E}$-monomorphisms and the one of $\mathcal{E}$-epimorphisms are both closed under compositions;
    
    \item\label{ES3} $\mathcal{E}$ is closed under the pushouts along any morphism: that is, any short exact sequence \[\begin{tikzcd}
	0 & X & Y & Z & 0 \\ 0 & D & PO & Z & 0 ,
	\arrow[from=1-1, to=1-2]
	\arrow["f",tail, from=1-2, to=1-3]
	\arrow[two heads,from=1-3, to=1-4]
	\arrow[from=1-4, to=1-5]
    \arrow[from=2-1, to=2-2]
	\arrow[tail, from=2-2, to=2-3]
	\arrow[two heads,from=2-3, to=2-4]
	\arrow[from=2-4, to=2-5]
    \arrow["h",from=1-2, to=2-2]
	\arrow[from=1-3, to=2-3]
    \arrow[equal,from=1-4, to=2-4]
 \end{tikzcd} \] where $f$ is an $\mathcal{E}$-monomorphism gives a short exact sequence is in $\mathcal{E}$ whenever taking the pushout along any morphism $h$;
    
\item \label{ES4} $\mathcal{E}$ is closed under the pullbacks along any morphism: that is, any short exact sequence \[\begin{tikzcd}
	0 & X & Y & Z & 0 \\ 0 & X & PB & F & 0 ,
	\arrow[from=1-1, to=1-2]
	\arrow[tail, from=1-2, to=1-3]
	\arrow["g",two heads,from=1-3, to=1-4]
	\arrow[from=1-4, to=1-5]
    \arrow[from=2-1, to=2-2]
	\arrow[tail, from=2-2, to=2-3]
	\arrow[two heads,from=2-3, to=2-4]
	\arrow[from=2-4, to=2-5]
    \arrow[equal,from=1-2, to=2-2]
	\arrow[from=2-3, to=1-3]
    \arrow["h",from=2-4, to=1-4]
 \end{tikzcd} \]where $g$ is an $\mathcal{E}$-epimorphism gives a short exact sequence in $\mathcal{E}$ whenever taking the pullback along any morphism $h$.
\end{enumerate}
If $\mathcal{E}$ is an exact structure on $\mathscr{A}$, then we refer to the pair $(\mathscr{A},\mathcal{E})$ as an \new{exact category}. We also refer to the exact structure containing all short exact sequences as $\mathcal{E}_{\max}$, and the one containing only the split short exact sequences as $\mathcal{E}_{\min}$. We denote by $\mathcal{E}(-,-)$ the additive subbifunctor of $\Ext^1_\mathscr{A}(-,-)$ that consists of extensions belonging to $\mathcal{E}$.

\end{definition}

\begin{definition}\label{def:Eprojectivesinjectives}
Let $(\mathscr{A}, \mathcal{E})$ be an exact category.
\begin{enumerate}[label=$\bullet$,itemsep=1mm]
    \item An object $P$ in $\mathscr{A}$ is \new{$\mathcal{E}$-projective} if for every $\mathcal{E}$-epimorphism $g: Y \to Z$, the map $g_* = \Hom_\mathscr{A}(P,g)$ is an epimorphism across all $Y, Z$ in $\mathscr{A}$. Equivalently, $P$ is $\mathcal{E}$-projective if the functor $\Hom_\mathscr{A}(P,-)$ preserves the exactness of $\mathcal{E}$-admissible sequences. We write $\proj_\mathcal{E}(\mathscr{A})$ for the collection of $\mathcal{E}$-projective objects in $\mathscr{A}$. We observe that $\proj_{\mathcal{E}_{\max}}(\mathscr{A}) = \proj(\mathscr{A})$ where $\mathcal{E}_{\max}$ is the exact structure containing all the short exact sequences. 
    \item An object $I$ in $\mathscr{A}$ is \new{$\mathcal{E}$-injective} if for every $\mathcal{E}$-monomorphism $f: X \to Y$, the map $f^* = \Hom_\mathscr{A}(f,I)$ is an epimorphism across all $X, Y$ in $\mathscr{A}$. Equivalently, $I$ is $\mathcal{E}$-injective if the functor $\Hom_\mathscr{A}(-,I)$ preserves the exactness of $\mathcal{E}$-admissible sequences. We write $\inj_{\mathcal{E}}(\mathscr{A})$ for the collection of $\mathcal{E}$-injective objects in $\mathscr{A}$. We observe that $\inj_{\mathcal{E}_{\max}}(\mathscr{A}) = \inj(\mathscr{A})$.
\end{enumerate}
\end{definition}

In the following, we focus on a relevant example of exact structure for $\mathscr{A} = \rep(Q)$ for $Q$ a type $A$ quiver, called the \emph{diamond exact structure}.

\begin{definition} \label{def:Ediamond}
Let $Q$ be an $A_n$ type quiver. We define the collection of short exact sequences $\mathcal{E}_{\diamond}$ given by the additive bifunctor $\mathcal{E_\diamond}(-,-)$ that is uniquely determined by
\[
\mathcal{E}_\diamond(N,M) = \{ \eta \in \Ext^1(N,M) \; | \; \eta \mbox{ is split or a diamond exact sequence}\}.\] with $M,N \in \rep Q$ indecomposable.
\end{definition}

\bigskip

The following result ensures that $\mathcal{E}_\diamond$ is indeed an exact structure of $\rep(Q)$.

\begin{prop}[\cite{BHRS24}]
 \label{prop:EdiamondExactStructureTypeA}
     Let $n \in \mathbb{N}^*$, and $Q$ be an $A_n$ type quiver. The collection $\mathcal{E}_{\diamond}$ is an exact structure on $\mathscr{A} = \rep(Q)$.
 \end{prop}

Before defining the \emph{Gen-Sub operators}, we recall and restate some useful results on diagrams where the exact structure $\mathcal{E}$ intervenes, based on the work of T. Bühler \cite{B10}.

\begin{lemma}[$3 \times 3$-lemma] \label{lem:3x3}
Let $(\mathscr{A}, \mathcal{E})$ be an exact category. Consider the following commutative diagram.
\[\begin{tikzcd}
         & 0 & 0 & 0 &  \\
	& A & B  & C &   \\
      & A' & B' & C' &  \\
       & A'' & B'' & C'' &  \\
       & 0 & 0 & 0 & 
	%\arrow[from=2-1, to=2-2]
	\arrow[from=2-2, to=2-3]
	\arrow[from=2-3, to=2-4]
	%\arrow[from=2-4, to=2-5]
        %\arrow[from=3-1, to=3-2]
	\arrow["f",from=3-2, to=3-3]
	\arrow["g",from=3-3, to=3-4]
	%\arrow[from=3-4, to=3-5]
        %\arrow[from=4-1, to=4-2]
	\arrow[from=4-2, to=4-3]
	\arrow[from=4-3, to=4-4]
	%\arrow[from=4-4, to=4-5]
        \arrow[from=4-2, to=5-2]
	\arrow[two heads, from=3-2, to=4-2]
	\arrow[tail,from=2-2, to=3-2]
	\arrow[from=1-2, to=2-2]
        \arrow[from=4-3, to=5-3]
	\arrow[two heads, from=3-3, to=4-3]
	\arrow[tail,from=2-3, to=3-3]
	\arrow[from=1-3, to=2-3]
        \arrow[from=4-4, to=5-4]
	\arrow[two heads, from=3-4, to=4-4]
	\arrow[two heads,from=2-4, to=3-4]
	\arrow[from=1-4, to=2-4]
\end{tikzcd} \]
Assume that:
\begin{enumerate}[label=$\bullet$, itemsep=1mm]
    \item all the short sequences given by the columns are $\mathcal{E}$-exact; and,
    \item one of the following assertions holds:
    \begin{enumerate}[label=$\bullet$,itemsep=1mm]
        \item two out of the three rows, including the middle one, give short $\mathcal{E}$-exact sequences; or,
        \item the short sequences given by the top and the bottom rows are $\mathcal{E}$-exact, and $gf = 0$.
    \end{enumerate}
\end{enumerate}
Then the short sequence given by the remaining row is $\mathcal{E}$-exact.
\end{lemma}

\begin{remark} \label{rem:dual3x3}
    We can exchange the roles of the rows and the columns to obtain a dual result applied to the columns of such a diagram.
\end{remark}

\begin{lemma}[Obscure axiom] \label{lem:obscure}
Let $(\mathscr{A}, \mathcal{E})$ be an exact category. Let $X,Y \in \mathscr{A}$ and $i : X \longrightarrow Y$ admitting a cokernel. If there exist $Z \in \mathscr{A}$, and $j : Y \longrightarrow Z$ such that $j \circ i$ is an $\mathcal{E}$-monomorphism, then $i$ is an $\mathcal{E}$-monomorphism. Dually, it holds for $\mathcal{E}$-epimorphisms.
\end{lemma}

\begin{remark} \label{rem:obscure}
    This lemma was initially stated as an axiom from D. Quillen's definition of an exact category \cite{Q73}. N. Yoneda proved in his thesis \cite{Y61} that this axiom is a consequence of the other axioms. Later on, B. Keller \cite{K90} rediscovered this redundancy. The name ``obscure axiom" was given by R. W. Thomason \cite{TT07}.
\end{remark}

\subsection{Gen-Sub operators}
\label{ss:GenSuboperators}

Fix $(\mathscr{A},\mathcal{E})$ an exact category. This section introduces a closure operator on subcategories of $\mathscr{A}$. More precisely, given a subcategory $\mathscr{C}$, we will construct the smallest full subcategory $\mathscr{D} \supseteq \mathscr{C}$ such that:
\begin{enumerate}[label=$\bullet$,itemsep=1mm]
    \item $\mathscr{D}$ is closed under cokernels of $\mathcal{E}$-monomorphisms; and,
    \item  $\mathscr{D}$ is closed under kernels of $\mathcal{E}$-epimorphisms.
\end{enumerate}

The \new{$\Gen_\mathcal{E}$-operator} on subcategories of $\mathscr{A}$ is defined as follows: given a subcategory $\mathscr{C} \subseteq \mathscr{A}$, we define $\Gen_\mathcal{E}(\mathscr{C})$ as the full subcategory, closed under sums and summands, of $\mathscr{A}$ containing the cokernel of any $\mathcal{E}$-monomorphism $f : X \longrightarrow Y$ such that $Y \in \mathscr{C}$. The \new{$\Sub_\mathcal{E}$-operator} on subcategories of $\mathscr{A}$ is defined analogously: given a subcategory $\mathscr{C} \subseteq \mathscr{A}$, we define $\Sub_\mathcal{E}(\mathscr{C})$ as the full subcategory, closed by sums and summands, of $\mathscr{A}$ containing the kernel of any $\mathcal{E}$-epimorphism $g : Y \longrightarrow X$  such that $Y \in \mathscr{C}$. Given an object $M \in \mathscr{C}$, by abusing of notation, we set $\Gen_\mathcal{E}(M) = \Gen_\mathcal{E}(\add(M))$ and $\Sub_\mathcal{E}(M) = \Sub_\mathcal{E}(\add(M))$.

Note that $\Gen_{\mathcal{E}_{\max}}(\mathscr{C}) = \Gen(\mathscr{C})$ and $\Sub_{\mathcal{E}_{\max}}(\mathscr{C}) = \Sub(\mathscr{C})$. As $\mathcal{E}$ is an exact structure on $\mathscr{A}$, then $\mathscr{C}$ is a subcategory of both $\Gen_\mathcal{E}(\mathscr{C})$ and $\Sub_\mathcal{E}(\mathscr{C})$.

Given a subcategory $\mathscr{C} \subseteq \mathscr{A}$, we define a sequence of subcategories $(\GS_{\mathcal{E}}^i(\mathscr{C}))_{i\in \mathbb{N}}$ as it follows:
\begin{enumerate}[label=$\bullet$, itemsep=1mm]
    \item we set $\GS_{\mathcal{E}}^0(\mathscr{C}) = \mathscr{C}$; and,
    \item for all $i \geqslant 1$, we define $\GS^{i}_{\mathcal{E}}(\mathscr{C})$ to be the subcategory of $\mathscr{A}$ additively generated by objects in both $\Gen_{\mathcal{E}}(\GS^{i-1}_{\mathcal{E}}(\mathscr{C}))$ and $\Sub_{\mathcal{E}}(\GS^{i-1}_{\mathcal{E}}(\mathscr{C}))$.
\end{enumerate}
By the previous construction, $(\GS_{\mathcal{E}}^i(\mathscr{C}))_{i\in \mathbb{N}}$ is an increasing sequence of additive subcategories of $\mathscr{A}$. Therefore, as $\mathscr{A}$ is abelian, this sequence of subcategories admits a colimit.

\begin{definition} \label{def:GSop}
    The \new{$\GS_\mathcal{E}$-operator} on subcategories of $\mathscr{A}$ is defined as follows: given a subcategory $\mathscr{C} \subseteq \mathscr{A}$, we define $\GS_{\mathcal{E}}(\mathscr{C})$ as the colimit of the sequence of subcategories $(\GS_\mathcal{E}^i(\mathscr{C}))_{i\in \mathbb{N}}$. As before, we set $\GS_\mathcal{E}(M) = \GS_\mathcal{E}(\add(M))$ by abuse of notations.
\end{definition}

Here are some apparent results coming directly from the definition.

\begin{lemma} \label{lem:proponGS}
Let $\mathscr{C} \subseteq \mathscr{A}$. The following assertions hold:
\begin{enumerate}[label=$(\alph*)$,itemsep=1mm]
    \item \label{GSa} The category $\GS_\mathcal{E}(\mathscr{C})$ is the (full additive) subcategory of $\mathscr{A}$ whose objects are the ones appearing in $\GS^i(\mathscr{C})$ for some $i \in \mathbb{N}$.
    \item \label{GSb} The category $\GS_\mathcal{E}(\mathscr{C})$ is the smallest subcategory $\mathscr{A}$ containing $\mathscr{C}$, closed under taking $\mathcal{E}$-subobjects and  $\mathcal{E}$-quotients.
\end{enumerate}
\end{lemma}

\begin{remark}\label{rem:weakerSerre} A subcategory $\mathscr{D} \subseteq \mathscr{A}$ satisfying \ref{GSb} must be closed under extensions to be a so-called \emph{$\mathcal{E}$-Serre} category. We recall that $\mathscr{D}$ is \new{$\mathcal{E}$-Serre} whenever for any short $\mathcal{E}$-exact sequence
\[\begin{tikzcd}
	\xi: & 0 & E & F & G &  0,
	\arrow[from=1-2, to=1-3]
	\arrow[tail, from=1-3, to=1-4]
	\arrow[two heads,from=1-4, to=1-5]
	\arrow[from=1-5, to=1-6]
\end{tikzcd}\] we have $F \in \mathscr{D}$ if and only if $E,G \in \mathscr{D}$.
\end{remark}

\begin{example} \label{ex:A7typeGS} In \cref{fig:GSCalc1}, we calculate $\GS_{\mathcal{E}_\diamond}(\mathscr{C})$ for a subcategory $\mathscr{C}$ of $\rep(Q)$ and a fixed quiver $Q$ of $A_7$ type. 
\begin{figure}[!ht]
    \leavevmode\\[-0.5ex]  % forces heading to appear before content
    \centering
    \begin{tikzpicture}
				\begin{scope}[line width=.5mm,->, >= angle 60]
                    \node at (-.6,0){$Q=$};
					\node (a) at (0,0){$1$};
					\node (b) at (1,0){$2$};
					\node (c) at (2,0){$3$};
					\node (d) at (3,0){$4$};
                    \node (e) at (4,0){$5$};
                    \node (f) at (5,0){$6$};
                    \node (g) at (6,0){$7$};
					\draw (a) -- (b);
					\draw (c) -- (b);
					\draw (c) -- (d);
                    \draw (d) -- (e);
                    \draw (f) -- (e);
                    \draw (f) -- (g);
				\end{scope}
                \begin{scope}[yshift=-2cm, xshift=-.5cm, line width=.3mm, ->, >= angle 60]
					\node (2) at (0,0){\scalebox{.7}{$\llrr{2}$}};
					\node (12) at (1,1){\scalebox{.7}{$\llrr{1,2}$}};
					\node[circle,line width=0.2mm,fill=darkgreen!20,draw] (2345) at (1,-1){\scalebox{.7}{$\llrr{2,5}$}};
					\node (45) at (0,-2){\scalebox{.7}{$\llrr{4,5}$}};
                    \node[circle,line width=0.2mm,double,fill=lava!20,draw] (5) at (-1,-3){\scalebox{.7}{$\llrr{5}$}};
                    \node[circle,line width=0.2mm,double,fill=lava!20,draw] (567) at (0,-4){\scalebox{.7}{$\llrr{5,7}$}};
                    \node (7) at (-1,-5){\scalebox{.7}{
							$\llrr{7}$}};
                    \node (56)[circle,line width=0.2mm,double,fill=lava!20,draw] at (1,-5){\scalebox{.7}{$\llrr{5,6}$}
                                };
                     \node (4) at (3,-5){\scalebox{.7}{$\llrr{4}$}};
                     \node[circle,line width=0.2mm,double,fill=lava!20,draw] (23) at (5,-5){\scalebox{.7}{$\llrr{2,3}$}};
                    \node (1) at (7,-5){\scalebox{.7}{$\llrr{1}$}};
                    \node (456) at (2,-4){\scalebox{.7}{$\llrr{4,6}$}};
                    \node (234) at (4,-4){\scalebox{.7}{$\llrr{2,4}$}};
                     \node[circle,line width=0.2mm,double,fill=lava!20,draw] (123) at (6,-4){\scalebox{.7}{$\llrr{1,3}$}};
                     \node (4567) at (1,-3){\scalebox{.7}{$\llrr{4;7}$}};
                     \node[circle,line width=0.2mm,dashed,fill=bleudefrance!20,draw] (23456) at (3,-3){\scalebox{.7}{$\llrr{2,6}$}};
                     \node (1234) at (5,-3){\scalebox{.7}{$\llrr{1,4}$}};
                     \node[circle,line width=0.2mm,fill=darkgreen!20,draw] (3) at (7,-3){\scalebox{.7}{$\llrr{3}$}};
					\node[circle,line width=0.2mm, double,fill=lava!20,draw] (12345) at (2,0){\scalebox{.7}{$\llrr{1,5}$}};
					\node[circle,line width=0.2mm, dashed,fill=bleudefrance!20,draw] (234567) at (2,-2){\scalebox{.7}{$\llrr{2,7}$}};
                    \node[circle,line width=0.2mm,fill=darkgreen!20,draw] (123456) at (4,-2){\scalebox{.7}{$\llrr{1,6}$}};
					\node (34) at (6,-2){\scalebox{.7}{$\llrr{3,4}$}};
					\node[circle,line width=0.2mm,fill=darkgreen!20,draw] (1234567) at (3,-1){\scalebox{.7}{$\llrr{1,7}$}};
                    \node[circle,line width=0.2mm,fill=darkgreen!20,draw] (3456) at (5,-1){\scalebox{.7}{$\llrr{3,6}$}};
                    \node[circle,line width=0.2mm,fill=darkgreen!20,draw] (34567) at (4,0){\scalebox{.7}{$\llrr{3,7}$}};
                    \node (6) at (6,0){\scalebox{.7}{$\llrr{6}$}};
                    \node[circle,line width=0.2mm,double,fill=lava!20,draw] (345) at (3,1){\scalebox{.7}{$\llrr{3,5}$}};
                    \node (67) at (5,1){\scalebox{.7}{$\llrr{6,7}$}};
					\draw (2) -- (12);
					\draw (2) -- (2345);
					\draw (45) -- (2345);
					\draw (12) -- (12345);
					\draw (2345) -- (12345);
					\draw (2345) -- (234567);
                    \draw (12345) -- (345);

                    \draw (5) -- (45);
                    \draw (7) -- (567);
                    \draw (567) -- (4567);
                    \draw (4567) -- (234567);
                    \draw (234567) -- (1234567);
                    \draw (1234567) -- (34567);
                    \draw (34567) -- (67);
                    \draw (56) -- (456);
                    \draw (456) -- (23456);
                    \draw (23456) -- (123456);
                    \draw (123456) -- (3456);
                    \draw (3456) -- (6);
                    \draw (4) -- (234);
                    \draw (234) -- (1234);
                    \draw (1234) -- (34);
                    \draw (23) -- (123);
                    \draw (123) -- (3);
					
					\draw (5) -- (567);
                    \draw (567) -- (56);
                    \draw (45) -- (4567);
                    \draw (4567) -- (456);
                    \draw (456) -- (4);
                    \draw (234567) -- (23456);
                    \draw (23456) -- (234);
                    \draw (234) -- (23);
                    \draw (12345) -- (1234567);
                    \draw (1234567) -- (123456);
                    \draw (123456) -- (1234);
                    \draw (1234) -- (123);
                    \draw (123) -- (1);
                    \draw (345) -- (34567);
                    \draw (34567) -- (3456);
                    \draw (3456) -- (34);
                    \draw (34) -- (3);
                    \draw (67) -- (6);	
				\end{scope}
    \end{tikzpicture}
    \caption[fragile]{Calculation of $\GS_{\mathcal{E}_\diamond}(\mathscr{C})$ where $\mathscr{C} = \add \left( \protect\begin{tikzpicture}[baseline={(0,-.1)}]
        \node[circle, line width=0.2mm, double, fill=lava!20,minimum size=1em,draw] at (0,0){$ $};
    \end{tikzpicture}\right)$. We get $\GS_{\mathcal{E}_\diamond}^1(\mathscr{C}) = \add \left(\protect\begin{tikzpicture}[baseline={(0,-.1)}]
        \node[circle, line width=0.2mm, double, fill=lava!20,minimum size=1em,draw] at (0,0){$ $};
    \end{tikzpicture} + \protect\begin{tikzpicture}[baseline={(0,-.1)}]
        \node[circle, line width=0.2mm, fill=darkgreen!20,minimum size=1em,draw] at (0,0){$ $};
    \end{tikzpicture}\right)$ and $\GS_{\mathcal{E}_\diamond}^2(\mathscr{C}) = \add \left(\protect\begin{tikzpicture}[baseline={(0,-.1)}]
        \node[circle, line width=0.2mm, double, fill=lava!20,minimum size=1em,draw] at (0,0){};
    \end{tikzpicture} + \protect\begin{tikzpicture}[baseline={(0,-.1)}]
        \node[circle, line width=0.2mm, fill=darkgreen!20,minimum size=1em,draw] at (0,0){$ $};
    \end{tikzpicture} + \protect\begin{tikzpicture}[baseline={(0,-.1)}]
        \node[circle, line width=0.2mm,dashed, fill=bleudefrance!20,minimum size=1em,draw] at (0,0){$ $};
    \end{tikzpicture}\right)$. Here, we obtain that $\GS_{\mathcal{E}_\diamond}(\mathscr{C}) = \GS_{\mathcal{E}_\diamond}^2(\mathscr{C})$.}
    \label{fig:GSCalc1}
\end{figure}
\end{example}

\subsection{Subcategories adapted to an exact structure}
\label{ss:Adaptpreserv} Let $(\mathscr{A},\mathcal{E})$ be an exact category. We recall that $\mathscr{A}$ is \emph{hereditary} if, for all $X,Y \in \mathscr{A}$, and for all $i \in \mathbb{N}_{\geqslant 2}$, we have $\Ext_\mathscr{A}^i(X,Y) = 0$. 

\begin{conv} \label{conv:hereditary}
    From now on, $\mathscr{A}$ is a hereditary abelian category.
\end{conv}

We define the following useful notion on subcategories of $\mathscr{A}$.

\begin{definition} \label{def:Eadpted}
    A subcategory $\mathscr{D} \subseteq \mathscr{A}$ is \new{$\mathcal{E}$-adapted} if for any pair $(X,Y)$ of objects in $\mathscr{D}$, we have $\Ext^1(X,Y) \subseteq \mathcal{E}$
\end{definition}
\begin{example} $ $
\begin{enumerate}[label=$\bullet$, itemsep=1mm]
    \item Any subcategory of $\mathscr{A}$ is $\mathcal{E}_{\max}$-adapted.
    \item A subcategory $\mathscr{C} \subseteq \mathscr{A}$ is $\mathcal{E}_{\min}$-adapted only if for any $(X,Y) \in \mathscr{C}^2$, we have $\Ext_\mathscr{A}^1 (X,Y) = 0$.
\end{enumerate}
\end{example}

In this section, we will show that for any $\mathcal{E}$-adapted subcategory $\mathscr{C} \subseteq \mathscr{A}$, the category $\GS_\mathcal{E}(\mathscr{C})$ is also $\mathcal{E}$-adapted.

\begin{lemma}
\label{lem:GenlocalEadapt}
Consider a subcategory $\mathscr{C} \subseteq \mathscr{A}$. Let $D,E \in \mathscr{C}$ such that $\Ext^1(D,E) \subseteq \mathcal{E}$. Then, for any $X \in \Gen(E)$, the following assertions hold:
\begin{enumerate}[label=$(\alph*)$,itemsep=1mm]
   \item Any $\xi \in \Ext^1(D,X)$ is obtained by a pushout of some $\eta \in \Ext^1(D,E^{\oplus r})$ along an epimorphism $ \begin{tikzcd}
	 g : E^{\oplus r} & X;
	\arrow[two heads,from=1-1, to=1-2]
   \end{tikzcd}$ and,
   \item $\Ext^1(D,X) \subseteq \mathcal{E}$.
\end{enumerate}
Furthermore, if $X \in \Gen_\mathcal{E}(E)$, then $g$ can be chosen as an $\mathcal{E}$-epimorphism.   
\end{lemma}

\begin{proof}
    Assume that $X \in \Gen(E)$.
Then there is a short exact sequence \[\begin{tikzcd}
	0 & K & E^{\oplus r}  & X&  0
	\arrow[from=1-1, to=1-2]
	\arrow[ tail,from=1-2, to=1-3]
	\arrow["g",two heads,from=1-3, to=1-4]
	\arrow[from=1-4, to=1-5]
\end{tikzcd} \] By applying the $\Hom(D,-)$ functor on this short exact sequence, we obtain the following long exact sequence
\[\begin{tikzcd}
	0 & \Hom(D,K) & \Hom(D,E^{\oplus r})  & \Hom(D,X) & \\
     & \Ext^1(D,K) & \Ext^1(D,E^{\oplus r}) & \Ext^1(D,X) & 0,
	\arrow[from=1-1, to=1-2]
	\arrow[tail,from=1-2, to=1-3]
	\arrow[from=1-3, to=1-4]
	\arrow[from=1-4, to=2-2]
    \arrow[from=2-2, to=2-3]
    \arrow[two heads,"g_{\bullet}",from=2-3, to=2-4]
    \arrow[from=2-4, to=2-5]
\end{tikzcd} \] where $g_\bullet = \Ext^1(D,g)$. As $\mathscr{A}$ is hereditary, we have that $\Ext^2(D,K)=0$. So $g_\bullet$ is surjective, and therefore, for any short exact sequence $\xi \in \Ext^1(D,X)$, there exists $\eta  \in \Ext^1(D,E^{\oplus r})$ such that $g_\bullet(\eta) = \xi$. Since $g_\bullet$ is the pushout along $g$, we proved the assertion $(a)$. 

By hypothesis, $\Ext^1(D,E) \subseteq \mathcal{E}$. So we have $\Ext^1(D,X) \subseteq \mathcal{E}$ as $\mathcal{E}$ is an exact structure. If moreover $X \in \Gen_\mathcal{E}(E)$, then $g$ can obviously be chosen as an $\mathcal{E}$-epimorphism. 
\end{proof}

\begin{lemma}
\label{lemma:gendiagram}
Let $\mathscr{C} \subseteq \mathscr{A}$ be a subcategory. Let $E,F \in \mathscr{C}$ with $\Ext^1(E,F) \subseteq \mathcal{E}$. Then, for any object $X$ in $\Gen_\mathcal{E}(E)$, we have $\Ext^1(X,F) \subseteq \mathcal{E}$.  
\end{lemma}

\begin{proof}
 Suppose we have the following short exact sequence
 \[\begin{tikzcd}
	\xi: & 0 & F & E' & X&  0.
	\arrow[from=1-2, to=1-3]
	\arrow[tail, from=1-3, to=1-4]
	\arrow[two heads,from=1-4, to=1-5]
	\arrow[from=1-5, to=1-6]
\end{tikzcd}\] Since $X \in \Gen_\mathcal{E}(E)$,
there is a short exact sequence \[\begin{tikzcd}
	0 & K & E^{\oplus r}  & X&  0 & \in \mathcal{E}.
	\arrow[from=1-1, to=1-2]
	\arrow[tail, from=1-2, to=1-3]
	\arrow[two heads,"\varphi",from=1-3, to=1-4]
	\arrow[from=1-4, to=1-5]
\end{tikzcd} \] We can construct the following diagram.
\[\begin{tikzcd}
         & & & 0 &  \\
	0 & F & E'  & X &  0 \\
      & & & E^{\oplus r} & \\
      & & & K & \\
       & & & 0 & 
	\arrow[from=2-1, to=2-2]
	\arrow[tail, from=2-2, to=2-3]
	\arrow[two heads,from=2-3, to=2-4]
	\arrow[from=2-4, to=2-5]
        \arrow[from=5-4, to=4-4]
	\arrow[tail, from=4-4, to=3-4]
	\arrow[two heads,"\varphi",from=3-4, to=2-4]
	\arrow[from=2-4, to=1-4]
\end{tikzcd} \]
We complete the diagram by taking the pullback along $\varphi$ and by using the kernel lifting property. 
\[\begin{tikzcd}
         & 0 & 0 & 0 &  \\
	0 & F & E'  & X &  0 \\
     0 & F & PB & E^{\oplus r} & 0 \\
      0 & 0 & K' & K & 0 \\
       & 0 & 0 & 0 & 
	\arrow[from=2-1, to=2-2]
	\arrow[tail, from=2-2, to=2-3]
	\arrow[two heads,from=2-3, to=2-4]
	\arrow[from=2-4, to=2-5]
        \arrow[from=3-1, to=3-2]
	\arrow[tail, from=3-2, to=3-3]
	\arrow[two heads,from=3-3, to=3-4]
	\arrow[from=3-4, to=3-5]
        \arrow[from=4-1, to=4-2]
	\arrow[tail, from=4-2, to=4-3]
	\arrow[two heads,from=4-3, to=4-4]
	\arrow[from=4-4, to=4-5]
        \arrow[from=5-2, to=4-2]
	\arrow[tail, from=4-2, to=3-2]
	\arrow[equals,from=3-2, to=2-2]
	\arrow[from=2-2, to=1-2]
        \arrow[from=5-3, to=4-3]
	\arrow[tail, from=4-3, to=3-3]
	\arrow[two heads,from=3-3, to=2-3]
	\arrow[from=2-3, to=1-3]
        \arrow[from=5-4, to=4-4]
	\arrow[tail, from=4-4, to=3-4]
	\arrow[two heads,"\varphi",from=3-4, to=2-4]
	\arrow[from=2-4, to=1-4]
\end{tikzcd} \]

As the short sequences given by the three columns and those given by the first two rows of the above diagram are exact, by \cref{lem:3x3}, the short sequence given by the last row is exact too. So $K' \cong K$. 

Moreover, the short exact sequences given by the third column and the middle row are in $\mathcal{E}$ by hypothesis. As $\mathcal{E}$ is an exact structure, the short exact sequence given by the middle column is in $\mathcal{E}$. As the short exact sequence given by the first column and the one provided by the last row are split, they are in $\mathcal{E}$.

By \cref{lem:3x3}, we have that $\xi \in {\mathcal{E}}$. 
\end{proof}

\begin{lemma}
\label{lem:SublocalEadapt}
Consider a subcategory $\mathscr{C} \subseteq \mathscr{A}$, and $E,F \in \mathscr{C}$ such that $\Ext^1(E,F) \subseteq \mathcal{E}$. Then, for any $X \in \Sub(E)$, the following assertions hold:
\begin{enumerate}[label=$(\alph*)$,itemsep=1mm]
   \item Any $\xi \in \Ext^1(X,F)$ is obtained by a pullback of some $\eta \in \Ext^1(E^{\oplus r},F)$ along a monomorphism $ \begin{tikzcd}
	 f: X & E^{\oplus r};
	\arrow[tail,from=1-1, to=1-2]
   \end{tikzcd}$ and,
   \item $\Ext^1(X,F) \subseteq \mathcal{E}$.
\end{enumerate}
Furthermore, if $X \in \Sub_\mathcal{E}(E)$, then $f$ can be chosen as an $\mathcal{E}$-monomorphism.   
\end{lemma}

\begin{proof}
This is the dual result of Lemma \ref{lem:GenlocalEadapt}.
\begin{comment}

Let $X\in \Sub(E)$. Then, there is a short exact sequence \[\begin{tikzcd}
	0 & X & E  & C&  0.
	\arrow[from=1-1, to=1-2]
	\arrow["f",tail, from=1-2, to=1-3]
	\arrow[two heads, from=1-3, to=1-4]
	\arrow[from=1-4, to=1-5]
\end{tikzcd} \]
We obtain the following long exact sequence by applying the $\Hom(-,F)$ functor on this short exact sequence.
\[\begin{tikzcd}
	0 & \Hom(C,F) & \Hom(E,F)  & \Hom(X,F) & \\
     & \Ext^1(C,F) & \Ext^1(E,F) & \Ext^1(X,F) & 0,
	\arrow[from=1-1, to=1-2]
	\arrow[from=1-2, to=1-3]
	\arrow[from=1-3, to=1-4]
	\arrow[from=1-4, to=2-2]
    \arrow[from=2-2, to=2-3]
    \arrow["\scriptscriptstyle{f^{\bullet}}",from=2-3, to=2-4]
    \arrow[from=2-4, to=2-5]
\end{tikzcd} \] 
where $f^{\bullet} = \Ext^1(f,F)$. As $\mathscr{A}$ is hereditary, we have $\Ext^2(C,F)=0$. So $f^{\bullet}$ is surjective: for any short exact sequence $\xi \in \Ext^1(X,F)$, there exists $\eta  \in \Ext^1(E,F)$ such that $f^{\bullet}(\eta) = \xi$. Since $f^{\bullet}$ is the pullback along $f$, it means that every short exact sequence in $\Ext^1(X,F)$ is obtained by a pullback of a short exact sequence in $\Ext^1(E,F)$. 

By hypothesis, $\Ext^1(E,F) \subseteq \mathcal{E}$. So we have $\Ext^1(X,F) \subseteq \mathcal{E}$ as $\mathcal{E}$ is an exact structure. If moreover $X \in \Sub_\mathcal{E}(E)$, then $f$ can obviously be chosen as $\mathcal{E}$-monomorphism.
\end{comment}
\end{proof}

\begin{lemma}
\label{lemma:subdiagram}
Let $\mathscr{C} \subseteq \mathscr{A}$ be a subcategory. Let $D,E$ be in $\mathscr{C}$ such that $\Ext^1(D,E) \subseteq \mathcal{E}$. Then, for any object $X$ in $\Sub_\mathcal{E}(E)$, we have $\Ext^1(D,X) \subseteq \mathcal{E}$.
\end{lemma}

\begin{proof}
This is the dual result of Lemma \ref{lemma:gendiagram}.
\end{proof}

\begin{prop}
\label{prop:Estableinduction}
Consider an exact category $(\mathscr{A}, \mathcal{E})$ and let $\mathscr{C} \subseteq \mathscr{A}$ be a subcategory. If $\mathscr{C}$ is $ {\mathcal{E}}$-adapted, then $\Ext^1\left(\mathscr{C}, \GS^1_{\mathcal{E}}(\mathscr{C})\right)$ and $\Ext^1 \left(\GS^1_{\mathcal{E}}(\mathscr{C}), \mathscr{C} \right)$ are contained in ${\mathcal{E}}$.
\end{prop}

\begin{proof}
Consider $X$ an object of $\GS^1_{\mathcal{E}}(\mathscr{C})$. Then $X$ is in $\Gen_{\mathcal{E}}(\mathscr{C})$ or in $\Sub_{\mathcal{E}}(\mathscr{C})$.

Assume that $X \in \Gen_{\mathcal{E}}(\mathscr{C})$. We can then suppose that $X \in \Gen_{\mathcal{E}}(E) \subseteq \Gen(E)$ for $E \in \mathscr{C}$. By hypothesis, we know for any $D,F\in \mathscr{C}$, $\Ext^1(D,E),\Ext^1(E,F) \subseteq \mathcal{E}$. By \cref{lem:GenlocalEadapt}, $\Ext^1(D,X) \subseteq \mathcal{E}$ for any $D \in \GS^{1}_{\mathcal{E}}(\mathscr{C})$. By \cref{lemma:gendiagram}, $\Ext^1(X,F) \subseteq \mathcal{E}$ for any $F \in \mathscr{C}$.

Assume that $X \in \Sub_{\mathcal{E}}(\mathscr{C})$. Similarly, we have that $X \in \Sub_{\mathcal{E}}(E) \subseteq \Sub(E)$ for some $E \in \mathscr{C}$. By hypothesis, we know for any $D,F \in \mathscr{C}$, $\Ext^1(D,E),\Ext^1(E,F) \subseteq \mathcal{E}$. By \cref{lem:SublocalEadapt}, $\Ext^1(X,F) \subseteq \mathcal{E}$ for any $F \in \mathscr{C}$. By \cref{lemma:subdiagram}, $\Ext^1(D,X) \subseteq \mathcal{E}$ for any $D \in \mathscr{C}$.

Therefore, we get to the conclusion that, for any $X\in \GS^1_{\mathcal{E}}(\mathscr{C})$ and $E' \in \mathscr{C}$, we have $\Ext^1(X,E'),\Ext^1(E',X)\subseteq \mathcal{E}$.
\end{proof}

We can now finally prove and state the fact that the $\GS_\mathcal{E}$-operator preserves $\mathcal{E}$-adaptability.

\begin{prop}
\label{prop:EAdapt}
Let $\mathscr{C} \subseteq \mathscr{A}$ be a subcategory. If $\mathscr{C}$ is $\mathcal{E}$-adapted, then $\GS_{\mathcal{E}}(\mathscr{C})$ is $\mathcal{E}$-adapted too.
\end{prop}

\begin{proof}
We prove by induction that $\GS^{i}_{\mathcal{E}}(\mathscr{C})$ is ${\mathcal{E}}$-adpated, for all $i \geq 0$. For $i=0$, the statement holds as $\mathcal{C}$ is $\mathcal{E}$-adapted.

Fix $k \geqslant 0$. Assume that  $\GS^{k}_{\mathcal{E}}(\mathscr{C})$ is ${\mathcal{E}}$-adapted. Suppose that $X,Y \in \GS^{k+1}_{\mathcal{E}}(\mathscr{C})$. Then $X \in \Gen_{\mathcal{E}}(\GS^{k}_{\mathcal{E}}(\mathscr{C}))$ or $X \in \Sub_{\mathcal{E}}(\GS^{k}_{\mathcal{E}}(\mathscr{C}))$.

If $X \in \Gen_{\mathcal{E}}(\GS^{k}_{\mathcal{E}}(\mathscr{C}))$. So $X \in \Gen_{\mathcal{E}}(E) \subseteq \Gen(E)$ for some $E \in \GS^{k}_{\mathcal{E}}(\mathscr{C})$. By induction hypothesis and \cref{prop:Estableinduction}, we get that $\Ext^1(Y,E),\Ext^1(E,Y) \subseteq \mathcal{E}$. Therefore, by \cref{lem:GenlocalEadapt}, $\Ext^1(Y,X) \subseteq \mathcal{E}$. Using \cref{lemma:gendiagram}, we also have $\Ext^1(X,Y) \subseteq \mathcal{E}$.

If $X \in \Sub_{\mathcal{E}}(\GS^{k}_{\mathcal{E}}(\mathscr{C}))$. Let $E \in \GS^{k}_{\mathcal{E}}(\mathscr{C})$ such that $X \in \Sub_{\mathcal{E}}(E) \subseteq \Sub(E)$. By induction hypothesis and \cref{prop:Estableinduction}, we get that both $\Ext^1(Y,E) \subseteq \mathcal{E}$ and $\Ext^1(E,Y) \subseteq \mathcal{E}$. Therefore, by \cref{lemma:subdiagram}, we also have that $\Ext^1(Y,X) \subseteq \mathcal{E}$.

In either case, we prove that $\Ext^1(Y,X) \subseteq \mathcal{E}$. So $\GS^{k+1}_{\mathcal{E}}(\mathscr{C})$ is $\mathcal{E}$-adapted. We get the desired result on $\GS^{i}_{\mathcal{E}}(\mathscr{C})$ for all $i \geq 0$. So, $\GS_{\mathcal{E}}(\mathscr{C})$ is  $\mathcal{E}$-adapted.
\end{proof}

\subsection{Preserving extension closure}
\label{ss:Extclosedpreserv}

In the following, we will show that $\GS_\mathcal{E}(\mathscr{C})$ is extension-closed whenever $\mathscr{C} \subseteq \mathscr{A}$ is an $\mathcal{E}$-adapted subcategory closed under extensions. To this end, we prove the following lemma.

\begin{lemma}
\label{lemmaExtclosedinduction}
Let $\mathscr{C} \subset \mathscr{A}$ be an $\mathcal{E}$-adapted subcategory closed under extensions. If there is a short exact sequence 
\[\begin{tikzcd}
	0 & X & E  & U &  0,
	\arrow[from=1-1, to=1-2]
	\arrow[tail, from=1-2, to=1-3]
	\arrow[two heads,from=1-3, to=1-4]
	\arrow[from=1-4, to=1-5]
\end{tikzcd} \] with $X \in \GS_{\mathcal{E}}(\mathscr{C})$ and $U \in \mathscr{C}$, then $E \in \GS_{\mathcal{E}}(\mathscr{C})$.
\end{lemma}

\begin{proof}
It is enough to show that, for any $m \in \mathbb{N}$, if we have a short exact sequence \[\begin{tikzcd}
	    0 & X & E  & U &  0,
	\arrow[from=1-1, to=1-2]
	\arrow[tail, from=1-2, to=1-3]
	\arrow[two heads,from=1-3, to=1-4]
	\arrow[from=1-4, to=1-5]
\end{tikzcd} \] with $X\in \GS^m_{\mathcal{E}}(\mathscr{C})$ and $U \in \mathscr{C}$, then $E \in  \GS_{\mathcal{E}}(\mathscr{C})$. We proceed by induction on $m \in \mathbb{N}$. 

If $m=0$, then $X \in \GS^0_{\mathcal{E}}(T) = \mathscr{C}$. As $\mathscr{C}$ is closed under extensions, the result follows. Fix $m \in \mathbb{N}$. Assume that if there is a short exact sequence \[\begin{tikzcd}
	0 & X & E  & U &  0,
	\arrow[from=1-1, to=1-2]
	\arrow[tail, from=1-2, to=1-3]
	\arrow[two heads,from=1-3, to=1-4]
	\arrow[from=1-4, to=1-5]
\end{tikzcd} \] with $X \in \GS^{m}_{\mathcal{E}}(\mathscr{C})$ and $U \in \mathscr{C}$, then $E \in \GS_{\mathcal{E}}(\mathscr{C})$.

Consider a short exact sequence \[\begin{tikzcd}
	0 & Y & E'  & U' &  0,
	\arrow[from=1-1, to=1-2]
	\arrow[tail, from=1-2, to=1-3]
	\arrow[two heads,from=1-3, to=1-4]
	\arrow[from=1-4, to=1-5]
\end{tikzcd} \] with $Y \in \GS^{m+1}_{\mathcal{E}}(\mathscr{C})$ and $U' \in \mathscr{C}$. Then we have that $Y \in \Sub_\mathcal{E}(\GS^{m}_{\mathcal{E}}(\mathscr{C}))$ or $Y \in \Gen_\mathcal{E}(\GS^{m}_{\mathcal{E}}(\mathscr{C})).$

Suppose $Y \in \Gen_\mathcal{E}(\GS^{m}_{\mathcal{E}}(\mathscr{C}))$. Let $F \in \GS^{m}_{\mathcal{E}}(\mathscr{C})$ such that $Y \in \Gen_{\mathcal{E}}(F)$. Note that $\Gen_{\mathcal{E}}(F) \subseteq \Gen(F)$. By \cref{prop:EAdapt}, we have that $\Ext^1(U,F) \subseteq \mathcal{E}$. Using \cref{lem:GenlocalEadapt}, we have:
\begin{enumerate}[label=$\bullet$, itemsep=1mm]
    \item any $\xi \in \Ext^1(U',Y)$ is obtained by a pushout of some $\eta \in \Ext^1(U',F)$ along an $\mathcal{E}$-epimorphism $\begin{tikzcd}
	 g: F & Y;
	\arrow[two heads,from=1-1, to=1-2]
   \end{tikzcd}$ and,
   \item $\Ext^1(U',Y) \subseteq \mathcal{E}$.
\end{enumerate}
We obtain the following commutative diagram.
\[\begin{tikzcd}
         & 0 & 0 & 0 &  \\
	0 & K' & K''  & 0 &  0 \\
     0 & F & E'' & U' & 0 \\
      0 & Y & PO & U' & 0 \\
       & 0 & 0 & 0 & 
	\arrow[from=2-1, to=2-2]
	\arrow[tail, from=2-2, to=2-3]
	\arrow[two heads,from=2-3, to=2-4]
	\arrow[from=2-4, to=2-5]
        \arrow[from=3-1, to=3-2]
	\arrow[tail, from=3-2, to=3-3]
	\arrow[two heads,from=3-3, to=3-4]
	\arrow[from=3-4, to=3-5]
        \arrow[from=4-1, to=4-2]
	\arrow[tail, from=4-2, to=4-3]
	\arrow[two heads,from=4-3, to=4-4]
	\arrow[from=4-4, to=4-5]
        \arrow[from=4-2, to=5-2]
	\arrow[two heads, from=3-2, to=4-2]
	\arrow[tail,from=2-2, to=3-2]
	\arrow[from=1-2, to=2-2]
        \arrow[from=4-3, to=5-3]
	\arrow[two heads, from=3-3, to=4-3]
	\arrow[tail,from=2-3, to=3-3]
	\arrow[from=1-3, to=2-3]
        \arrow[from=4-4, to=5-4]
	\arrow[equals, from=3-4, to=4-4]
	\arrow[tail,from=2-4, to=3-4]
	\arrow[from=1-4, to=2-4]
\end{tikzcd} \]

The short sequences given by the three columns and those given by the last two rows are exact. By \cref{lem:3x3}, the short sequence given by the first row is exact, too. Therefore $K' \cong K''$.

As $F \in \GS^{m}_{\mathcal{E}}(\mathscr{C})$, we have that $E'' \in \GS_{\mathcal{E}}(\mathscr{C})$ by induction hypothesis. So, in the diagram above, the short sequences given by the three rows and those given by the first and third columns are $\mathcal{E}$-exact. By \cref{lem:3x3}, the short sequence given by the middle column is $\mathcal{E}$-exact. Thus $PO \in \Gen_\mathcal{E}(E'') \subseteq \GS_{\mathcal{E}}(\mathscr{C})$.

As every short exact sequence in $\Ext^1(U',Y)$ is obtained by a pushout of a short exact sequence in $\Ext^1(U',F)$, we get that $PO \cong E' \in \GS_{\mathcal{E}}(\mathscr{C})$.

\bigskip

Assume that  $Y \in \Sub_\mathcal{E}(\GS^{m}_{\mathcal{E}}(\mathscr{C}))$. There exists a short exact sequence \[\begin{tikzcd}
	0 & Y & F  & C&  0,
	\arrow[from=1-1, to=1-2]
	\arrow[tail, from=1-2, to=1-3]
	\arrow[two heads,from=1-3, to=1-4]
	\arrow[from=1-4, to=1-5]
\end{tikzcd} \] with $F \in \GS^{m}_{\mathcal{E}}(\mathscr{C})$. We obtain the following pushout diagram.
\[\begin{tikzcd}
         & 0 & 0 & 0 &  \\
	0 & C' & C''  & 0 &  0 \\
     0 & F & PO & U' & 0 \\
      0 & Y & E' & U' & 0 \\
       & 0 & 0 & 0 & 
	\arrow[from=2-1, to=2-2]
	\arrow[tail, from=2-2, to=2-3]
	\arrow[two heads,from=2-3, to=2-4]
	\arrow[from=2-4, to=2-5]
        \arrow[from=3-1, to=3-2]
	\arrow[tail, from=3-2, to=3-3]
	\arrow[two heads,from=3-3, to=3-4]
	\arrow[from=3-4, to=3-5]
        \arrow[from=4-1, to=4-2]
	\arrow[tail, from=4-2, to=4-3]
	\arrow[two heads,from=4-3, to=4-4]
	\arrow[from=4-4, to=4-5]
        \arrow[from=5-2, to=4-2]
	\arrow[tail, from=4-2, to=3-2]
	\arrow[two heads,from=3-2, to=2-2]
	\arrow[from=2-2, to=1-2]
        \arrow[from=5-3, to=4-3]
	\arrow[tail, from=4-3, to=3-3]
	\arrow[two heads,from=3-3, to=2-3]
	\arrow[from=2-3, to=1-3]
        \arrow[from=5-4, to=4-4]
	\arrow[equals, from=4-4, to=3-4]
	\arrow[two heads,from=3-4, to=2-4]
	\arrow[from=2-4, to=1-4]
\end{tikzcd} \]

The short sequences given by the three columns and those given by the last two rows are exact. By \cref{lem:3x3}, the short sequence given by the first row is exact, too. So $C' \cong C''$.

As $F,Y,T \in \GS_{\mathcal{E}}(\mathscr{C})$, the two last rows are in $\mathcal{E}$ by \cref{prop:EAdapt}. Moreover we know that $F \in \GS^{m}_{\mathcal{E}}(\mathscr{C})$, which implies that $PO \in \GS_{\mathcal{E}}(\mathscr{C})$ by induction hypothesis.

The short sequences given by the three rows and those given by the first and third columns are $\mathcal{E}$-exact. By \cref{lem:3x3}, the short sequence given by the middle column is also $\mathcal{E}$-exact. Then, $E' \in \Sub_\mathcal{E}(PO) \subseteq \GS_{\mathcal{E}}(\mathscr{C})$. 

In either case, we get $E' \in \GS_\mathcal{E}(\mathscr{C})$. This completes the proof by induction.
\end{proof}

\begin{prop}
\label{prop:Extclosed}
Let $\mathscr{C} \subseteq \mathscr{A}$ be an $\mathcal{E}$-adapted subcategory closed under extensions. Then $\GS_{\mathcal{E}}(\mathscr{C})$ is closed under extensions.
\end{prop}

\begin{proof}
 Let us prove by induction that, for any $m \in \mathbb{N}$, if we have a short exact sequence
\[\begin{tikzcd}
	0 & X & E  & Y &  0,
	\arrow[from=1-1, to=1-2]
	\arrow[tail, from=1-2, to=1-3]
	\arrow[two heads,from=1-3, to=1-4]
	\arrow[from=1-4, to=1-5]
\end{tikzcd} \] with  $X\in \GS_{\mathcal{E}}(\mathscr{C})$ and $Y\in \GS^m_{\mathcal{E}} (\mathscr{C})$, then $E \in  \GS_{\mathcal{E}}(\mathscr{C})$. The $m=0$ case follows \cref{lemmaExtclosedinduction}.

Fix $m \in \mathbb{N}$. Assume that if we have the short exact sequence 
\[\begin{tikzcd}
	0 & U & E' & V &  0,
	\arrow[from=1-1, to=1-2]
	\arrow[tail, from=1-2, to=1-3]
	\arrow[two heads,from=1-3, to=1-4]
	\arrow[from=1-4, to=1-5]
\end{tikzcd} \]
with $U \in \GS_{\mathcal{E}}(\mathscr{C})$ and $V \in \GS^{m}_{\mathcal{E}}(\mathscr{C})$, then $E' \in \GS_{\mathcal{E}}(\mathscr{C})$. Consider a short exact sequence \[\begin{tikzcd}
	0 & X & E  & Y &  0,
	\arrow[from=1-1, to=1-2]
	\arrow[tail, from=1-2, to=1-3]
	\arrow[two heads,from=1-3, to=1-4]
	\arrow[from=1-4, to=1-5]
\end{tikzcd} \] with $X \in \GS_\mathcal{E}(\mathscr{C})$ and $Y \in \GS_\mathcal{E}^{m+1}(\mathscr{C})$. As before, we must consider two cases.

Suppose $Y \in \Sub_\mathcal{E}(F)$ for some $F \in \GS^{m}_{\mathcal{E}}(\mathscr{C})$. By \cref{prop:EAdapt}, we have that $\Ext^1(F,X) \subseteq \mathcal{E}$. By \cref{lem:SublocalEadapt}, we have:
\begin{enumerate}[label=$\bullet$, itemsep=1mm]
    \item any $\xi \in \Ext^1(Y,X)$ is obtained by a pullback of some $\eta \in \Ext^1(F,X)$ along an $\mathcal{E}$-monomorphism $\begin{tikzcd}
	 f: Y & F;
	\arrow[tail,from=1-1, to=1-2]
   \end{tikzcd}$ and,
    \item $\Ext^1(Y,X) \subseteq \mathcal{E}$.
\end{enumerate}
We obtain the following commutative diagram.
\[\begin{tikzcd}
         & 0 & 0 & 0 &  \\
	0 & 0 & C''  & C' &  0 \\
     0 & X & E'' & F & 0 \\
      0 & X & PB & Y & 0 \\
       & 0 & 0 & 0 & 
	\arrow[from=2-1, to=2-2]
	\arrow[tail, from=2-2, to=2-3]
	\arrow[two heads,from=2-3, to=2-4]
	\arrow[from=2-4, to=2-5]
        \arrow[from=3-1, to=3-2]
	\arrow[tail, from=3-2, to=3-3]
	\arrow[two heads,from=3-3, to=3-4]
	\arrow[from=3-4, to=3-5]
        \arrow[from=4-1, to=4-2]
	\arrow[tail, from=4-2, to=4-3]
	\arrow[two heads,from=4-3, to=4-4]
	\arrow[from=4-4, to=4-5]
        \arrow[from=5-2, to=4-2]
	\arrow[equals, from=4-2, to=3-2]
	\arrow[two heads,from=3-2, to=2-2]
	\arrow[from=2-2, to=1-2]
        \arrow[from=5-3, to=4-3]
	\arrow[tail, from=4-3, to=3-3]
	\arrow[two heads,from=3-3, to=2-3]
	\arrow[from=2-3, to=1-3]
        \arrow[from=5-4, to=4-4]
	\arrow[tail, from=4-4, to=3-4]
	\arrow[two heads,from=3-4, to=2-4]
	\arrow[from=2-4, to=1-4]
\end{tikzcd} \]

The short sequences given by the three columns and those given by the last two rows are exact. By \cref{lem:3x3}, the short sequence given by the first row is exact, too. Then $C'' \cong C$. Moreover, as $F \in \GS^{m}_{\mathcal{E}}(\mathscr{C})$, we have that $E'' \in \GS_{\mathcal{E}}(\mathscr{C})$ by induction hypothesis.

The short exact sequences given by the three rows and those given by the first and third columns are in ${\mathcal{E}}$. By \cref{lem:3x3}, the short sequence given by the middle column is also $\mathcal{E}$-exact. So $PB \in \Sub_\mathcal{E}(E'') \subseteq \GS_{\mathcal{E}}(\mathscr{C})$. By the fact that every short exact sequence in $\Ext^1(Y,X)$ is obtained by a pullback of a short exact sequence in $\Ext^1(E',X)$, we get that $PB \cong E \in \GS_{\mathcal{E}}(\mathscr{C})$.

Now, assume that $Y \in \Gen_\mathcal{E}(F)$ for some $F \in \GS_\mathcal{E}^m(\mathscr{C})$. Therefore, there exists a short exact sequence \[\begin{tikzcd}
	0 & K' & F  & Y &  0 & \in \mathcal{E}.
	\arrow[from=1-1, to=1-2]
	\arrow[tail, from=1-2, to=1-3]
	\arrow[two heads,from=1-3, to=1-4]
	\arrow[from=1-4, to=1-5]
\end{tikzcd} \] We obtain the following pullback diagram.
\[\begin{tikzcd}
         & 0 & 0 & 0 &  \\
	0 & X & E  & Y &  0 \\
     0 & X & PB & F & 0 \\
      0 & 0 & K & K' & 0 \\
       & 0 & 0 & 0 & 
	\arrow[from=2-1, to=2-2]
	\arrow[tail, from=2-2, to=2-3]
	\arrow[two heads,from=2-3, to=2-4]
	\arrow[from=2-4, to=2-5]
        \arrow[from=3-1, to=3-2]
	\arrow[tail, from=3-2, to=3-3]
	\arrow[two heads,from=3-3, to=3-4]
	\arrow[from=3-4, to=3-5]
        \arrow[from=4-1, to=4-2]
	\arrow[tail, from=4-2, to=4-3]
	\arrow[two heads,from=4-3, to=4-4]
	\arrow[from=4-4, to=4-5]
        \arrow[from=5-2, to=4-2]
	\arrow[tail, from=4-2, to=3-2]
	\arrow[equals,from=3-2, to=2-2]
	\arrow[from=2-2, to=1-2]
        \arrow[from=5-3, to=4-3]
	\arrow[tail, from=4-3, to=3-3]
	\arrow[two heads,from=3-3, to=2-3]
	\arrow[from=2-3, to=1-3]
        \arrow[from=5-4, to=4-4]
	\arrow[tail, from=4-4, to=3-4]
	\arrow[two heads,from=3-4, to=2-4]
	\arrow[from=2-4, to=1-4]
\end{tikzcd} \]
Once again, by \cref{lem:3x3}, the short sequence given by the last row is exact, and so $K \cong K'$. Furthermore, we have that $F,X,Y \in \GS_{\mathcal{E}}(\mathscr{C})$. So the first two rows are in $\mathcal{E}$ by \cref{prop:EAdapt}. Moreover $F \in \GS^{m}_{\mathcal{E}}(\mathscr{C})$. So we have that $PB \in \GS_{\mathcal{E}}(\mathscr{C})$ by induction hypothesis. Using \cref{lem:3x3}, the short sequence given by the middle column is ${\mathcal{E}}$-exact. Therefore $E \in \Gen_\mathcal{E}(PB) \subseteq \GS_{\mathcal{E}}(\mathscr{C})$. 

This completes the induction proof, and allows us to get that $\GS_\mathcal{E}(\mathscr{C})$ is closed under extensions.
\end{proof}

We can state the following corollary, which is a direct consequence of \cref{lem:proponGS} \ref{GSb} and the last proposition.

\begin{cor} \label{cor:ESerre} Let $\mathscr{C} \subseteq \mathscr{A}$ be an $\mathcal{E}$-adapted subcategory closed under extensions. Then the following statements hold:
\begin{enumerate}[label=$(\alph*)$, itemsep=1mm]
    \item $\GS_\mathcal{E}(\mathscr{C})$ is an ${\mathcal{E}}$-torsion class;
    \item $\GS_\mathcal{E}(\mathscr{C})$ is an ${\mathcal{E}}$-cotorsion class;
    \item $\GS_\mathcal{E}(\mathscr{C})$ is $\mathcal{E}$-Serre; and,
    \item $\GS_\mathcal{E}(\mathscr{C})$ is the smallest $\mathcal{E}$-Serre category containing $\mathscr{C}$.
\end{enumerate}
\end{cor}

Note the extension-closed assumption on $\mathscr{C}$ is essential.

\begin{example} \label{ex:Eadaptednonextclosed}
Consider the exact category $(\rep Q, \mathcal{E}_\diamond)$, where $Q = 1 \longrightarrow 2 \longrightarrow 3$. 
Let $\mathscr{C} = \add(X_{\llrr{2,3}}, X_{\llrr{1,2}})$ (See \cref{fig:NonExt-closed}). We can check that $\GS_{\mathcal{E}_\diamond}(\mathscr{C}) = \mathscr{C}$, and $\mathscr{C}$ is $\mathcal{E}_\diamond$-adapted but not closed under extensions. 
\begin{figure}[H]
    \centering
    \begin{tikzpicture}[line width=.3mm, ->, >= angle 60]
		\node[circle,line width=0.2mm,double,fill=lava!20,draw] (23) at (1,-1){\scalebox{.9}{$\llrr{2,3}$}};
		\node (3) at (0,-2){\scalebox{.9}{$\llrr{3}$}};
		\node (13) at (2,0){\scalebox{.9}{$\llrr{1,3}$}};
		\node (2) at (2,-2){\scalebox{.9}{$\llrr{2}$}};
		\node[circle,line width=0.2mm,double,fill=lava!20,draw] (12) at (3,-1){\scalebox{.9}{$\llrr{1,2}$}};
		\node (1) at (4,-2){\scalebox{.9}{$\llrr{1}$}};
								
		\draw (3) -- (23);
		\draw (23) -- (13);
		\draw (23) -- (2);
		\draw (13) -- (12);
		\draw (2) -- (12);
		\draw (12) -- (1);
        \end{tikzpicture}
    \caption[fragile]{The Auslander--Reiten quiver of $Q$ in \cref{ex:Eadaptednonextclosed}. We have $\mathscr{C} = \add \left( \protect\begin{tikzpicture}[baseline={(0,-.1)}]
        \node[circle, line width=0.2mm, double, fill=lava!20,minimum size=1em,draw] at (0,0){$ $};
    \end{tikzpicture}\right) = \GS_{\mathcal{E}_{\diamond}}(\mathscr{C})$, and we can check that $\mathscr{C}$ is not closed under extensions even though it is $\mathcal{E}_{\diamond}$-adapted.}
    \label{fig:NonExt-closed}
\end{figure}
\end{example}

\subsection{Tilting objects in hereditary subcategories}
\label{ss:Tilting} In this section, we highlight some interesting properties of the $\GS_\mathcal{E}$-operator with tilting objects in $\mathscr{A}$. Recall that we assume $\mathscr{A}$ hereditary (\cref{conv:hereditary}).

\begin{definition}
\label{def:rigidtilting}
An object $T \in \mathscr{A}$ is said to be \new{rigid} whenever $\Ext_\mathscr{A}^1(T,T) = 0$. An object $T \in \mathscr{A}$ is said to be \new{tilting} if the following statements hold:
\begin{enumerate}[label=$\bullet$, itemsep=1mm]
\item $T$ is \emph{rigid}.
\item For any projective object $P$, there is a short exact sequence 
\[\begin{tikzcd}
	\varsigma: & 0 & P & T_0 & T_1&  0.
	\arrow[from=1-2, to=1-3]
	\arrow[tail, from=1-3, to=1-4]
	\arrow[two heads,from=1-4, to=1-5]
	\arrow[from=1-5, to=1-6]
\end{tikzcd}\] with $T_0,T_1 \in \add(T)$.
\end{enumerate} 
Denote by $\Tilt(\mathscr{A})$ the collection of tilting objects in $\mathscr{A}$, considered up to isomorphisms.
\end{definition}

We recall the following crucial result, which gives us an easier criterion for tilting objects of $\mathscr{A}$.

\begin{prop} \label{prop:numberindectilt}
Assume that $\mathscr{A}$ admits $n$ nonisomorphic indecomposable projective objects. Then, for any $T \in \mathscr{A}$, we have that $T \in \Tilt(\mathscr{A})$ if and only if $T$ is both rigid and made of $n$ nonisomorphic indecomposable summands.
\end{prop}

\begin{example} \label{ex:TiltA4} In \cref{fig:TiltA4}, we enumerate the fourteen tilting modules in $\rep(Q)$ for some $A_4$ type quiver $Q$.
 \begin{figure}[!ht]
        \centering
        \begin{tikzpicture}
				
				\begin{scope}[xshift=-3.4cm, yshift = -2cm,line width=.5mm,scale=.85]
					\node (a) at (0,0){\scalebox{.5}{\begin{tikzpicture}[line width=.3mm, ->, >= angle 60]
								\node[rectangle,line width=0.2mm, double,rounded corners,fill=orange!20,draw] (2) at (0,0){\scalebox{.9}{$\llrr{2}$}};
								\node[rectangle,line width=0.2mm, double,rounded corners,fill=orange!20,draw] (12) at (1,1){\scalebox{.9}{$\llrr{1,2}$}};
								\node[rectangle,line width=0.2mm, double,rounded corners,fill=orange!20,draw] (24) at (1,-1){\scalebox{.9}{$\llrr{2,4}$}};
								\node[rectangle,line width=0.2mm, double,rounded corners,fill=orange!20,draw] (4) at (0,-2){\scalebox{.9}{$\llrr{4}$}};
								\node (14) at (2,0){\scalebox{.9}{$\llrr{1,4}$}};
								\node (23) at (2,-2){\scalebox{.9}{$\llrr{2,3}$}};
								\node (34) at (3,1){\scalebox{.9}{$\llrr{3,4}$}};
								\node (13) at (3,-1){\scalebox{.9}{$\llrr{1,3}$}};
								\node (3) at (4,0){\scalebox{.9}{$\llrr{3}$}};
								\node (1) at (4,-2){\scalebox{.9}{$\llrr{1}$}};
								\draw (2) -- (12);
								\draw (2) -- (24);
								\draw (4) -- (24);
								\draw (12) -- (14);
								\draw (24) -- (14);
								\draw (24) -- (23);
								\draw (14) -- (34);
								\draw (14) -- (13);
								\draw (23) -- (13);
								\draw (34) -- (3);
								\draw (13) -- (3);
								\draw (13) -- (1);
					\end{tikzpicture}}};
					\node (b) at (3.5,0){\scalebox{.5}{\begin{tikzpicture}[line width=.3mm, ->, >= angle 60]
								\node (2) at (0,0){\scalebox{.9}{$\llrr{2}$}};
								\node[rectangle,line width=0.2mm, double,rounded corners,fill=orange!20,draw] (12) at (1,1){\scalebox{.9}{$\llrr{1,2}$}};
								\node[rectangle,line width=0.2mm, double,rounded corners,fill=orange!20,draw] (24) at (1,-1){\scalebox{.9}{$\llrr{2,4}$}};
								\node[rectangle,line width=0.2mm, double,rounded corners,fill=orange!20,draw] (4) at (0,-2){\scalebox{.9}{$\llrr{4}$}};
								\node[rectangle,line width=0.2mm, double,rounded corners,fill=orange!20,draw] (14) at (2,0){\scalebox{.9}{$\llrr{1,4}$}};
								\node (23) at (2,-2){\scalebox{.9}{$\llrr{2,3}$}};
								\node (34) at (3,1){\scalebox{.9}{$\llrr{3,4}$}};
								\node (13) at (3,-1){\scalebox{.9}{$\llrr{1,3}$}};
								\node (3) at (4,0){\scalebox{.9}{$\llrr{3}$}};
								\node (1) at (4,-2){\scalebox{.9}{$\llrr{1}$}};
								\draw (2) -- (12);
								\draw (2) -- (24);
								\draw (4) -- (24);
								\draw (12) -- (14);
								\draw (24) -- (14);
								\draw (24) -- (23);
								\draw (14) -- (34);
								\draw (14) -- (13);
								\draw (23) -- (13);
								\draw (34) -- (3);
								\draw (13) -- (3);
								\draw (13) -- (1);
					\end{tikzpicture}}};
					\node (c) at (7,0){\scalebox{.5}{\begin{tikzpicture}[line width=.3mm, ->, >= angle 60]
								\node[rectangle,line width=0.2mm, double,rounded corners,fill=orange!20,draw] (2) at (0,0){\scalebox{.9}{$\llrr{2}$}};
								\node[rectangle,line width=0.2mm, double,rounded corners,fill=orange!20,draw] (12) at (1,1){\scalebox{.9}{$\llrr{1,2}$}};
								\node[rectangle,line width=0.2mm, double,rounded corners,fill=orange!20,draw] (24) at (1,-1){\scalebox{.9}{$\llrr{2,4}$}};
								\node (4) at (0,-2){\scalebox{.9}{$\llrr{4}$}};
								\node (14) at (2,0){\scalebox{.9}{$\llrr{1,4}$}};
								\node[rectangle,line width=0.2mm, double,rounded corners,fill=orange!20,draw] (23) at (2,-2){\scalebox{.9}{$\llrr{2,3}$}};
								\node (34) at (3,1){\scalebox{.9}{$\llrr{3,4}$}};
								\node (13) at (3,-1){\scalebox{.9}{$\llrr{1,3}$}};
								\node (3) at (4,0){\scalebox{.9}{$\llrr{3}$}};
								\node (1) at (4,-2){\scalebox{.9}{$\llrr{1}$}};
								\draw (2) -- (12);
								\draw (2) -- (24);
								\draw (4) -- (24);
								\draw (12) -- (14);
								\draw (24) -- (14);
								\draw (24) -- (23);
								\draw (14) -- (34);
								\draw (14) -- (13);
								\draw (23) -- (13);
								\draw (34) -- (3);
								\draw (13) -- (3);
								\draw (13) -- (1);
					\end{tikzpicture}}};
					\node (d) at (10.5,0){\scalebox{.5}{\begin{tikzpicture}[line width=.3mm, ->, >= angle 60]
								\node (2) at (0,0){\scalebox{.9}{$\llrr{2}$}};
								\node[rectangle,line width=0.2mm, double,rounded corners,fill=orange!20,draw] (12) at (1,1){\scalebox{.9}{$\llrr{1,2}$}};
								\node[rectangle,line width=0.2mm, double,rounded corners,fill=orange!20,draw] (24) at (1,-1){\scalebox{.9}{$\llrr{2,4}$}};
								\node (4) at (0,-2){\scalebox{.9}{$\llrr{4}$}};
								\node[rectangle,line width=0.2mm, double,rounded corners,fill=orange!20,draw] (14) at (2,0){\scalebox{.9}{$\llrr{1,4}$}};
								\node[rectangle,line width=0.2mm, double,rounded corners,fill=orange!20,draw] (23) at (2,-2){\scalebox{.9}{$\llrr{2,3}$}};
								\node (34) at (3,1){\scalebox{.9}{$\llrr{3,4}$}};
								\node (13) at (3,-1){\scalebox{.9}{$\llrr{1,3}$}};
								\node (3) at (4,0){\scalebox{.9}{$\llrr{3}$}};
								\node (1) at (4,-2){\scalebox{.9}{$\llrr{1}$}};
								\draw (2) -- (12);
								\draw (2) -- (24);
								\draw (4) -- (24);
								\draw (12) -- (14);
								\draw (24) -- (14);
								\draw (24) -- (23);
								\draw (14) -- (34);
								\draw (14) -- (13);
								\draw (23) -- (13);
								\draw (34) -- (3);
								\draw (13) -- (3);
								\draw (13) -- (1);
					\end{tikzpicture}}};
					\node (e) at (0,-3){\scalebox{.5}{\begin{tikzpicture}[line width=.3mm, ->, >= angle 60]
								\node (2) at (0,0){\scalebox{.9}{$\llrr{2}$}};
								\node (12) at (1,1){\scalebox{.9}{$\llrr{1,2}$}};
								\node[rectangle,line width=0.2mm, double,rounded corners,fill=orange!20,draw] (24) at (1,-1){\scalebox{.9}{$\llrr{2,4}$}};
								\node[rectangle,line width=0.2mm, double,rounded corners,fill=orange!20,draw] (4) at (0,-2){\scalebox{.9}{$\llrr{4}$}};
								\node[rectangle,line width=0.2mm, double,rounded corners,fill=orange!20,draw] (14) at (2,0){\scalebox{.9}{$\llrr{1,4}$}};
								\node (23) at (2,-2){\scalebox{.9}{$\llrr{2,3}$}};
								\node[rectangle,line width=0.2mm, double,rounded corners,fill=orange!20,draw] (34) at (3,1){\scalebox{.9}{$\llrr{3,4}$}};
								\node (13) at (3,-1){\scalebox{.9}{$\llrr{1,3}$}};
								\node (3) at (4,0){\scalebox{.9}{$\llrr{3}$}};
								\node (1) at (4,-2){\scalebox{.9}{$\llrr{1}$}};
								\draw (2) -- (12);
								\draw (2) -- (24);
								\draw (4) -- (24);
								\draw (12) -- (14);
								\draw (24) -- (14);
								\draw (24) -- (23);
								\draw (14) -- (34);
								\draw (14) -- (13);
								\draw (23) -- (13);
								\draw (34) -- (3);
								\draw (13) -- (3);
								\draw (13) -- (1);
					\end{tikzpicture}}};
					\node (f) at (3.5,-3){\scalebox{.5}{\begin{tikzpicture}[line width=.3mm, ->, >= angle 60]
								\node (2) at (0,0){\scalebox{.9}{$\llrr{2}$}};
								\node[rectangle,line width=0.2mm, double,rounded corners,fill=orange!20,draw] (12) at (1,1){\scalebox{.9}{$\llrr{1,2}$}};
								\node (24) at (1,-1){\scalebox{.9}{$\llrr{2,4}$}};
								\node[rectangle,line width=0.2mm, double,rounded corners,fill=orange!20,draw] (4) at (0,-2){\scalebox{.9}{$\llrr{4}$}};
								\node[rectangle,line width=0.2mm, double,rounded corners,fill=orange!20,draw] (14) at (2,0){\scalebox{.9}{$\llrr{1,4}$}};
								\node (23) at (2,-2){\scalebox{.9}{$\llrr{2,3}$}};
								\node (34) at (3,1){\scalebox{.9}{$\llrr{3,4}$}};
								\node (13) at (3,-1){\scalebox{.9}{$\llrr{1,3}$}};
								\node (3) at (4,0){\scalebox{.9}{$\llrr{3}$}};
								\node[rectangle,line width=0.2mm, double,rounded corners,fill=orange!20,draw] (1) at (4,-2){\scalebox{.9}{$\llrr{1}$}};
								\draw (2) -- (12);
								\draw (2) -- (24);
								\draw (4) -- (24);
								\draw (12) -- (14);
								\draw (24) -- (14);
								\draw (24) -- (23);
								\draw (14) -- (34);
								\draw (14) -- (13);
								\draw (23) -- (13);
								\draw (34) -- (3);
								\draw (13) -- (3);
								\draw (13) -- (1);
					\end{tikzpicture}}};
					\node (g) at (7,-3){\scalebox{.5}{\begin{tikzpicture}[line width=.3mm, ->, >= angle 60]
								\node (2) at (0,0){\scalebox{.9}{$\llrr{2}$}};
								\node[rectangle,line width=0.2mm, double,rounded corners,fill=orange!20,draw] (12) at (1,1){\scalebox{.9}{$\llrr{1,2}$}};
								\node (24) at (1,-1){\scalebox{.9}{$\llrr{2,4}$}};
								\node (4) at (0,-2){\scalebox{.9}{$\llrr{4}$}};
								\node[rectangle,line width=0.2mm, double,rounded corners,fill=orange!20,draw] (14) at (2,0){\scalebox{.9}{$\llrr{1,4}$}};
								\node[rectangle,line width=0.2mm, double,rounded corners,fill=orange!20,draw] (23) at (2,-2){\scalebox{.9}{$\llrr{2,3}$}};
								\node (34) at (3,1){\scalebox{.9}{$\llrr{3,4}$}};
								\node[rectangle,line width=0.2mm, double,rounded corners,fill=orange!20,draw] (13) at (3,-1){\scalebox{.9}{$\llrr{1,3}$}};
								\node (3) at (4,0){\scalebox{.9}{$\llrr{3}$}};
								\node (1) at (4,-2){\scalebox{.9}{$\llrr{1}$}};
								\draw (2) -- (12);
								\draw (2) -- (24);
								\draw (4) -- (24);
								\draw (12) -- (14);
								\draw (24) -- (14);
								\draw (24) -- (23);
								\draw (14) -- (34);
								\draw (14) -- (13);
								\draw (23) -- (13);
								\draw (34) -- (3);
								\draw (13) -- (3);
								\draw (13) -- (1);
					\end{tikzpicture}}};
					\node (h) at (10.5,-3){\scalebox{.5}{\begin{tikzpicture}[line width=.3mm, ->, >= angle 60]
								\node (2) at (0,0){\scalebox{.9}{$\llrr{2}$}};
								\node (12) at (1,1){\scalebox{.9}{$\llrr{1,2}$}};
								\node[rectangle,line width=0.2mm, double,rounded corners,fill=orange!20,draw] (24) at (1,-1){\scalebox{.9}{$\llrr{2,4}$}};
								\node (4) at (0,-2){\scalebox{.9}{$\llrr{4}$}};
								\node[rectangle,line width=0.2mm, double,rounded corners,fill=orange!20,draw] (14) at (2,0){\scalebox{.9}{$\llrr{1,4}$}};
								\node[rectangle,line width=0.2mm, double,rounded corners,fill=orange!20,draw] (23) at (2,-2){\scalebox{.9}{$\llrr{2,3}$}};
								\node[rectangle,line width=0.2mm, double,rounded corners,fill=orange!20,draw] (34) at (3,1){\scalebox{.9}{$\llrr{3,4}$}};
								\node (13) at (3,-1){\scalebox{.9}{$\llrr{1,3}$}};
								\node (3) at (4,0){\scalebox{.9}{$\llrr{3}$}};
								\node (1) at (4,-2){\scalebox{.9}{$\llrr{1}$}};
								\draw (2) -- (12);
								\draw (2) -- (24);
								\draw (4) -- (24);
								\draw (12) -- (14);
								\draw (24) -- (14);
								\draw (24) -- (23);
								\draw (14) -- (34);
								\draw (14) -- (13);
								\draw (23) -- (13);
								\draw (34) -- (3);
								\draw (13) -- (3);
								\draw (13) -- (1);
					\end{tikzpicture}}};
					\node (i) at (0,-6){\scalebox{.5}{\begin{tikzpicture}[line width=.3mm, ->, >= angle 60]
								\node (2) at (0,0){\scalebox{.9}{$\llrr{2}$}};
								\node (12) at (1,1){\scalebox{.9}{$\llrr{1,2}$}};
								\node (24) at (1,-1){\scalebox{.9}{$\llrr{2,4}$}};
								\node[rectangle,line width=0.2mm, double,rounded corners,fill=orange!20,draw] (4) at (0,-2){\scalebox{.9}{$\llrr{4}$}};
								\node[rectangle,line width=0.2mm, double,rounded corners,fill=orange!20,draw] (14) at (2,0){\scalebox{.9}{$\llrr{1,4}$}};
								\node (23) at (2,-2){\scalebox{.9}{$\llrr{2,3}$}};
								\node[rectangle,line width=0.2mm, double,rounded corners,fill=orange!20,draw] (34) at (3,1){\scalebox{.9}{$\llrr{3,4}$}};
								\node (13) at (3,-1){\scalebox{.9}{$\llrr{1,3}$}};
								\node (3) at (4,0){\scalebox{.9}{$\llrr{3}$}};
								\node[rectangle,line width=0.2mm, double,rounded corners,fill=orange!20,draw] (1) at (4,-2){\scalebox{.9}{$\llrr{1}$}};
								\draw (2) -- (12);
								\draw (2) -- (24);
								\draw (4) -- (24);
								\draw (12) -- (14);
								\draw (24) -- (14);
								\draw (24) -- (23);
								\draw (14) -- (34);
								\draw (14) -- (13);
								\draw (23) -- (13);
								\draw (34) -- (3);
								\draw (13) -- (3);
								\draw (13) -- (1);
					\end{tikzpicture}}};
					\node (j) at (3.5,-6){\scalebox{.5}{\begin{tikzpicture}[line width=.3mm, ->, >= angle 60]
								\node (2) at (0,0){\scalebox{.9}{$\llrr{2}$}};
								\node[rectangle,line width=0.2mm, double,rounded corners,fill=orange!20,draw] (12) at (1,1){\scalebox{.9}{$\llrr{1,2}$}};
								\node (24) at (1,-1){\scalebox{.9}{$\llrr{2,4}$}};
								\node (4) at (0,-2){\scalebox{.9}{$\llrr{4}$}};
								\node[rectangle,line width=0.2mm, double,rounded corners,fill=orange!20,draw] (14) at (2,0){\scalebox{.9}{$\llrr{1,4}$}};
								\node (23) at (2,-2){\scalebox{.9}{$\llrr{2,3}$}};
								\node (34) at (3,1){\scalebox{.9}{$\llrr{3,4}$}};
								\node[rectangle,line width=0.2mm, double,rounded corners,fill=orange!20,draw] (13) at (3,-1){\scalebox{.9}{$\llrr{1,3}$}};
								\node (3) at (4,0){\scalebox{.9}{$\llrr{3}$}};
								\node[rectangle,line width=0.2mm, double,rounded corners,fill=orange!20,draw] (1) at (4,-2){\scalebox{.9}{$\llrr{1}$}};
								\draw (2) -- (12);
								\draw (2) -- (24);
								\draw (4) -- (24);
								\draw (12) -- (14);
								\draw (24) -- (14);
								\draw (24) -- (23);
								\draw (14) -- (34);
								\draw (14) -- (13);
								\draw (23) -- (13);
								\draw (34) -- (3);
								\draw (13) -- (3);
								\draw (13) -- (1);
					\end{tikzpicture}}};
					\node (k) at (7,-6){\scalebox{.5}{\begin{tikzpicture}[line width=.3mm, ->, >= angle 60]
								\node (2) at (0,0){\scalebox{.9}{$\llrr{2}$}};
								\node (12) at (1,1){\scalebox{.9}{$\llrr{1,2}$}};
								\node (24) at (1,-1){\scalebox{.9}{$\llrr{2,4}$}};
								\node (4) at (0,-2){\scalebox{.9}{$\llrr{4}$}};
								\node[rectangle,line width=0.2mm, double,rounded corners,fill=orange!20,draw] (14) at (2,0){\scalebox{.9}{$\llrr{1,4}$}};
								\node[rectangle,line width=0.2mm, double,rounded corners,fill=orange!20,draw] (23) at (2,-2){\scalebox{.9}{$\llrr{2,3}$}};
								\node[rectangle,line width=0.2mm, double,rounded corners,fill=orange!20,draw] (34) at (3,1){\scalebox{.9}{$\llrr{3,4}$}};
								\node[rectangle,line width=0.2mm, double,rounded corners,fill=orange!20,draw] (13) at (3,-1){\scalebox{.9}{$\llrr{1,3}$}};
								\node (3) at (4,0){\scalebox{.9}{$\llrr{3}$}};
								\node (1) at (4,-2){\scalebox{.9}{$\llrr{1}$}};
								\draw (2) -- (12);
								\draw (2) -- (24);
								\draw (4) -- (24);
								\draw (12) -- (14);
								\draw (24) -- (14);
								\draw (24) -- (23);
								\draw (14) -- (34);
								\draw (14) -- (13);
								\draw (23) -- (13);
								\draw (34) -- (3);
								\draw (13) -- (3);
								\draw (13) -- (1);
					\end{tikzpicture}}};
					\node (l) at (10.5,-6){\scalebox{.5}{\begin{tikzpicture}[line width=.3mm, ->, >= angle 60]
								\node (2) at (0,0){\scalebox{.9}{$\llrr{2}$}};
								\node (12) at (1,1){\scalebox{.9}{$\llrr{1,2}$}};
								\node (24) at (1,-1){\scalebox{.9}{$\llrr{2,4}$}};
								\node (4) at (0,-2){\scalebox{.9}{$\llrr{4}$}};
								\node[rectangle,line width=0.2mm, double,rounded corners,fill=orange!20,draw] (14) at (2,0){\scalebox{.9}{$\llrr{1,4}$}};
								\node (23) at (2,-2){\scalebox{.9}{$\llrr{2,3}$}};
								\node[rectangle,line width=0.2mm, double,rounded corners,fill=orange!20,draw] (34) at (3,1){\scalebox{.9}{$\llrr{3,4}$}};
								\node[rectangle,line width=0.2mm, double,rounded corners,fill=orange!20,draw] (13) at (3,-1){\scalebox{.9}{$\llrr{1,3}$}};
								\node (3) at (4,0){\scalebox{.9}{$\llrr{3}$}};
								\node[rectangle,line width=0.2mm, double,rounded corners,fill=orange!20,draw] (1) at (4,-2){\scalebox{.9}{$\llrr{1}$}};
								\draw (2) -- (12);
								\draw (2) -- (24);
								\draw (4) -- (24);
								\draw (12) -- (14);
								\draw (24) -- (14);
								\draw (24) -- (23);
								\draw (14) -- (34);
								\draw (14) -- (13);
								\draw (23) -- (13);
								\draw (34) -- (3);
								\draw (13) -- (3);
								\draw (13) -- (1);
					\end{tikzpicture}}};
					\node (m) at (3.5,-9){\scalebox{.5}{\begin{tikzpicture}[line width=.3mm, ->, >= angle 60]
								\node (2) at (0,0){\scalebox{.9}{$\llrr{2}$}};
								\node (12) at (1,1){\scalebox{.9}{$\llrr{1,2}$}};
								\node (24) at (1,-1){\scalebox{.9}{$\llrr{2,4}$}};
								\node (4) at (0,-2){\scalebox{.9}{$\llrr{4}$}};
								\node (14) at (2,0){\scalebox{.9}{$\llrr{1,4}$}};
								\node[rectangle,line width=0.2mm, double,rounded corners,fill=orange!20,draw] (23) at (2,-2){\scalebox{.9}{$\llrr{2,3}$}};
								\node[rectangle,line width=0.2mm, double,rounded corners,fill=orange!20,draw] (34) at (3,1){\scalebox{.9}{$\llrr{3,4}$}};
								\node[rectangle,line width=0.2mm, double,rounded corners,fill=orange!20,draw] (13) at (3,-1){\scalebox{.9}{$\llrr{1,3}$}};
								\node[rectangle,line width=0.2mm, double,rounded corners,fill=orange!20,draw] (3) at (4,0){\scalebox{.9}{$\llrr{3}$}};
								\node (1) at (4,-2){\scalebox{.9}{$\llrr{1}$}};
								\draw (2) -- (12);
								\draw (2) -- (24);
								\draw (4) -- (24);
								\draw (12) -- (14);
								\draw (24) -- (14);
								\draw (24) -- (23);
								\draw (14) -- (34);
								\draw (14) -- (13);
								\draw (23) -- (13);
								\draw (34) -- (3);
								\draw (13) -- (3);
								\draw (13) -- (1);
					\end{tikzpicture}}};
					\node (n) at (7,-9){\scalebox{.5}{\begin{tikzpicture}[line width=.3mm, ->, >= angle 60]
								\node (2) at (0,0){\scalebox{.9}{$\llrr{2}$}};
								\node (12) at (1,1){\scalebox{.9}{$\llrr{1,2}$}};
								\node (24) at (1,-1){\scalebox{.9}{$\llrr{2,4}$}};
								\node (4) at (0,-2){\scalebox{.9}{$\llrr{4}$}};
								\node (14) at (2,0){\scalebox{.9}{$\llrr{1,4}$}};
								\node (23) at (2,-2){\scalebox{.9}{$\llrr{2,3}$}};
								\node[rectangle,line width=0.2mm, double,rounded corners,fill=orange!20,draw] (34) at (3,1){\scalebox{.9}{$\llrr{3,4}$}};
								\node[rectangle,line width=0.2mm, double,rounded corners,fill=orange!20,draw] (13) at (3,-1){\scalebox{.9}{$\llrr{1,3}$}};
								\node[rectangle,line width=0.2mm, double,rounded corners,fill=orange!20,draw] (3) at (4,0){\scalebox{.9}{$\llrr{3}$}};
								\node[rectangle,line width=0.2mm, double,rounded corners,fill=orange!20,draw] (1) at (4,-2){\scalebox{.9}{$\llrr{1}$}};
								\draw (2) -- (12);
								\draw (2) -- (24);
								\draw (4) -- (24);
								\draw (12) -- (14);
								\draw (24) -- (14);
								\draw (24) -- (23);
								\draw (14) -- (34);
								\draw (14) -- (13);
								\draw (23) -- (13);
								\draw (34) -- (3);
								\draw (13) -- (3);
								\draw (13) -- (1);
					\end{tikzpicture}}};
				\end{scope}
			\end{tikzpicture}
        \caption[fragile]{Tilting representations of $\rep(Q)$ for some $A_4$ type quiver $Q$. A tilting object in this figure is the direct sum of indecomposable objects in  $\protect\begin{tikzpicture}[baseline={(0,-.1)}]
        \node[rectangle, line width=0.2mm, double, rounded corners, fill=orange!20,minimum size=1em,draw] at (0,0){$\ \ $};
    \end{tikzpicture}$ for one copy of the Auslander--Reiten quiver of $Q$}
        \label{fig:TiltA4}
    \end{figure}
\end{example}

The following result allows us to handle objects in $\Gen_\mathcal{E}(T)$ and those in $\Sub_\mathcal{E}(T)$ for $T \in \Tilt(\mathscr{A})$. See \cite{ASS06} for more details.

\begin{prop} \label{prop:GenSubT} Let $T \in \Tilt(\mathscr{A})$. The following assertions hold:
\begin{enumerate}[label=$(\roman*)$, itemsep=1mm]
    \item \label{GenET1} for any $M \in \Gen_\mathcal{E}(T)$, there exists a short $\mathcal{E}$-exact sequence \[\begin{tikzcd}
	 0 & T_1 & T_0 & M&  0.
	\arrow[from=1-1, to=1-2]
	\arrow[tail, from=1-2, to=1-3]
	\arrow[two heads,from=1-3, to=1-4]
	\arrow[from=1-4, to=1-5]
\end{tikzcd}\] where $T_0,T_1 \in \add(T)$; and,
\item \label{SubET1} for any $M \in \Sub_\mathcal{E}(T)$, there exists a short $\mathcal{E}$-exact sequence \[\begin{tikzcd}
	 0 & M & \Theta_0 &  \Theta_1 &  0.
	\arrow[from=1-1, to=1-2]
	\arrow[tail, from=1-2, to=1-3]
	\arrow[two heads,from=1-3, to=1-4]
	\arrow[from=1-4, to=1-5]
\end{tikzcd}\] where $\Theta_0,\Theta_1 \in \add(T)$.
\end{enumerate}   
\end{prop}

Before being further in the study of $\GS_\mathcal{E}(T)$ for any rigid object $T \in \mathscr{A}$, we can state a few noticeable results related to $\mathcal{E}$-adaptability.

\begin{lemma} \label{lem:TiltEadapt}
    For any rigid object $T \in \mathscr{A}$, the category $\add(T)$ is $\mathcal{E}$-adapted and closed under extension.
\end{lemma}

Thanks to what we proved previously, we have the following result.

\begin{cor} \label{cor:rigidEadapt}
    If $T \in \mathscr{A}$ is rigid, then $\GS_{\mathcal{E}}(T)$ is $\mathcal{E}$-adapted and closed under extension.
\end{cor}

\begin{proof}
     It follows from \cref{prop:Extclosed} and \cref{lem:TiltEadapt}.
\end{proof}

\begin{theorem}
\label{thm:maximal}
 Let $T \in \Tilt(\mathscr{A})$ be a tilting object. Then $\GS_\mathcal{E}(T)$ is a maximal $\mathcal{E}$-adapted subcategory of $\mathscr{A}$ closed under extensions which contains $T$.
\end{theorem}

\begin{proof}
Let $\mathscr{C} \subseteq \mathscr{A}$ be an $\mathcal{E}$-adapted subcategory closed under extensions that contains $T$. By \cref{conv:addsubcat}, to prove that $\mathscr{C} \subseteq \GS_\mathcal{E}(T)$, it is enough to show that all the objects in $\mathscr{C}$ are in $\GS_\mathcal{E}(T)$. 

Consider $V \in \mathscr{C}$. Suppose $V \in \proj(\mathscr{A})$. Since $T$ is tilting, by \cref{def:rigidtilting}, there is a short exact sequence 
\[\begin{tikzcd}
	\xi: & 0 & V & \Theta_0 & \Theta_1 &  0, & 
	\arrow[from=1-2, to=1-3]
	\arrow[tail, from=1-3, to=1-4]
	\arrow[two heads,from=1-4, to=1-5]
	\arrow[from=1-5, to=1-6]
\end{tikzcd}\] with $\Theta_0,\Theta_1 \in \add(T)$.
Since $V,\Theta_1 \in \mathscr{C}$ which is $\mathcal{E}$-adapted, then $\xi \in \mathcal{E}$. Thus $V \in \Sub_\mathcal{E}(T_0) \subseteq \GS_\mathcal{E}(T)$.

Assume that $V \notin \proj(\mathscr{A})$. Consider the projective resolution of $V$
\[\begin{tikzcd}
     0 & P_1 & P_0 & V &  0 & 
	\arrow[from=1-1, to=1-2]
	\arrow[tail, from=1-2, to=1-3]
	\arrow[two heads,from=1-3, to=1-4]
	\arrow[from=1-4, to=1-5]
\end{tikzcd}\] with $P_0,P_1 \in \proj(\mathscr{A})$. Since $T$ is tilting, by \cref{def:rigidtilting}, there is a short exact sequence
\[\begin{tikzcd}
   0 & P_1 & \Theta_0 & \Theta_1 &  0, & 
	\arrow[from=1-1, to=1-2]
	\arrow[tail,"\psi", from=1-2, to=1-3]
	\arrow[two heads,from=1-3, to=1-4]
	\arrow[from=1-4, to=1-5]
\end{tikzcd}\] with $\Theta_0,\Theta_1 \in \add(T)$. 
We obtain the following diagram by taking the pushout along $\psi$ and by using the cokernel lifting property.
\[\begin{tikzcd}
         & 0 & 0 & 0 &  \\
	0 & P_1 & P_0  & V &  0 \\
     0 & \Theta_0 & PO & V & 0 \\
      0 & \Theta_1 & C & 0 & 0 \\
       & 0 & 0 & 0 & 
	\arrow[from=2-1, to=2-2]
	\arrow[tail, from=2-2, to=2-3]
	\arrow[two heads,from=2-3, to=2-4]
	\arrow[from=2-4, to=2-5]
        \arrow[from=3-1, to=3-2]
	\arrow[tail, from=3-2, to=3-3]
	\arrow[two heads,from=3-3, to=3-4]
	\arrow[from=3-4, to=3-5]
        \arrow[from=4-1, to=4-2]
	\arrow[tail, from=4-2, to=4-3]
	\arrow[two heads,from=4-3, to=4-4]
	\arrow[from=4-4, to=4-5]
        \arrow[from=4-2, to=5-2]
	\arrow[two heads, from=3-2, to=4-2]
	\arrow[tail,"\psi",from=2-2, to=3-2]
	\arrow[from=1-2, to=2-2]
        \arrow[from=4-3, to=5-3]
	\arrow[two heads, from=3-3, to=4-3]
	\arrow[tail,from=2-3, to=3-3]
	\arrow[from=1-3, to=2-3]
        \arrow[from=4-4, to=5-4]
	\arrow[two heads, from=3-4, to=4-4]
	\arrow[equals,from=2-4, to=3-4]
	\arrow[from=1-4, to=2-4]
\end{tikzcd} \]
As the short sequences given by the three columns and those given by the first two rows of the above diagram are exact, by \cref{lem:3x3}, the short sequence given by the last row is exact too. So $C \cong \Theta_1$. Since $V,\Theta_0 \in \mathscr{C}$, by assumption on $\mathscr{C}$, then the middle row lies in $\mathcal{E}$, and $W = PO \in \mathscr{C}$.

Since $T$ is tilting, by \cref{def:rigidtilting}, there is also a short exact sequence 
\[\begin{tikzcd}
     0 & P_0 & \Theta_2 & \Theta_3 &  0, & 
	\arrow[from=1-1, to=1-2]
	\arrow[tail,"\phi", from=1-2, to=1-3]
	\arrow[two heads,from=1-3, to=1-4]
	\arrow[from=1-4, to=1-5]
\end{tikzcd}\] with $\Theta_2,\Theta_3 \in \add(T)$. Given the short exact sequence and the middle column of the diagram above, we obtain the following diagram by taking the pushout along $\phi$ and applying the cokernel lifting property again.
\[\begin{tikzcd}
         & 0 & 0 & 0 &  \\
	0 & P_0 & \Theta_2  & \Theta_3 &  0 \\
     0 & W & PO' & C' & 0 \\
      0 & \Theta_1 & \Theta_1 & 0 & 0 \\
       & 0 & 0 & 0 & 
	\arrow[from=2-1, to=2-2]
	\arrow[tail,"\phi", from=2-2, to=2-3]
	\arrow[two heads,from=2-3, to=2-4]
	\arrow[from=2-4, to=2-5]
        \arrow[from=3-1, to=3-2]
	\arrow[tail, from=3-2, to=3-3]
	\arrow[two heads,from=3-3, to=3-4]
	\arrow[from=3-4, to=3-5]
        \arrow[from=4-1, to=4-2]
	\arrow[equals, from=4-2, to=4-3]
	\arrow[two heads,from=4-3, to=4-4]
	\arrow[from=4-4, to=4-5]
        \arrow[from=4-2, to=5-2]
	\arrow[two heads, from=3-2, to=4-2]
	\arrow[tail,from=2-2, to=3-2]
	\arrow[from=1-2, to=2-2]
        \arrow[from=4-3, to=5-3]
	\arrow[two heads, from=3-3, to=4-3]
	\arrow[tail,from=2-3, to=3-3]
	\arrow[from=1-3, to=2-3]
        \arrow[from=4-4, to=5-4]
	\arrow[two heads, from=3-4, to=4-4]
	\arrow[from=2-4, to=3-4]
	\arrow[from=1-4, to=2-4]
\end{tikzcd} \]
As the short sequences given by the three rows and those given by the first two columns of the above diagram are exact, by \cref{lem:3x3}, the short sequence given by the last column is exact too. So $C' \cong \Theta_3$. Since $T$ is tilting, the middle column splits, hence $PO' \cong \Theta_1 \oplus \Theta_2$.

We get the following short exact sequence.
\[\begin{tikzcd}
    \varsigma: & 0 & W & \Theta_1 \oplus \Theta_2 & \Theta_3 &  0 & 
	\arrow[from=1-2, to=1-3]
	\arrow[tail, from=1-3, to=1-4]
	\arrow[two heads,from=1-4, to=1-5]
	\arrow[from=1-5, to=1-6]
\end{tikzcd}\]
As $W,\Theta_3 \in \mathscr{C}$, as $\mathscr{C}$ is $\mathcal{E}$-adapted, then $\varsigma \in \mathcal{E}$ and $W \in \Sub_\mathcal{E}(\Theta_1 \oplus \Theta_2) \subseteq \GS_\mathcal{E}(T)$.

We finally obtain the following short exact sequence from the middle row of the first diagram
\[\begin{tikzcd}
    \mu: & 0 & T_0 & W & V &  0 & 
	\arrow[from=1-2, to=1-3]
	\arrow[tail, from=1-3, to=1-4]
	\arrow[two heads,from=1-4, to=1-5]
	\arrow[from=1-5, to=1-6]
\end{tikzcd}\]
As $\mu \in \mathcal{E}$ and $W \in \GS_\mathcal{E}(T)$, we get that $V \in \Gen_\mathcal{E}(W) \subseteq \GS_\mathcal{E}(T)$. 

So all the objects in $\mathscr{C}$ are in $\GS_\mathcal{E}(T)$, and therefore, we get the desired result.
\end{proof}

Following the proof, the corollary below occurs.

\begin{cor} \label{cor:GS2=GSanduniqueESerre}
   Let $T \in \Tilt(\mathscr{A})$ be a tilting object. The following assertions hold:
   \begin{enumerate}[label=$(\alph*)$, itemsep=1mm]
       \item we have $\GS_\mathcal{E}(T) = \GS_\mathcal{E}^2(T)$, and moreover, if $T$ is an $\mathcal{E}$-projective object, then $\GS_\mathcal{E}(T) = \GS_\mathcal{E}^1(T)$; and,
       \item $\GS_\mathcal{E}(T)$ is the unique $\mathcal{E}$-adapted $\mathcal{E}$-Serre subcategory of $\mathscr{A}$ containing $T$.
   \end{enumerate}
\end{cor}

\begin{proof}
    The assertion $(a)$ follows directly from the proof of \cref{thm:maximal}. The assertion $(b)$ is a direct consequence of both \cref{thm:maximal} and \cref{cor:ESerre} $(d)$.
\end{proof}

Note that, in general, a maximal $\mathcal{E}$-adapted subcategory is not extension-closed.

\begin{example} \label{ex:dontworkforanyE}
Consider $\mathscr{A} = \rep Q$ with
$Q = 1 \longrightarrow 2 \longrightarrow 3$. Fix the exact structure $\mathcal{E}'$ generated by the two following Auslander--Reiten sequences.
\[\begin{tikzcd}
    \xi_1: & 0 &  X_{\llrr{3}} & X_{\llrr{2,3}} & X_{\llrr{2}} & 0 
	\arrow[from=1-2, to=1-3]
	\arrow[tail, from=1-3, to=1-4]
	\arrow[two heads,from=1-4, to=1-5]
	\arrow[from=1-5, to=1-6]
\end{tikzcd}\]
\[\begin{tikzcd}
    \xi_2: & 0 &  X_{\llrr{2}} & X_{\llrr{1,2}} & X_{\llrr{1}} & 0 
	\arrow[from=1-2, to=1-3]
	\arrow[tail, from=1-3, to=1-4]
	\arrow[two heads,from=1-4, to=1-5]
	\arrow[from=1-5, to=1-6]
\end{tikzcd}\]

Let $\mathscr{C} = \add(X_{\llrr{1}}, X_{\llrr{2}}, X_{\llrr{3}}, X_{\llrr{1;3}})$ (see \cref{fig:NonExt-closed2}). We can check that $\mathscr{C}$ is a maximal $\mathcal{E}'$-adapted subcategory of $\mathscr{A}$, but $\mathscr{C}$ is not closed under extension.

\begin{figure}[!ht]
    \centering
    \begin{tikzpicture}[line width=.3mm, ->, >= angle 60]
		\node (23) at (1,-1){\scalebox{.9}{$\llrr{2,3}$}};
		\node[circle,line width=0.2mm,double,fill=lava!20,draw] (3) at (0,-2){\scalebox{.9}{$\llrr{3}$}};
		\node[circle,line width=0.2mm,double,fill=lava!20,draw] (13) at (2,0){\scalebox{.9}{$\llrr{1,3}$}};
		\node[circle,line width=0.2mm,double,fill=lava!20,draw] (2) at (2,-2){\scalebox{.9}{$\llrr{2}$}};
		\node (12) at (3,-1){\scalebox{.9}{$\llrr{1,2}$}};
		\node[circle,line width=0.2mm,double,fill=lava!20,draw] (1) at (4,-2){\scalebox{.9}{$\llrr{1}$}};
								
		\draw (3) -- (23);
		\draw (23) -- (13);
		\draw (23) -- (2);
		\draw (13) -- (12);
		\draw (2) -- (12);
		\draw (12) -- (1);
        \end{tikzpicture}
    \caption[fragile]{The Auslander--Reiten quiver of $Q$ in \cref{ex:Eadaptednonextclosed}. We have $\mathscr{C} = \add \left( \protect\begin{tikzpicture}[baseline={(0,-.1)}]
        \node[circle, line width=0.2mm, double, fill=lava!20,minimum size=1em,draw] at (0,0){$ $};
    \end{tikzpicture}\right) = \GS_{\mathcal{E}'}(\mathscr{C})$, and we can check that $\mathscr{C}$ is not closed under extensions even though it is a maximal $\mathcal{E}'$-adapted subcategory of $\mathscr{A}$.}
    \label{fig:NonExt-closed2}
\end{figure}           
\end{example}

\section{Jordan recoverability}
\label{sec:JR}
\pagestyle{plain}

Let $(Q,R)$ be a bounded quiver. In this section, we recall the notion of Jordan recoverability on subcategories of $\rep(Q,R)$, and a few results relative to this subject. We refer the reader to \cite{GPT19,D22,D23}. 

\subsection{Jordan recoverability}
\label{ss:JR}

Let $E \in \rep(Q,R)$. A \new{nilpotent endomorphism} of $E$ is a morphism $N : E \longrightarrow E$  such that there exists a $k \in \mathbb{N}^*$ such that $N^k = 0$, or equivalently, for all $q \in Q_0$, the linear map $N_q$ is nilpotent. We denote the set of nilpotent endomorphisms of $E$ by $\NEnd(E)$. \cite{GPT19} shows that $\NEnd(E)$ is an irreducible algebraic variety.

An \new{integer partition} is a finite weakly decreasing sequence $\lambda = (\lambda_1,\ldots,\lambda_k)$ of positive integers. We also consider, by convention, that $(0)$ is an integer partition called \emph{zero partition}. The \new{parts} of $\lambda$ are the values, counted with multiplicities, taken by $\lambda$. The \new{size} of $\lambda$, denoted by $|\lambda|$, is the sum of all its parts; meaning $|\lambda| = \lambda_1 + \ldots + \lambda_k$. More precisely, if $|\lambda| = m$ for some $m \in \mathbb{N}$, we say that $\lambda$ is an integer partition of $m$, and we write $\lambda \vdash m$. The \new{length} of $\lambda$, denoted by $\ell(\lambda)$, is the number of parts of $\lambda$; meaning $\ell(\lambda) = k$. Let $m \in \mathbb{N}$. We define the \new{dominance order} $\leqslant$ on integer partitions of $m$ as follows: for any $\lambda, \mu \vdash m$, we say that $\lambda \leqslant \mu$ whenever for all $1 \leqslant i \leqslant \min(\ell(\lambda), \ell(\mu))$, we have $\lambda_1 + \ldots + \lambda_i \leqslant \mu_1 + \ldots + \mu_i$. 

Let $p \in \mathbb{N}^*$. We extend some notions to $p$-tuples of integer partitions, which are called \new{p-integer partitions}. Let $\pmb{\lambda} = (\lambda^1, \ldots, \lambda^p)$ be a $p$-integer partition. The \new{size vector} of $\pmb{\lambda}$ is $|\pmb{\lambda}| = (|\lambda^1|, \ldots, |\lambda^p|)$. Given $\pmb{d} = (d_1,\ldots,d_p) \in \mathbb{N}^p$, we say that $\pmb{\lambda}$ is a $p$-integer partition of $\pmb{d}$ whenever $|\pmb{\lambda}| = \pmb{d}$. In such case, we write $\pmb{\lambda} \pmb{\vdash} \pmb{d}$. We define the \emph{dominance order} on $p$-integer partitions of $\pmb{d}$ as follows: for any $\pmb{\lambda},\pmb{\mu} \pmb{\vdash} \pmb{d}$, we say that $\pmb{\lambda} \domleq \pmb{\mu}$ if, for all $j \in \{1,\ldots,p\}$, we have $\lambda^j \leqslant \mu^j$.

Let $N \in \NEnd(E)$. For any $q \in Q_0$, we denote $\JF(N_q)$ the \emph{Jordan form} of the nilpotent linear map $N_q$, understood as an integer partition whose parts are the sizes of the Jordan blocks. So $\JF(N_q) \vdash d_q$ for any $q \in Q_0$. Write $\vJF(N)$ for the \new{Jordan form} of $N$, defined as the $\#Q_0$-integer partitions $(\JF(N_q))_{q \in Q_0}$. Note that $\vJF(N) \pmb{\vdash} \vdim(E)$.

\begin{theorem}[\cite{GPT19}] \label{thm:GenJFexists} Let $(Q,R)$ be a bounded quiver, $\mathbb{K}$ be an algebraically closed field, and $E \in \rep(Q,R)$. Then  $\vJF$ admits a maximum value on $\NEnd(E)$, and it is attained in a dense open (for Zariski topology) subset of $\NEnd(E)$.
\end{theorem}

We define $\GenJF(E)$ as the maximal value of $\vJF$ on $\NEnd(E)$. We call it \new{generic Jordan form} of $E$. Note that $\GenJF$ defines an invariant on isomorphism classes of representations of $(Q,R)$. Still, it is not a complete invariant in general: that is, we can find $E,F \in \rep(Q,R)$ such that $E \ncong F$ and $\GenJF(E) = \GenJF(F)$. However, we can focus on the subcategories of $\rep(Q,R)$ such that $\GenJF$ is complete.

\begin{definition} \label{def:JR}
    Let $(Q,R)$ be a bounded quiver. A subcategory $\mathscr{C}$ or $\rep(Q,R)$ is said to be \new{Jordan recoverable} if, for any $\#Q_0$-integer partition $\pmb{\lambda}$, there exists at most one $E \in \rep(Q,R)$, up to isomorphism, such that $\GenJF(E) = \pmb{\lambda}$.
\end{definition}

To determine whether a subcategory of $\rep(Q,R)$ is Jordan recoverable, we are led to solve a system of tropical equations, which can be a challenging task. The following section introduces an algebraic way to build a representation from any $\#Q_0$-integer partition. Under some more restrictions, it allows one to return a representation of $(Q,R)$ from its generic Jordan form.

\subsection{Canonical Jordan recoverability}
\label{ss:CJR}

Let $\pmb{\lambda}$ be a $\#Q_0$-integer partition, and set $|\pmb{\lambda}| = \pmb{d}$. Consider $N = \left(N_q : \mathbb{K}^{d_q} \longrightarrow \mathbb{K}^{d_q} \right)_{q \in Q_0}$ to be a $\#Q_0$-tuple of nilpotent linear maps such that $\JF(N_q) = \lambda^q$. We say that a representation $E$ is \new{compatible with} $N$ whenever there exists $N' \in \NEnd(E)$ such that, for all $q \in Q_0$, the linear maps $N'_q$ and $N_q$ are similar. We denote by $\rep((Q,R),N)$ the collection of representations compatible with $N$.

In $\rep((Q,R),N)$, we ask whether there exists a dense open set (for Zariski topology) such that all its representations are isomorphic. In such a case, we define $\GenRep(N)$ as a representation in the corresponding dense open set. Note that $\GenRep(N)$ only depends on its Jordan form. Therefore, we define $\GenRep(\pmb{\lambda})$ the \new{generic representation} of $\pmb{\lambda}$ to be one of the $\GenRep(N)$ where $N$ is a collection of nilpotent linear maps compatible with a nilpotent endomorphism $N' \in \NEnd(E)$ such that $\JF(N') = \pmb{\lambda}$.

\begin{definition} \label{def:CJR} A subcategory $\mathscr{C} \subseteq \rep(Q,R)$ is said to be \new{canonically Jordan recoverable} if, for any $X \in \mathscr{C}$, $\GenRep(\GenJF(X))$ is well-defined and is isomorphic to $X$.
\end{definition}

Note that, in general, given a $\#Q_0$-integer partition, it could happen that $\GenRep(\pmb{\lambda})$ cannot be defined. We refer to \cite[Examples 2.21 and 2.23]{D22} for examples in the case where $(Q,R)$ is a gentle quiver.

The following result shows that, at least, under some restrictions, we can ensure that $\GenRep(\pmb{\lambda})$ is well-defined for any $\#Q_0$-integer partition.

\begin{prop}[\cite{GPT19,D23}] \label{prop:GenRepDynkin}
    Let $Q$ be an $ ADE$-type Dynkin quiver. Then, for any $\#Q_0$-integer partition $\pmb{\lambda}$, $\GenRep(\pmb{\lambda})$ is well-defined.
\end{prop}

Now, we restrict ourselves to type $A_n$ quivers for some $n \in \mathbb{N}^*$, and we recall useful results. Fix $n \in \mathbb{N}^*$. Let us first recall the combinatorial statement that characterizes all the canonically Jordan recoverable subcategories of type $A_n$.

Two intervals $K,L \in \mathcal{I}_n$ are said to be \new{adjacent} if either $b(K) = e(L)+1$ or $b(L) = e(K)+1$. Note that $K$ and $L$ are not adjacent if $K \cap L \neq \varnothing$, $b(K) > e(L)+1$ or $b(L) > e(K)+1$. We say that an interval subset $\mathscr{J} \subseteq \mathcal{I}_n$ is \new{adjacency-avoiding} if, for all $(K,L) \in \mathscr{J}^2$, the intervals $K$ and $L$ are not adjacent.

\begin{theorem} \label{thm:typeACJRcombi}Let $Q$ be an $A_n$ type quiver. A subcategory $\mathscr{C} \subset \rep(Q)$ is canonically Jordan recoverable if and only if $\Int(\mathscr{C})$ is adjacency-avoiding.
\end{theorem}

Determining all the maximal adjacency-avoiding subsets for the inclusion of $\mathcal{I}_n$ was the primary key to proving this result. Let us recall the exact statement. Given $\B, \E \subseteq \{1,\ldots n\}$, we define the interval subset \[\opJ(\B,\E) = \{K \in \mathcal{I}_n \mid b(K) \in \B,\ e(K) \in \E\}.\] Given $\A \subset\{1,\ldots,n\}$, we set $\A[1] = \{a+1 \mid a \in \A\}$.

\begin{prop} \label{prop:allmaxadjavo} The following assertions hold:
    \begin{enumerate}[label = $(\roman*)$,itemsep=1mm]
        \item For all $\B,\E \subseteq \{1,\ldots,n\}$, if $(\B,\E[1])$ is a set partition of $\{1,\ldots,n+1\}$, then $\opJ(\B,\E)$ is a maximal adjacency-avoiding subset of $\mathcal{I}_n$; and,
        \item Any maximal adjacency-avoiding subset of $\mathcal{I}$ is equal to $\opJ(\B,\E)$ for some $\B,\E \subseteq \{1,\ldots,n\}$ such that $(\B,\E[1])$ is a set partition of $\{1,\ldots,n+1\}$.
    \end{enumerate}
\end{prop}

Another relevant key to proving \cref{thm:typeACJRcombi} is to show that the notion of canonical Jordan recoverability on type $A$ quivers does not depend on the orientation of the quiver. This, indeed, is a consequence of the techniques we will not detail here.

Define the subcategory $\opC(\B,\E) = \Cat(\opJ(\B,\E))$ for any $\B,\E \subseteq \{1,\ldots,n\}$. We can state the following result. 

\begin{theorem} \label{thm:CJRmaxcomb}
    Let $Q$ be an $A_n$ type quiver. A subcategory $\mathscr{C} \subseteq \rep(Q)$ is a maximal canonically Jordan recoverable subcategory if and only if $\mathscr{C} = \opC(\B,\E)$ for some $\B,\E \subseteq \{1,\ldots,n\}$ such that $(\B,\E[1])$ is a set partition of $\{1,\ldots,n+1\}$.
\end{theorem}

\begin{ex} \label{ex:A7typeMCJR} By taking back the $A_7$ type quiver $Q$ in \cref{ex:A7typeGS}, we explicit, in \cref{fig:MaxCJR1}, a maximal canonically Jordan recoverable subcategory of $\rep(Q)$ given by the pair $(\B,\E)$ with $\B = \{1,2,3,5\}$ and $\E = \{3,5,6,7\}$.
    \begin{figure}[!ht]
    \leavevmode\\[-0.5ex]  % forces heading to appear before content
    \centering
    \begin{tikzpicture}
                \begin{scope}[line width=.3mm, ->, >= angle 60]
					\node (2) at (0,0){\scalebox{.7}{$\llrr{2}$}};
					\node (12) at (1,1){\scalebox{.7}{$\llrr{1,2}$}};
					\node[rectangle,line width=0.2mm,fill=Plum!20,draw] (2345) at (1,-1){\scalebox{.7}{$\llrr{2,5}$}};
					\node (45) at (0,-2){\scalebox{.7}{$\llrr{4,5}$}};
                    \node[rectangle,line width=0.2mm,fill=Plum!20,draw] (5) at (-1,-3){\scalebox{.7}{$\llrr{5}$}};
                    \node[rectangle,line width=0.2mm,fill=Plum!20,draw] (567) at (0,-4){\scalebox{.7}{$\llrr{5,7}$}};
                    \node (7) at (-1,-5){\scalebox{.7}{
							$\llrr{7}$}};
                    \node (56)[rectangle,line width=0.2mm,fill=Plum!20,draw] at (1,-5){\scalebox{.7}{$\llrr{5,6}$}
                                };
                     \node (4) at (3,-5){\scalebox{.7}{$\llrr{4}$}};
                     \node[rectangle,line width=0.2mm,fill=Plum!20,draw] (23) at (5,-5){\scalebox{.7}{$\llrr{2,3}$}};
                    \node (1) at (7,-5){\scalebox{.7}{$\llrr{1}$}};
                    \node (456) at (2,-4){\scalebox{.7}{$\llrr{4,6}$}};
                    \node (234) at (4,-4){\scalebox{.7}{$\llrr{2,4}$}};
                     \node[rectangle,line width=0.2mm,fill=Plum!20,draw] (123) at (6,-4){\scalebox{.7}{$\llrr{1,3}$}};
                     \node (4567) at (1,-3){\scalebox{.7}{$\llrr{4;7}$}};
                     \node[rectangle,line width=0.2mm,fill=Plum!20,draw] (23456) at (3,-3){\scalebox{.7}{$\llrr{2,6}$}};
                     \node (1234) at (5,-3){\scalebox{.7}{$\llrr{1,4}$}};
                     \node[rectangle,line width=0.2mm,fill=Plum!20,draw] (3) at (7,-3){\scalebox{.7}{$\llrr{3}$}};
					\node[rectangle,line width=0.2mm,fill=Plum!20,draw] (12345) at (2,0){\scalebox{.7}{$\llrr{1,5}$}};
					\node[rectangle,line width=0.2mm,fill=Plum!20,draw] (234567) at (2,-2){\scalebox{.7}{$\llrr{2,7}$}};
                    \node[rectangle,line width=0.2mm,fill=Plum!20,draw] (123456) at (4,-2){\scalebox{.7}{$\llrr{1,6}$}};
					\node (34) at (6,-2){\scalebox{.7}{$\llrr{3,4}$}};
					\node[rectangle,line width=0.2mm,fill=Plum!20,draw] (1234567) at (3,-1){\scalebox{.7}{$\llrr{1,7}$}};
                    \node[rectangle,line width=0.2mm,fill=Plum!20,draw] (3456) at (5,-1){\scalebox{.7}{$\llrr{3,6}$}};
                    \node[rectangle,line width=0.2mm,fill=Plum!20,draw] (34567) at (4,0){\scalebox{.7}{$\llrr{3,7}$}};
                    \node (6) at (6,0){\scalebox{.7}{$\llrr{6}$}};
                    \node[rectangle,line width=0.2mm,fill=Plum!20,draw] (345) at (3,1){\scalebox{.7}{$\llrr{3,5}$}};
                    \node (67) at (5,1){\scalebox{.7}{$\llrr{6,7}$}};
					\draw (2) -- (12);
					\draw (2) -- (2345);
					\draw (45) -- (2345);
					\draw (12) -- (12345);
					\draw (2345) -- (12345);
					\draw (2345) -- (234567);
                    \draw (12345) -- (345);

                    \draw (5) -- (45);
                    \draw (7) -- (567);
                    \draw (567) -- (4567);
                    \draw (4567) -- (234567);
                    \draw (234567) -- (1234567);
                    \draw (1234567) -- (34567);
                    \draw (34567) -- (67);
                    \draw (56) -- (456);
                    \draw (456) -- (23456);
                    \draw (23456) -- (123456);
                    \draw (123456) -- (3456);
                    \draw (3456) -- (6);
                    \draw (4) -- (234);
                    \draw (234) -- (1234);
                    \draw (1234) -- (34);
                    \draw (23) -- (123);
                    \draw (123) -- (3);
					
					\draw (5) -- (567);
                    \draw (567) -- (56);
                    \draw (45) -- (4567);
                    \draw (4567) -- (456);
                    \draw (456) -- (4);
                    \draw (234567) -- (23456);
                    \draw (23456) -- (234);
                    \draw (234) -- (23);
                    \draw (12345) -- (1234567);
                    \draw (1234567) -- (123456);
                    \draw (123456) -- (1234);
                    \draw (1234) -- (123);
                    \draw (123) -- (1);
                    \draw (345) -- (34567);
                    \draw (34567) -- (3456);
                    \draw (3456) -- (34);
                    \draw (34) -- (3);
                    \draw (67) -- (6);	
				\end{scope}
    \end{tikzpicture}
    \caption[fragile]{Example of a maximal canonically Jordan recoverable subcategory $\mathscr{C} = \add \left( \protect\begin{tikzpicture}[baseline={(0,-.1)}]
        \node[rectangle, line width=0.2mm, fill=Plum!20,minimum size=1em,draw] at (0,0){$ $};
    \end{tikzpicture}\right)$ in $\rep(Q)$, for $Q$ given in \cref{fig:GSCalc1}.}
    \label{fig:MaxCJR1}
\end{figure}
\end{ex}

There is also a representation-theoretic interpretation of \cref{thm:typeACJRcombi}, and it starts by noticing the following result, an obvious consequence from \cref{thm:exttypeA}.

\begin{lemma} \label{lem:adj=arrow} Let $Q$ be an $A_n$ type quiver. Then two intervals $K,L \in \mathcal{I}_n$ are adjacent if and only if one of the following assertions holds:
\begin{enumerate}[label=$\bullet$, itemsep=1mm]
    \item there exists a short exact sequence \[\begin{tikzcd}
	0 & X_K & E  & X_L&  0,
	\arrow[from=1-1, to=1-2]
	\arrow[tail, from=1-2, to=1-3]
	\arrow[two heads,from=1-3, to=1-4]
	\arrow[from=1-4, to=1-5]
\end{tikzcd} \] such that $E \in \Ind(Q)$; or,
    \item there exists a short exact sequence \[\begin{tikzcd}
	0 & X_L & E  & X_K&  0,
	\arrow[from=1-1, to=1-2]
	\arrow[tail, from=1-2, to=1-3]
	\arrow[two heads,from=1-3, to=1-4]
	\arrow[from=1-4, to=1-5]
\end{tikzcd} \] such that $E \in \Ind(Q)$.
\end{enumerate}
\end{lemma}

This allows us to state the following result.

 \begin{theorem} \label{thm:CJRreptheory}
     Let $Q$ be an $A_n$ type quiver. A subcategory $\mathscr{C} \subset \rep(Q)$ is canonically Jordan recoverable if and only if for any nonsplit short exact sequence \[\begin{tikzcd}
	0 & E & F  & G &  0,
	\arrow[from=1-1, to=1-2]
	\arrow[tail, from=1-2, to=1-3]
	\arrow[two heads,from=1-3, to=1-4]
	\arrow[from=1-4, to=1-5]
\end{tikzcd}\] we have $F \notin \Ind(Q)$ whenever $E,G \in \mathscr{C}$.
 \end{theorem}

\begin{remark}\label{rem:TwoThings} Two things:
\begin{enumerate}[label=$\bullet$,itemsep=1mm]
    \item One can restate the result by saying that a subcategory $\mathscr{C} \subset \rep(Q)$ is canonically Jordan recoverable if and only if all the short exact sequences whose extremities are in $\mathscr{C}$ must be in $\mathcal{E}_\diamond$.
    \item One can do a parallel with the definition of \emph{maximal almost rigid} representations \cite{BGMS23}. A representation $M \in \rep(Q)$ is said to be \new{maximal almost rigid} if for any nonsplit short exact sequence \[\begin{tikzcd}
	0 & E & F  & G &  0,
	\arrow[from=1-1, to=1-2]
	\arrow[tail, from=1-2, to=1-3]
	\arrow[two heads,from=1-3, to=1-4]
	\arrow[from=1-4, to=1-5]
\end{tikzcd}\] where $E$ and $G$ are indecomposable summands of $M$, we have $F \in \ind(Q)$.
\end{enumerate}
\end{remark}

In the following section, we exhibit a strong link between $\GS$ operators on $\rep(Q)$ for  $\mathcal{E}_\diamond$ and maximal canonically Jordan recoverable subcategories.

\section{Type ADE algebras case}
\label{sec:TypeA}
%\input{TypeA.tex}
%modifications en cours
\pagestyle{plain}

\begin{conv}\label{conv:typeA}
    We fix $Q$ an $ADE$ type quiver. We endow $\rep(Q)$ with an exact structure $\mathcal{E}$.
\end{conv}

Denote by $\Tilt(Q)$ the collection $\Tilt(\rep(Q))$.
Given any finite collection of additive full subcategories $\mathscr{C}_i \subseteq \rep Q$, we write $\bigoplus_i \mathscr{C}_i$ for the additive full subcategory of $\rep(Q)$ generated by all the subcategories $\mathscr{C}_i$.

In this section, our aim is to describe, for any $T \in \Tilt(Q)$, that the subcategory $\GS_\mathcal{E}(T)$ in terms of a full and additive subcategory generated by tilting objects that are related to $T$ via so-called \emph{$\mathcal{E}$-mutations}. We apply this description to canonically Jordan recoverable subcategories, by restricting ourselves to type $A$ quivers. 

\subsection{Tilting \texorpdfstring{$\mathcal{E}$}{TEXT}-mutations}
\label{ss:Tiltmut} In this subsection, we recall the notion of \emph{$\mathcal{E}$-mutations} on $\Tilt(Q)$, and we enumerate crucial and well-known results on left and right approximations, stated in our type $ADE$ quivers framework.

\begin{definition}[\cite{AS94}] \label{def:leftrightminapprox}
Let $\mathscr{C} \subseteq \rep(Q)$ be a subcategory, and $E \in \rep(Q)$. 
\begin{enumerate}[label=$\bullet$, itemsep=1mm]
    \item A \new{left $\mathscr{C}$-approximation} of $E$ is a morphism $\begin{tikzcd}
	f: E & C
	\arrow[from=1-1, to=1-2]
	\end{tikzcd}$ with $C \in \mathscr{C}$ such that every morphism $\begin{tikzcd}
	\varphi: E & C'
	\arrow[from=1-1, to=1-2]
	\end{tikzcd}$ where $C' \in \mathscr{C}$ factors through $f$.
    \[\begin{tikzcd}
	 E & & C' \\
     & C & 
	\arrow["\forall\varphi",from=1-1, to=1-3]
        \arrow["f"', from=1-1, to=2-2]
        \arrow["\exists \psi"', dashed, from=2-2,to=1-3]
	\end{tikzcd}\]
    \item Such a left $\mathscr{C}$-approximation of $E$ is said to be \new{minimal} if any map $\begin{tikzcd}
	\alpha: C & C
	\arrow[from=1-1, to=1-2]
	\end{tikzcd}$ such that $\alpha f = f$ is bijective.
    Moreover, if $f$ is an $\mathcal{E}$-monomorphism, it is a \new{left $(\mathscr{C},\mathcal{E})$-approximation} of $E$.
    \item A \new{right $\mathscr{C}$-approximation} of $E$ is a morphism $\begin{tikzcd}
	g: C & E
	\arrow[from=1-1, to=1-2]
	\end{tikzcd}$ with $C \in \mathscr{C}$ such that every morphism $\begin{tikzcd}
	\varphi: C' & E
	\arrow[from=1-1, to=1-2]
	\end{tikzcd}$ where $C' \in \mathscr{C}$ factors through $g$.
    \[\begin{tikzcd}
	 C' & & E \\
     & C & 
	\arrow["\forall\varphi",from=1-1, to=1-3]
        \arrow["\exists \psi"',dashed, from=1-1, to=2-2]
        \arrow["g"', from=2-2,to=1-3]
	\end{tikzcd}\]
    \item Such a right $\mathscr{C}$-approximation of $E$ is said to be \new{minimal} if any map $\begin{tikzcd}
	\beta: C & C
	\arrow[from=1-1, to=1-2]
	\end{tikzcd}$ such that $g \beta = g$ is bijective.
    Moreover, if $g$ is an $\mathcal{E}$-epimorphsim, it is a  \new{right $(\mathscr{C},\mathcal{E})$-approximation} of $E$. 
\end{enumerate}
\end{definition}

This first proposition states that, given a subcategory $\mathscr{C} \subset \rep(Q)$, from any $\mathscr{C}$-approximation, we can build a minimal one.

\begin{prop} \label{prop:Approxismonoepi}
    Let $\mathscr{C} \subseteq \rep(Q)$ be a subcategory. Let $E \in \rep(Q)$. The following assertions hold:
    \begin{enumerate}[label=$(\roman*)$, itemsep=1mm]
        \item if $E$ admits a left $\mathscr{C}$-approximation, then $E$ admits a minimal left $\mathscr{C}$-approximation.
        \item if $E$ admits a right $\mathscr{C}$-approximation, then $E$ admits a minimal right $\mathscr{C}$-approximation.
    \end{enumerate}
\end{prop}

The following statement asserts that, for objects chosen rightfully relative to $\mathscr{C}$, any $\mathscr{C}$-approximation is an $\mathcal{E}$-epimorphism or an $\mathcal{E}$-monomorphism.

\begin{prop} \label{prop:Approxismonoepi2}
Let $\mathscr{C} \subseteq \rep(Q)$ be a subcategory. Let $E \in \rep(Q)$. The following assertions hold:
    \begin{enumerate}[label=$(\roman*)$, itemsep=1mm]
        \item if $E \in \Gen_\mathcal{E}(\mathscr{C})$, then any minimal right $\mathscr{C}$-approximation $\begin{tikzcd}
	g: C & E
	\arrow[from=1-1, to=1-2]
        \end{tikzcd}$ is an $\mathcal{E}$-epimorphism.
        \item if $E \in \Sub_\mathcal{E}(\mathscr{C})$, then any minimal left $\mathscr{C}$-approximation $\begin{tikzcd}
	f: E & C
	\arrow[from=1-1, to=1-2]
        \end{tikzcd}$ is a $\mathcal{E}$-monomorphism.
    \end{enumerate}
\end{prop}

\begin{lemma} \label{lem:TiltandApprox} Let $M \in \rep (Q)$. The following assertions hold:
\begin{enumerate}[label=$(\roman*)$, itemsep=1mm]
    \item \label{GenET2} Any $E \in \Gen_\mathcal{E}(M)$ admits a right $(\add(M), \mathcal{E})$-approximation; and,
    \item \label{SubET2} Any $E \in \Sub_\mathcal{E}(M)$ admits a left $(\add(M),\mathcal{E})$-approximation. 
\end{enumerate}
\end{lemma}

\begin{remark}
\label{rem:addMcase}
    In particular, for any given $M \in \rep(Q)$, any object $N \in \rep(Q)$ admits a right and a left $\add(M)$-approximation.
\end{remark}

We can now define the notion of $\mathcal{E}$-mutations on tilting objects of $\rep(Q)$.

\begin{definition} \label{def:leftrightEmutation}
Let $T \in \Tilt(Q)$ and $U \in \ind(Q)$.
\begin{enumerate}[label=$\bullet$, itemsep=1mm]
    \item We say that $T$ admits a \new{left $\mathcal{E}$-mutation} at $U$ if $U$ is an indecomposable summand of $T$, and, by writing $T = \widetilde{T} \oplus U$, we have that $U$ admits a left minimal $(\add(\widetilde{T}),\mathcal{E})$-approximation $f$. In such a case, we define the \new{left $\mathcal{E}$-mutation} of $T$ at $U$ as $\pmb{\mu}_{U,\mathcal{E}}^{-}(T) = \widetilde{T} \oplus \Coker(f)$. 
    \item We say that $T$ admits a \new{right $\mathcal{E}$-mutation} at $U$ if $U$ is an indecomposable summand of $T$, and, by writing $T = \widetilde{T} \oplus U$, we have that $U$ admit a right minimal $(\add(\widetilde{T}),\mathcal{E})$-approximation $g$. In such a case, we define the \new{right $\mathcal{E}$-mutation} of $T$ at $U$ as $\pmb{\mu}_{U,\mathcal{E}}^{+}(T) = \widetilde{T} \oplus \Ker(g)$. 
\end{enumerate}
\end{definition}

The following proposition allows us to talk about $\mathcal{E}$-mutations.

\begin{prop} \label{prop:OneEmutation}
    Let $T \in \Tilt(Q)$ and $U \in \ind(Q)$. Then $T$ admits at most one left $\mathcal{E}$-mutation or at most one right $\mathcal{E}$-mutation at $U$. Moreover if $T$ admits a left $\mathcal{E}$-mutation at $U$, then $T$ does not admits a right $\mathcal{E}$-mutation at $U$.
\end{prop}

\begin{remark} \label{rem:RS}
C. Riedtmann and A. Schofield \cite{RS91} proved this result for $\mathcal{E} = \mathcal{E}_{\max}$, and thus this proposition follows. 
\end{remark}

\begin{definition} \label{def:Emut}
    Let $U \in \rep(Q)$. We define the $\mathcal{E}$-mutation operator on $\Tilt(Q)$ as it follows:
    \[\forall T \in \Tilt(Q),\ \pmb{\mu}_{U,\mathcal{E}}(T) = \begin{cases}
        \pmb{\mu}_{U,\mathcal{E}}^{-}(T) & \text{if } T \text{ admits a left } \mathcal{E}\text{-mutation at } U; \\
        \pmb{\mu}_{U,\mathcal{E}}^{+}(T) & \text{if } T \text{ admits a right } \mathcal{E}\text{-mutation at } U; \\
        T & \text{otherwise.}
    \end{cases}\]
\end{definition}

\begin{definition} \label{def:Ereachable}
Let $T,T' \in \Tilt(Q)$. We say that $T'$ is \new{$\mathcal{E}$-reachable from} $T$ if there exists a finite sequence $(U_1, \ldots, U_m)$ of indecomposable representations of $Q$, such that \[T' \cong \pmb{\mu}_{U_m,\mathcal{E}} \circ \cdots \circ \pmb{\mu}_{U_1,\mathcal{E}} (T).\] 
In the case where $T' \cong T$ or $\pmb{\mu}_{U_i,\mathcal{E}}= \pmb{\mu}_{U_i,\mathcal{E}}^{-}$ for all $i \in \{1,\ldots,m\}$, we say that $T'$ is \new{left $\mathcal E$-reachable} from $T$. Dually, if $T' \cong T$ or
$\pmb{\mu}_{U_i,\mathcal{E}}= \pmb{\mu}_{U_i,\mathcal{E}}^{+}$ for all $i \in \{1,\ldots,m\}$, we say that $T'$ is \new{right $\mathcal E$-reachable} from $T$.
\end{definition}

The relation of $\mathcal{E}$-reachability on $\Tilt(Q)$ defines an equivalence relation, as $\pmb{\mu}_{U,\mathcal{E}}$ is an involution by \cref{prop:OneEmutation}. We denote this relation by $\approx_\mathcal{E}$. In the following, given $T \in \Tilt(Q)$, we denote by $[T]_{\approx_\mathcal{E}}$ the equivalence class of tilting representation under $\approx_\mathcal{E}$.

In our setting, \cite{HU05} shows that the right (and the left) $\mathcal{E}$-reachability relation defines an order relation on $\Tilt(Q)$. In the following, for $T,T' \in \Tilt(Q)$, we write $T' \TleqE T$ whenever $T'$ is right $\mathcal{E}$-reachable from $T$. In the case where $\mathcal{E} = \mathcal{E}_{\max}$, we simply write $\Tleq$ this relation. We also write $T' \TldotE T$ whenever $T$ \emph{covers} $T'$, meaning that there exists $U \in \ind(Q)$ such that $\pmb{\mu}_{U,\mathcal{E}}^{+}(T) \cong T'$. 

\begin{example} \label{ex:TiltA4Ediamond} In \cref{fig:posettiltEdiamond}, we represent the poset $(\Tilt(Q), \Tleq)$ for $Q$ the $A_4$ type quiver seen \cref{ex:TiltA4}. We also gather together tilting representations with respect to the equivalence relation $\approx_{\mathcal{E}_\diamond}$.
     \begin{figure}[!ht]
    \centering
    \begin{tikzpicture}[line width=.5mm,scale=1.7]
                        \filldraw[lava, fill=lava,opacity=0.2, loosely dotted, rounded corners] (.5,-.5) to[bend right=60] (-1.8,1.8) to[bend right=60] cycle;
					\filldraw[lava, fill=lava,opacity=0.2, loosely dotted, rounded corners] (1.75,0.9) to (1.8,1.9) to[bend right=-10] (2.8,2.6) to[bend right=60] (2.8,3.5) to[bend right=20] (.7,2.75) to[bend right=20] (.75,0.9) to[bend right=20]  cycle;
					\filldraw[lava, fill=lava,opacity=0.2, loosely dotted, rounded corners] (0,3.6) to[bend right=20] (3,5.7) to[bend right=60] (2.5,6.75) to[bend right=0] (1.5,5.6) to[bend right=-10] (-.25,4.6) to[bend right=30]  cycle;
					\filldraw[lava, fill=lava,opacity=0.2, loosely dotted, rounded corners] (-1.8,5.6) to[bend right=50] (1.4,7.5) to[bend right=40] cycle;
					\draw[lava, fill=lava, opacity=0.2, loosely dotted] (-2.25,2.5) circle (.7cm);
					\draw[lava, fill=lava, opacity=0.2, loosely dotted] (-.75,2.5) circle (.65cm);
					\draw[lava, fill=lava, opacity=0.2, loosely dotted] (-3,4.75) circle (.7cm);
					\draw[lava, fill=lava, opacity=0.2, loosely dotted] (-1.25,4.75) circle (.65cm);
					\node (a) at (0,0){\scalebox{.325}{\begin{tikzpicture}[line width=.3mm, ->, >= angle 60]
								\node[rectangle,line width=0.2mm, double,rounded corners,fill=orange!20,draw] (2) at (0,0){\scalebox{.9}{$\llrr{2}$}};
								\node[rectangle,line width=0.2mm, double,rounded corners,fill=orange!20,draw] (12) at (1,1){\scalebox{.9}{$\llrr{1,2}$}};
								\node[rectangle,line width=0.2mm, double,rounded corners,fill=orange!20,draw] (24) at (1,-1){\scalebox{.9}{$\llrr{2,4}$}};
								\node[rectangle,line width=0.2mm, double,rounded corners,fill=orange!20,draw] (4) at (0,-2){\scalebox{.9}{$\llrr{4}$}};
								\node (14) at (2,0){\scalebox{.9}{$\llrr{1,4}$}};
								\node (23) at (2,-2){\scalebox{.9}{$\llrr{2,3}$}};
								\node (34) at (3,1){\scalebox{.9}{$\llrr{3,4}$}};
								\node (13) at (3,-1){\scalebox{.9}{$\llrr{1,3}$}};
								\node (3) at (4,0){\scalebox{.9}{$\llrr{3}$}};
								\node (1) at (4,-2){\scalebox{.9}{$\llrr{1}$}};
								\draw (2) -- (12);
								\draw (2) -- (24);
								\draw (4) -- (24);
								\draw (12) -- (14);
								\draw (24) -- (14);
								\draw (24) -- (23);
								\draw (14) -- (34);
								\draw (14) -- (13);
								\draw (23) -- (13);
								\draw (34) -- (3);
								\draw (13) -- (3);
								\draw (13) -- (1);
					\end{tikzpicture}}};
					\node (b) at (-1.25,1.25){\scalebox{.325}{\begin{tikzpicture}[line width=.3mm, ->, >= angle 60]
								\node (2) at (0,0){\scalebox{.9}{$\llrr{2}$}};
								\node[rectangle,line width=0.2mm, double,rounded corners,fill=orange!20,draw] (12) at (1,1){\scalebox{.9}{$\llrr{1,2}$}};
								\node[rectangle,line width=0.2mm, double,rounded corners,fill=orange!20,draw] (24) at (1,-1){\scalebox{.9}{$\llrr{2,4}$}};
								\node[rectangle,line width=0.2mm, double,rounded corners,fill=orange!20,draw] (4) at (0,-2){\scalebox{.9}{$\llrr{4}$}};
								\node[rectangle,line width=0.2mm, double,rounded corners,fill=orange!20,draw] (14) at (2,0){\scalebox{.9}{$\llrr{1,4}$}};
								\node (23) at (2,-2){\scalebox{.9}{$\llrr{2,3}$}};
								\node (34) at (3,1){\scalebox{.9}{$\llrr{3,4}$}};
								\node (13) at (3,-1){\scalebox{.9}{$\llrr{1,3}$}};
								\node (3) at (4,0){\scalebox{.9}{$\llrr{3}$}};
								\node (1) at (4,-2){\scalebox{.9}{$\llrr{1}$}};
								\draw (2) -- (12);
								\draw (2) -- (24);
								\draw (4) -- (24);
								\draw (12) -- (14);
								\draw (24) -- (14);
								\draw (24) -- (23);
								\draw (14) -- (34);
								\draw (14) -- (13);
								\draw (23) -- (13);
								\draw (34) -- (3);
								\draw (13) -- (3);
								\draw (13) -- (1);
					\end{tikzpicture}}};
					\node (c) at (1.25,1.25){\scalebox{.325}{\begin{tikzpicture}[line width=.3mm, ->, >= angle 60]
								\node[rectangle,line width=0.2mm, double,rounded corners,fill=orange!20,draw] (2) at (0,0){\scalebox{.9}{$\llrr{2}$}};
								\node[rectangle,line width=0.2mm, double,rounded corners,fill=orange!20,draw] (12) at (1,1){\scalebox{.9}{$\llrr{1,2}$}};
								\node[rectangle,line width=0.2mm, double,rounded corners,fill=orange!20,draw] (24) at (1,-1){\scalebox{.9}{$\llrr{2,4}$}};
								\node (4) at (0,-2){\scalebox{.9}{$\llrr{4}$}};
								\node (14) at (2,0){\scalebox{.9}{$\llrr{1,4}$}};
								\node[rectangle,line width=0.2mm, double,rounded corners,fill=orange!20,draw] (23) at (2,-2){\scalebox{.9}{$\llrr{2,3}$}};
								\node (34) at (3,1){\scalebox{.9}{$\llrr{3,4}$}};
								\node (13) at (3,-1){\scalebox{.9}{$\llrr{1,3}$}};
								\node (3) at (4,0){\scalebox{.9}{$\llrr{3}$}};
								\node (1) at (4,-2){\scalebox{.9}{$\llrr{1}$}};
								\draw (2) -- (12);
								\draw (2) -- (24);
								\draw (4) -- (24);
								\draw (12) -- (14);
								\draw (24) -- (14);
								\draw (24) -- (23);
								\draw (14) -- (34);
								\draw (14) -- (13);
								\draw (23) -- (13);
								\draw (34) -- (3);
								\draw (13) -- (3);
								\draw (13) -- (1);
					\end{tikzpicture}}};
					\node (d) at (1.25,2.5){\scalebox{.325}{\begin{tikzpicture}[line width=.3mm, ->, >= angle 60]
								\node (2) at (0,0){\scalebox{.9}{$\llrr{2}$}};
								\node[rectangle,line width=0.2mm, double,rounded corners,fill=orange!20,draw] (12) at (1,1){\scalebox{.9}{$\llrr{1,2}$}};
								\node[rectangle,line width=0.2mm, double,rounded corners,fill=orange!20,draw] (24) at (1,-1){\scalebox{.9}{$\llrr{2,4}$}};
								\node (4) at (0,-2){\scalebox{.9}{$\llrr{4}$}};
								\node[rectangle,line width=0.2mm, double,rounded corners,fill=orange!20,draw] (14) at (2,0){\scalebox{.9}{$\llrr{1,4}$}};
								\node[rectangle,line width=0.2mm, double,rounded corners,fill=orange!20,draw] (23) at (2,-2){\scalebox{.9}{$\llrr{2,3}$}};
								\node (34) at (3,1){\scalebox{.9}{$\llrr{3,4}$}};
								\node (13) at (3,-1){\scalebox{.9}{$\llrr{1,3}$}};
								\node (3) at (4,0){\scalebox{.9}{$\llrr{3}$}};
								\node (1) at (4,-2){\scalebox{.9}{$\llrr{1}$}};
								\draw (2) -- (12);
								\draw (2) -- (24);
								\draw (4) -- (24);
								\draw (12) -- (14);
								\draw (24) -- (14);
								\draw (24) -- (23);
								\draw (14) -- (34);
								\draw (14) -- (13);
								\draw (23) -- (13);
								\draw (34) -- (3);
								\draw (13) -- (3);
								\draw (13) -- (1);
					\end{tikzpicture}}};
					\node (e) at (-.75,2.5){\scalebox{.325}{\begin{tikzpicture}[line width=.3mm, ->, >= angle 60]
								\node (2) at (0,0){\scalebox{.9}{$\llrr{2}$}};
								\node (12) at (1,1){\scalebox{.9}{$\llrr{1,2}$}};
								\node[rectangle,line width=0.2mm, double,rounded corners,fill=orange!20,draw] (24) at (1,-1){\scalebox{.9}{$\llrr{2,4}$}};
								\node[rectangle,line width=0.2mm, double,rounded corners,fill=orange!20,draw] (4) at (0,-2){\scalebox{.9}{$\llrr{4}$}};
								\node[rectangle,line width=0.2mm, double,rounded corners,fill=orange!20,draw] (14) at (2,0){\scalebox{.9}{$\llrr{1,4}$}};
								\node (23) at (2,-2){\scalebox{.9}{$\llrr{2,3}$}};
								\node[rectangle,line width=0.2mm, double,rounded corners,fill=orange!20,draw] (34) at (3,1){\scalebox{.9}{$\llrr{3,4}$}};
								\node (13) at (3,-1){\scalebox{.9}{$\llrr{1,3}$}};
								\node (3) at (4,0){\scalebox{.9}{$\llrr{3}$}};
								\node (1) at (4,-2){\scalebox{.9}{$\llrr{1}$}};
								\draw (2) -- (12);
								\draw (2) -- (24);
								\draw (4) -- (24);
								\draw (12) -- (14);
								\draw (24) -- (14);
								\draw (24) -- (23);
								\draw (14) -- (34);
								\draw (14) -- (13);
								\draw (23) -- (13);
								\draw (34) -- (3);
								\draw (13) -- (3);
								\draw (13) -- (1);
					\end{tikzpicture}}};
					\node (f) at (-2.25,2.5){\scalebox{.325}{\begin{tikzpicture}[line width=.3mm, ->, >= angle 60]
								\node (2) at (0,0){\scalebox{.9}{$\llrr{2}$}};
								\node[rectangle,line width=0.2mm, double,rounded corners,fill=orange!20,draw] (12) at (1,1){\scalebox{.9}{$\llrr{1,2}$}};
								\node (24) at (1,-1){\scalebox{.9}{$\llrr{2,4}$}};
								\node[rectangle,line width=0.2mm, double,rounded corners,fill=orange!20,draw] (4) at (0,-2){\scalebox{.9}{$\llrr{4}$}};
								\node[rectangle,line width=0.2mm, double,rounded corners,fill=orange!20,draw] (14) at (2,0){\scalebox{.9}{$\llrr{1,4}$}};
								\node (23) at (2,-2){\scalebox{.9}{$\llrr{2,3}$}};
								\node (34) at (3,1){\scalebox{.9}{$\llrr{3,4}$}};
								\node (13) at (3,-1){\scalebox{.9}{$\llrr{1,3}$}};
								\node (3) at (4,0){\scalebox{.9}{$\llrr{3}$}};
								\node[rectangle,line width=0.2mm, double,rounded corners,fill=orange!20,draw] (1) at (4,-2){\scalebox{.9}{$\llrr{1}$}};
								\draw (2) -- (12);
								\draw (2) -- (24);
								\draw (4) -- (24);
								\draw (12) -- (14);
								\draw (24) -- (14);
								\draw (24) -- (23);
								\draw (14) -- (34);
								\draw (14) -- (13);
								\draw (23) -- (13);
								\draw (34) -- (3);
								\draw (13) -- (3);
								\draw (13) -- (1);
					\end{tikzpicture}}};
					\node (g) at (2.5,3.){\scalebox{.325}{\begin{tikzpicture}[line width=.3mm, ->, >= angle 60]
								\node (2) at (0,0){\scalebox{.9}{$\llrr{2}$}};
								\node[rectangle,line width=0.2mm, double,rounded corners,fill=orange!20,draw] (12) at (1,1){\scalebox{.9}{$\llrr{1,2}$}};
								\node (24) at (1,-1){\scalebox{.9}{$\llrr{2,4}$}};
								\node (4) at (0,-2){\scalebox{.9}{$\llrr{4}$}};
								\node[rectangle,line width=0.2mm, double,rounded corners,fill=orange!20,draw] (14) at (2,0){\scalebox{.9}{$\llrr{1,4}$}};
								\node[rectangle,line width=0.2mm, double,rounded corners,fill=orange!20,draw] (23) at (2,-2){\scalebox{.9}{$\llrr{2,3}$}};
								\node (34) at (3,1){\scalebox{.9}{$\llrr{3,4}$}};
								\node[rectangle,line width=0.2mm, double,rounded corners,fill=orange!20,draw] (13) at (3,-1){\scalebox{.9}{$\llrr{1,3}$}};
								\node (3) at (4,0){\scalebox{.9}{$\llrr{3}$}};
								\node (1) at (4,-2){\scalebox{.9}{$\llrr{1}$}};
								\draw (2) -- (12);
								\draw (2) -- (24);
								\draw (4) -- (24);
								\draw (12) -- (14);
								\draw (24) -- (14);
								\draw (24) -- (23);
								\draw (14) -- (34);
								\draw (14) -- (13);
								\draw (23) -- (13);
								\draw (34) -- (3);
								\draw (13) -- (3);
								\draw (13) -- (1);
					\end{tikzpicture}}};
					\node (h) at (.25,4.25){\scalebox{.325}{\begin{tikzpicture}[line width=.3mm, ->, >= angle 60]
								\node (2) at (0,0){\scalebox{.9}{$\llrr{2}$}};
								\node (12) at (1,1){\scalebox{.9}{$\llrr{1,2}$}};
								\node[rectangle,line width=0.2mm, double,rounded corners,fill=orange!20,draw] (24) at (1,-1){\scalebox{.9}{$\llrr{2,4}$}};
								\node (4) at (0,-2){\scalebox{.9}{$\llrr{4}$}};
								\node[rectangle,line width=0.2mm, double,rounded corners,fill=orange!20,draw] (14) at (2,0){\scalebox{.9}{$\llrr{1,4}$}};
								\node[rectangle,line width=0.2mm, double,rounded corners,fill=orange!20,draw] (23) at (2,-2){\scalebox{.9}{$\llrr{2,3}$}};
								\node[rectangle,line width=0.2mm, double,rounded corners,fill=orange!20,draw] (34) at (3,1){\scalebox{.9}{$\llrr{3,4}$}};
								\node (13) at (3,-1){\scalebox{.9}{$\llrr{1,3}$}};
								\node (3) at (4,0){\scalebox{.9}{$\llrr{3}$}};
								\node (1) at (4,-2){\scalebox{.9}{$\llrr{1}$}};
								\draw (2) -- (12);
								\draw (2) -- (24);
								\draw (4) -- (24);
								\draw (12) -- (14);
								\draw (24) -- (14);
								\draw (24) -- (23);
								\draw (14) -- (34);
								\draw (14) -- (13);
								\draw (23) -- (13);
								\draw (34) -- (3);
								\draw (13) -- (3);
								\draw (13) -- (1);
					\end{tikzpicture}}};
					\node (i) at (-1.25,4.75){\scalebox{.325}{\begin{tikzpicture}[line width=.3mm, ->, >= angle 60]
								\node (2) at (0,0){\scalebox{.9}{$\llrr{2}$}};
								\node (12) at (1,1){\scalebox{.9}{$\llrr{1,2}$}};
								\node (24) at (1,-1){\scalebox{.9}{$\llrr{2,4}$}};
								\node[rectangle,line width=0.2mm, double,rounded corners,fill=orange!20,draw] (4) at (0,-2){\scalebox{.9}{$\llrr{4}$}};
								\node[rectangle,line width=0.2mm, double,rounded corners,fill=orange!20,draw] (14) at (2,0){\scalebox{.9}{$\llrr{1,4}$}};
								\node (23) at (2,-2){\scalebox{.9}{$\llrr{2,3}$}};
								\node[rectangle,line width=0.2mm, double,rounded corners,fill=orange!20,draw] (34) at (3,1){\scalebox{.9}{$\llrr{3,4}$}};
								\node (13) at (3,-1){\scalebox{.9}{$\llrr{1,3}$}};
								\node (3) at (4,0){\scalebox{.9}{$\llrr{3}$}};
								\node[rectangle,line width=0.2mm, double,rounded corners,fill=orange!20,draw] (1) at (4,-2){\scalebox{.9}{$\llrr{1}$}};
								\draw (2) -- (12);
								\draw (2) -- (24);
								\draw (4) -- (24);
								\draw (12) -- (14);
								\draw (24) -- (14);
								\draw (24) -- (23);
								\draw (14) -- (34);
								\draw (14) -- (13);
								\draw (23) -- (13);
								\draw (34) -- (3);
								\draw (13) -- (3);
								\draw (13) -- (1);
					\end{tikzpicture}}};
					\node (j) at (-3,4.75){\scalebox{.325}{\begin{tikzpicture}[line width=.3mm, ->, >= angle 60]
								\node (2) at (0,0){\scalebox{.9}{$\llrr{2}$}};
								\node[rectangle,line width=0.2mm, double,rounded corners,fill=orange!20,draw] (12) at (1,1){\scalebox{.9}{$\llrr{1,2}$}};
								\node (24) at (1,-1){\scalebox{.9}{$\llrr{2,4}$}};
								\node (4) at (0,-2){\scalebox{.9}{$\llrr{4}$}};
								\node[rectangle,line width=0.2mm, double,rounded corners,fill=orange!20,draw] (14) at (2,0){\scalebox{.9}{$\llrr{1,4}$}};
								\node (23) at (2,-2){\scalebox{.9}{$\llrr{2,3}$}};
								\node (34) at (3,1){\scalebox{.9}{$\llrr{3,4}$}};
								\node[rectangle,line width=0.2mm, double,rounded corners,fill=orange!20,draw] (13) at (3,-1){\scalebox{.9}{$\llrr{1,3}$}};
								\node (3) at (4,0){\scalebox{.9}{$\llrr{3}$}};
								\node[rectangle,line width=0.2mm, double,rounded corners,fill=orange!20,draw] (1) at (4,-2){\scalebox{.9}{$\llrr{1}$}};
								\draw (2) -- (12);
								\draw (2) -- (24);
								\draw (4) -- (24);
								\draw (12) -- (14);
								\draw (24) -- (14);
								\draw (24) -- (23);
								\draw (14) -- (34);
								\draw (14) -- (13);
								\draw (23) -- (13);
								\draw (34) -- (3);
								\draw (13) -- (3);
								\draw (13) -- (1);
					\end{tikzpicture}}};
					\node (k) at (1.5,5){\scalebox{.325}{\begin{tikzpicture}[line width=.3mm, ->, >= angle 60]
								\node (2) at (0,0){\scalebox{.9}{$\llrr{2}$}};
								\node (12) at (1,1){\scalebox{.9}{$\llrr{1,2}$}};
								\node (24) at (1,-1){\scalebox{.9}{$\llrr{2,4}$}};
								\node (4) at (0,-2){\scalebox{.9}{$\llrr{4}$}};
								\node[rectangle,line width=0.2mm, double,rounded corners,fill=orange!20,draw] (14) at (2,0){\scalebox{.9}{$\llrr{1,4}$}};
								\node[rectangle,line width=0.2mm, double,rounded corners,fill=orange!20,draw] (23) at (2,-2){\scalebox{.9}{$\llrr{2,3}$}};
								\node[rectangle,line width=0.2mm, double,rounded corners,fill=orange!20,draw] (34) at (3,1){\scalebox{.9}{$\llrr{3,4}$}};
								\node[rectangle,line width=0.2mm, double,rounded corners,fill=orange!20,draw] (13) at (3,-1){\scalebox{.9}{$\llrr{1,3}$}};
								\node (3) at (4,0){\scalebox{.9}{$\llrr{3}$}};
								\node (1) at (4,-2){\scalebox{.9}{$\llrr{1}$}};
								\draw (2) -- (12);
								\draw (2) -- (24);
								\draw (4) -- (24);
								\draw (12) -- (14);
								\draw (24) -- (14);
								\draw (24) -- (23);
								\draw (14) -- (34);
								\draw (14) -- (13);
								\draw (23) -- (13);
								\draw (34) -- (3);
								\draw (13) -- (3);
								\draw (13) -- (1);
					\end{tikzpicture}}};
					\node (l) at (-1,6){\scalebox{.325}{\begin{tikzpicture}[line width=.3mm, ->, >= angle 60]
								\node (2) at (0,0){\scalebox{.9}{$\llrr{2}$}};
								\node (12) at (1,1){\scalebox{.9}{$\llrr{1,2}$}};
								\node (24) at (1,-1){\scalebox{.9}{$\llrr{2,4}$}};
								\node (4) at (0,-2){\scalebox{.9}{$\llrr{4}$}};
								\node[rectangle,line width=0.2mm, double,rounded corners,fill=orange!20,draw] (14) at (2,0){\scalebox{.9}{$\llrr{1,4}$}};
								\node (23) at (2,-2){\scalebox{.9}{$\llrr{2,3}$}};
								\node[rectangle,line width=0.2mm, double,rounded corners,fill=orange!20,draw] (34) at (3,1){\scalebox{.9}{$\llrr{3,4}$}};
								\node[rectangle,line width=0.2mm, double,rounded corners,fill=orange!20,draw] (13) at (3,-1){\scalebox{.9}{$\llrr{1,3}$}};
								\node (3) at (4,0){\scalebox{.9}{$\llrr{3}$}};
								\node[rectangle,line width=0.2mm, double,rounded corners,fill=orange!20,draw] (1) at (4,-2){\scalebox{.9}{$\llrr{1}$}};
								\draw (2) -- (12);
								\draw (2) -- (24);
								\draw (4) -- (24);
								\draw (12) -- (14);
								\draw (24) -- (14);
								\draw (24) -- (23);
								\draw (14) -- (34);
								\draw (14) -- (13);
								\draw (23) -- (13);
								\draw (34) -- (3);
								\draw (13) -- (3);
								\draw (13) -- (1);
					\end{tikzpicture}}};
					\node (m) at (2.5,6){\scalebox{.325}{\begin{tikzpicture}[line width=.3mm, ->, >= angle 60]
								\node (2) at (0,0){\scalebox{.9}{$\llrr{2}$}};
								\node (12) at (1,1){\scalebox{.9}{$\llrr{1,2}$}};
								\node (24) at (1,-1){\scalebox{.9}{$\llrr{2,4}$}};
								\node (4) at (0,-2){\scalebox{.9}{$\llrr{4}$}};
								\node (14) at (2,0){\scalebox{.9}{$\llrr{1,4}$}};
								\node[rectangle,line width=0.2mm, double,rounded corners,fill=orange!20,draw] (23) at (2,-2){\scalebox{.9}{$\llrr{2,3}$}};
								\node[rectangle,line width=0.2mm, double,rounded corners,fill=orange!20,draw] (34) at (3,1){\scalebox{.9}{$\llrr{3,4}$}};
								\node[rectangle,line width=0.2mm, double,rounded corners,fill=orange!20,draw] (13) at (3,-1){\scalebox{.9}{$\llrr{1,3}$}};
								\node[rectangle,line width=0.2mm, double,rounded corners,fill=orange!20,draw] (3) at (4,0){\scalebox{.9}{$\llrr{3}$}};
								\node (1) at (4,-2){\scalebox{.9}{$\llrr{1}$}};
								\draw (2) -- (12);
								\draw (2) -- (24);
								\draw (4) -- (24);
								\draw (12) -- (14);
								\draw (24) -- (14);
								\draw (24) -- (23);
								\draw (14) -- (34);
								\draw (14) -- (13);
								\draw (23) -- (13);
								\draw (34) -- (3);
								\draw (13) -- (3);
								\draw (13) -- (1);
					\end{tikzpicture}}};
					\node (n) at (0.5,7){\scalebox{.325}{\begin{tikzpicture}[line width=.3mm, ->, >= angle 60]
								\node (2) at (0,0){\scalebox{.9}{$\llrr{2}$}};
								\node (12) at (1,1){\scalebox{.9}{$\llrr{1,2}$}};
								\node (24) at (1,-1){\scalebox{.9}{$\llrr{2,4}$}};
								\node (4) at (0,-2){\scalebox{.9}{$\llrr{4}$}};
								\node (14) at (2,0){\scalebox{.9}{$\llrr{1,4}$}};
								\node (23) at (2,-2){\scalebox{.9}{$\llrr{2,3}$}};
								\node[rectangle,line width=0.2mm, double,rounded corners,fill=orange!20,draw] (34) at (3,1){\scalebox{.9}{$\llrr{3,4}$}};
								\node[rectangle,line width=0.2mm, double,rounded corners,fill=orange!20,draw] (13) at (3,-1){\scalebox{.9}{$\llrr{1,3}$}};
								\node[rectangle,line width=0.2mm, double,rounded corners,fill=orange!20,draw] (3) at (4,0){\scalebox{.9}{$\llrr{3}$}};
								\node[rectangle,line width=0.2mm, double,rounded corners,fill=orange!20,draw] (1) at (4,-2){\scalebox{.9}{$\llrr{1}$}};
								\draw (2) -- (12);
								\draw (2) -- (24);
								\draw (4) -- (24);
								\draw (12) -- (14);
								\draw (24) -- (14);
								\draw (24) -- (23);
								\draw (14) -- (34);
								\draw (14) -- (13);
								\draw (23) -- (13);
								\draw (34) -- (3);
								\draw (13) -- (3);
								\draw (13) -- (1);
					\end{tikzpicture}}};
					\draw (a) -- (b);
					\draw (a) -- (c);
					\draw (b) -- (d);
					\draw (b) -- (e);
					\draw (b) -- (f);
					\draw (c) -- (d);
					\draw (d) -- (g);
					\draw (d) -- (h);
					\draw (e) -- (h);
					\draw (e) -- (i);
					\draw (f) -- (i);
					\draw (f) -- (j);
					\draw (g) -- (k);
					\draw[bend left=5] (g) edge (j);
					\draw (h) -- (k);
					\draw (i) -- (l);
					\draw (j) -- (l);
					\draw (k) -- (l);
					\draw (k) -- (m);
					\draw (l) -- (n);
					\draw (m) -- (n);
			\end{tikzpicture}
    \caption{The poset $(\Tilt(Q), \protect\Tleq)$ for the $A_4$ type quiver $Q$ seen in \cref{ex:TiltA4}. We gather together tiltings of the same equivalence class for $\approx_{\mathcal{E}_\diamond}$.}
    \label{fig:posettiltEdiamond}
\end{figure}
\end{example}

\begin{remark} \label{rem:latticetilt}
    By the work of \cite{IRTT15}, we know that this poset $(\Tilt(Q), \Tleq)$ is isomorphic to the poset of the torsion classes in $\rep(Q)$, and, therefore, it can be endowed with a lattice structure. The obtained lattice is isomorphic to the positive part of the well-known \emph{Cambrian lattice} in the sense of N. Reading \cite{R06}.
\end{remark}

\subsection{Gen-Sub operators on tilting objects}
\label{ss:GSandTiltmut}
In this section, we establish that, for any $T \in \Tilt(Q)$, the subcategory $\GS_\mathcal{E}(T)$ is the category additively generated by tilting representations $T'$ obtained from $T$ by applying a finite sequence of $\mathcal{E}$-mutations. Let us consider first the following equivalence relation on $\Tilt(Q)$.

\begin{definition} \label{def:GSErelation}
    We define the \new{$\GS_{\mathcal{E}}$-relation}, denoted by $\sim_{\mathcal{E}}$, on $\Tilt(Q)$ as follows. Given $T,T' \in \Tilt(Q)$, write $T \sim_{\mathcal{E}} T'$ whenever we have $\GS_{\mathcal{E}}(T) = \GS_{\mathcal{E}}(T')$.
\end{definition}

\begin{cor}\label{cor:propUniquetilting}
Let $T,T' \in \Tilt(Q)$. Then $T' \sim_\mathcal{E} T$ if and only if $T' \in \GS_{\mathcal{E}}(T)$. 
\end{cor}

Given $T,T' \in \Tilt(Q)$, our aim is to prove that $T \sim_{\mathcal{E}} T'$ if and only if $T \approx_{\mathcal{E}} T'$.

\begin{prop}
\label{prop:Tstable}
Let $T,T' \in \Tilt(Q)$. If $T' \approx_{\mathcal{E}} T$, then $T' \sim_{\mathcal{E}} T $.
\end{prop}

\begin{proof}
To prove this proposition, it is enough to show that $\GS_{\mathcal{E}}(T') = \GS_{\mathcal{E}}(T)$ whenever $T' \cong \pmb{\mu}_{U,\mathcal{E}}(T)$ for some $U \in \ind(Q)$. We obtain the desired result by induction on the length of the sequence of $\mathcal{E}$-mutations, which allows us to have $T'\approx_{\mathcal{E}} T$.

We have three cases to treat.
\begin{enumerate}[label=$\bullet$, itemsep=1mm]
    \item If $\pmb{\mu}_{U,\mathcal{E}}(T) \cong T$, then the result is obvious.
    \item If $\pmb{\mu}_{U,\mathcal{E}}(T) = \pmb{\mu}_{U,\mathcal{E}}^-(T)$, by writing $T = \widetilde{T} \oplus U$, we have the following short $\mathcal{E}$-exact sequence: 
\[\begin{tikzcd}
	0 & U & \Theta  & \Coker(f)&  0,
	\arrow[from=1-1, to=1-2]
	\arrow["f",tail, from=1-2, to=1-3]
	\arrow[two heads,from=1-3, to=1-4]
	\arrow[from=1-4, to=1-5]
\end{tikzcd} \] where $f$ is a left minimal $(\add(\widetilde{T}), \mathcal{E})$-approximation of $U$. As $\Theta \in \add(T) \subset \GS_\mathcal{E}(T)$, and $f$ is an $\mathcal{E}$-monomorphism, by \cref{lem:proponGS}, we have that $\Coker(f) \in  \GS_{\mathcal{E}}(T)$. Therefore $T' \cong \widetilde{T} \oplus \Coker(f) \in \GS_{\mathcal{E}}(T)$. 
    \item If $\pmb{\mu}_{U,\mathcal{E}}(T) = \pmb{\mu}_{U, \mathcal{E}}^{+}(X)$, by writing $T = \widetilde{T} \oplus U$, we have the following short $\mathcal{E}$-exact sequence: 
\[\begin{tikzcd}
	0 & \Ker(f) & \Theta' & U &  0,
	\arrow[from=1-1, to=1-2]
	\arrow[tail, from=1-2, to=1-3]
	\arrow["g",two heads,from=1-3, to=1-4]
	\arrow[from=1-4, to=1-5]
\end{tikzcd} \] where $g$ is a right minimal $(\add(\widetilde{T}), \mathcal{E})$-approximation of $U$. As $\Theta' \in \add(\widetilde{T})$ and $g$ is an $\mathcal{E}$-epimorphism, by \cref{lem:proponGS}, we have that $\Ker(f) \in \GS_{\mathcal{E}}(T)$. So $T' \cong \widetilde{T} \oplus \Ker(f) \in \GS_{\mathcal{E}}(T)$.
\end{enumerate}
In either case, we got that $T' \sim_{\mathcal{E}} T$.
\end{proof}

Besides establishing that \(T \sim_{\mathcal{E}} T'\) implies \(T \approx_{\mathcal{E}} T'\), we also show that the class \([T]_{\approx_{\mathcal{E}}}\) additively generates \(\GS_{\mathcal{E}}(T)\).
To do so, we will have to show a few more results. First, we recall the notion of the Jacobson radical of a category.

\begin{definition}\label{def:jacobson-radical} Let $\mathscr{A}$ be an abelian category.
The \emph{(Jacobson) radical} of $\mathscr{A}$ is the two-sided ideal
$\Rad_{\mathscr{A}}$ defined by
\[
\Rad_{\mathscr{A}}(X,Y)
\;=\;
\bigl\{\, h \in \Hom(X,Y)
\;\big|\;
\forall g \in \Hom(Y,X),\ \text{id}_Y - hg
\text{ is invertible}
\,\bigr\}
\]
for all objects $X,Y \in \mathscr{A}$.
\end{definition}

\begin{remark}
Following the definition, if $h \in \Rad_{\mathscr{A}}(X,Y)$, then $h$ is not an isomorhism.
\end{remark}

\begin{lemma}[\cite{ASS06}]\label{lem:rad-on-sums}
Let $\mathscr{A}$ be an abelian category and let
\[
X \;=\; \bigoplus_{i=1}^n X_i,
\qquad
Y \;=\; \bigoplus_{j=1}^m Y_j.
\]
For $f \in \Hom(X,Y)$, we write $f_{ji} = \pi_{Y_j}\, f\, \iota_{X_i} \in \Hom(X_i,Y_j)$
using the canonical inclusions $\iota_{X_i}:X_i\to X$ and projections $\pi_{Y_j}:Y\to Y_j$.
Then the following are equivalent:
\begin{enumerate}[label=$(\roman*)$, itemsep=1mm]
  \item $f \in \Rad_{\mathscr{A}}(X,Y)$.
  \item $f_{ji}\in \Rad_{\mathscr{A}}(X_i,Y_j)$ for all $1\le i\le n$ and $1\le j\le m$.
\end{enumerate}
\end{lemma}

\begin{lemma}\label{lem:minimalrad}
Let $\mathscr{A}$ be an abelian category and $M \in \mathscr{A}$. Consider a right minimal $\add(M)$-approximation $g$ of $X \in \Gen(M)$. Then we have the following short exact sequence 
    \[\begin{tikzcd}
	\xi : 0 & K_M & E_M  & X &  0,
	\arrow[from=1-1, to=1-2]
	\arrow["f",tail, from=1-2, to=1-3]
	\arrow["g",two heads,from=1-3, to=1-4]
	\arrow[from=1-4, to=1-5]
    \end{tikzcd} \]
    with $f \in \Rad_{\mathscr{A}}(K_M,E_M)$. Dually, if $ X \in \Sub(T)$ and $f : X \rightarrow E_M$ is a left minimal $\add(M)$-approximation of $X$, then $g \in \Rad_{\mathscr{A}}(E_M,\Coker(f))$. 
\end{lemma}

\begin{proof}
By \cref{prop:Approxismonoepi2}, the map $g$ is an epimorphism and we get the short exact sequence
\[\begin{tikzcd}
	\xi : 0 & K_M & E_M  & X &  0,
	\arrow[from=1-1, to=1-2]
	\arrow["f",tail, from=1-2, to=1-3]
	\arrow["g",two heads,from=1-3, to=1-4]
	\arrow[from=1-4, to=1-5]
    \end{tikzcd} \] with $E_M \in \add(M)$.

Consider the morphism $\text{id}_{E_M} - fh \in \End(E_M)$ with $h \in \Hom(E_M,K_M)$. We have $g(\text{id}_{E_M} - fh) = g - gfh = g$. Since $g$ is minimal, $\text{id}_{E_M} - fh$ is invertible for any $h \in \Hom(E_M,K_M)$. Hence, $f \in \Rad_{\mathscr{A}}(K_M,E_M)$.
\end{proof}

\begin{lemma}[\cite{APT15}] \label{lem:Ibra} 
    Let $T \in \Tilt(Q)$ and $X \in \Gen(T)$. Consider a right minimal $\add(T)$-approximation $g$ of $X$, which is an epimorphism by \cref{prop:Approxismonoepi2}. Then we have the following short exact sequence 
    \[\begin{tikzcd}
	\xi : 0 & K_T & E_T  & X &  0,
	\arrow[from=1-1, to=1-2]
	\arrow[tail, from=1-2, to=1-3]
	\arrow["g",two heads,from=1-3, to=1-4]
	\arrow[from=1-4, to=1-5]
    \end{tikzcd} \]
    with $K_T \in \add(T)$. Dually, if $ X \in \Sub(T)$ and $f$ is a left minimal $\add(T)$-approximation of $X$, then $\Coker(f) \in \add(T)$.
\end{lemma}

\begin{lemma} \label{lem:UinSubXinGen}
Let $T\in\Tilt(Q)$ and $X \notin \add(T)$ indecomposable. The following assertions hold:
\begin{enumerate}[label=$(\roman*)$, itemsep=1mm]
    \item if $X \in \Gen_\mathcal{E}(T)$, then there is an indecomposable $U \in \add(T)$ such that $U \in \Sub_\mathcal{E}(M)$ and $X \in \Gen_\mathcal{E}(M)$ with $T = U \oplus M$; and,
    \item if $X \in \Sub_\mathcal{E}(T)$, then there is an indecomposable $U \in \add(T)$ such that $U \in \Gen_\mathcal{E}(M)$ and $X \in \Sub_\mathcal{E}(M)$ with $T = U \oplus M$.
\end{enumerate}
\end{lemma}

\begin{proof}
Consider $X \in \Gen_\mathcal{E}(T)$. By \cref{prop:Approxismonoepi} and \cref{lem:TiltandApprox},
there exists a right minimal $(\add(T),\mathcal{E})$-approximation $g$ of X :
\[
\begin{tikzcd}
\xi : 0 \arrow[r] & K_T \arrow[r,tail,"f"] & E_T \arrow[r,two heads,"g"] & X \arrow[r] & 0
\end{tikzcd}
\]
with $E_T\in\add(T)$. Given that $X \notin \add(T)$, $\xi$ does not split and $K_T \neq 0$. Moreover, $K_T \in \add(T)$ by \cref{lem:Ibra} and $f \in \Rad_{\mathscr{A}}(K_T,E_T)$ by \cref{lem:minimalrad}.

\smallskip

We will show that $\add(K_T) \cap \add(E_T) = 0$. 

\smallskip
Consider any morphism $\alpha \in \Hom(K_T,E_T)$. Given that $X$ is indecomposable and $Q$ is of $ADE$ Dynkin type, then $X$ is rigid. By applying the functor $\Hom(-,X)$, we get the short exact sequence

 \[\begin{tikzcd}
	\xi : 0 & \Hom(X,X) & \Hom(E_T,X)  & \Hom(K_T,X) &  0,
	\arrow[from=1-1, to=1-2]
	\arrow[tail, from=1-2, to=1-3]
	\arrow["{\scriptstyle \Hom(f,X)}",two heads,from=1-3, to=1-4]
	\arrow[from=1-4, to=1-5]
    \end{tikzcd} \]
Therefore, there exist $\beta \in \Hom(E_T,X)$ such that $\beta f = g\alpha$.

\smallskip
    
We can construct the following commutative diagram:

\[
\begin{tikzcd}[column sep=large]
& 0 \arrow[r] & K_T \arrow[r, tail, "f"] \arrow["\alpha",d]
            & E_T \arrow[r, two heads, "g"] \arrow["\beta",d, dashed]
            & X   \arrow[r]
            & 0   \\
0 \arrow[r] & K_T \arrow[r, tail, "f"']
    & E_T \arrow[r, two heads, "g"']
    & X   \arrow[r]
    & 0 &
\end{tikzcd}
\]

Since $g$ is a right $\add(T)$-approximation of $X$, there is a morphism $\beta' \in \End(E_T)$ such that $\beta = g\beta'$. We get that $g(\alpha - \beta'f) = g\alpha - g\beta' f = \beta f - \beta f = 0$. By the kernel universal property, there is a morphism $\alpha' \in \End(K_T)$ such that $(\alpha - \beta'f) = f\alpha'$. Therefore, $\alpha = f\alpha' + \beta'f$. By \cref{lem:minimalrad}, $f \in \Rad_{\rep Q}(K_T,E_T)$, so is $\alpha$ because $\Rad_{\rep Q}$ is a two-sided ideal of $\rep Q$.

\smallskip

Given that any morphism $\alpha \in \Hom(K_T,E_T)$ is in $\Rad_{\rep Q}(K_T,E_T)$, by \cref{lem:rad-on-sums}, we get that $\add(K_T) \cap \add(E_T) = 0$.

\smallskip
Now, consider an indecomposable $U \in \add(K_T) \subseteq \add(T)$. Since $U \notin \add(E_T)$, we have $E_T \in \add(M)$ with $T = U \oplus M$. Therefore, since $\xi$ is an $\mathcal{E}$-admissible short exact sequence, then $K_T \in \Sub_\mathcal{E}(M)$ and $X \in \Gen_\mathcal{E}(M)$. Hence, $U \in \Sub_\mathcal{E}(M)$ and $X \in \Gen_\mathcal{E}(M)$. 
\end{proof}

\begin{remark} \label{rem:citeIA}
    The above proof was mainly inspired by the proof of \cite[Lemma 2.25]{AI12}. 
\end{remark}

\begin{prop}\label{prop:noMutation}
Let $T\in\Tilt(Q)$. If $T$ admits no left $\mathcal{E}$-mutation at any indecomposable summands, then $\Gen_\mathcal{E}(T) =\add(T)$. Dually, If $T$ admits no right $\mathcal{E}$-mutation at any indecomposable summands, then $\Sub_\mathcal{E}(T) =\add(T)$.
\end{prop}

\begin{proof}
Suppose that $\add(T) \subsetneq \Gen_\mathcal{E}(T)$. Given that, there exist $X \in \Gen_\mathcal{E}(T)$ such that $X \notin \add(T)$. By \cref{lem:UinSubXinGen}, there exist an indecomposable $U \in \add(T)$ such that $U \in \Sub_\mathcal{E}(\widetilde{T})$ with $T = \widetilde{T} \oplus U $. By \cref{prop:Approxismonoepi} and \cref{lem:TiltandApprox}, $U$ admits a left minimal $(\add(\widetilde{T}),\mathcal{E})$-approximation $f$. In the sense of \cref{def:leftrightEmutation}, $T$ admits a left $\mathcal{E}$-mutation at $U$ which is a contradiction. Hence, $\Gen_\mathcal{E}(T) =\add(T)$.
\end{proof}

The following results highlight the description of $\Gen_\mathcal{E}(T)$ via the tilting representations that cover $T$, and the description of $\Sub_\mathcal{E}(T)$  via the tilting representations that $ T$ covers for $\TleqE$.

\begin{prop}\label{prop:InductiveStepGenSub}
Let $T \in \Tilt(Q)$. Then the following equalities hold:
\begin{enumerate}[label=$(\roman*)$, itemsep=1mm]
    \item
    $\displaystyle
        \Gen_{\mathcal{E}}(T) 
        = \add(T) \oplus 
        \Bigl(\,\bigoplus_{T \TldotEunder T'} \Gen_{\mathcal{E}}(T') \Bigr);$ and,

    \item $\displaystyle 
        \Sub_{\mathcal{E}}(T) 
        = \add(T) \oplus 
        \Bigl(\,\bigoplus_{T' \TldotEunder T} \Sub_{\mathcal{E}}(T') \Bigr).$
\end{enumerate}
\end{prop}
%\Sunny{In progress. Ceci est la dernière pièce manquante}
\begin{proof} The assertions $(i)$ and $(ii)$ are dual statements. It suffices for us to show $(i)$.

First, we will show that $  
         \add(T) \oplus 
        \Bigl(\,\bigoplus_{T \TldotEunder T'} \Gen_{\mathcal{E}}(T') \Bigr) \subseteq \Gen_{\mathcal{E}}(T)$.
        
Let $X \in \add(T) \oplus 
        \Bigl(\,\bigoplus_{T \TldotEunder T'} \Gen_{\mathcal{E}}(T') \Bigr)$ be an indecomposable representation. If $X \in \add(T)$, then obviously $X \in \Gen_\mathcal{E}(T)$.  Assume that $X \in \Gen_\mathcal{E}(T')$ for some $T \TldotE T'$. Set $U \in \ind(Q)$ such that $T \ncong \pmb{\mu}_{U,\mathcal{E}}^{-}(T) = T'$. We have the following $\mathcal{E}$-admissible short exact sequence
\[
\begin{tikzcd}
\xi : 0 \arrow[r] & U \arrow[r,tail,"f"] & \Theta \arrow[r,two heads,"g"] & \Coker(f) \arrow[r] & 0,
\end{tikzcd}
\]
with $f$ a left minimal $(\add(\widetilde{T}),\mathcal{E})$-approximation of $U$ with $T=\widetilde{T}\oplus U$. Since $\Theta \in \add(\widetilde{T}) \subset \add(T)$, then $\Coker(f) \in \Gen_\mathcal{E}(T)$. Given that $T' = \widetilde{T} \oplus \Coker(f)$, then $\Gen_\mathcal{E}(T') \subseteq \Gen_\mathcal{E}(T)$, so $X \in \Gen_\mathcal{E}(T)$.

Now, we will show that $  
        \Gen_{\mathcal{E}}(T) \subseteq \add(T) \oplus 
        \Bigl(\,\bigoplus_{T \TldotEunder T'} \Gen_{\mathcal{E}}(T') \Bigr)$.

Let $X \in \Gen_\mathcal{E}(T)$ be an indecomposable representation. If $X \in \add(T)$, we are done. Otherwise, by \cref{lem:UinSubXinGen}, there is an indecomposable representation $U \in \add(T)$ such that $X \in \Gen_\mathcal{E}(\widetilde{T})$ and $U \in \Sub_\mathcal{E}(\widetilde{T})$ with $T =  \widetilde{T} \oplus U$. By \cref{prop:Approxismonoepi} and \cref{lem:TiltandApprox}, $U$ admits a left minimal $(\add(\widetilde{T}),\mathcal{E})$-approximation $f$. In the sense of \cref{def:leftrightEmutation}, $T$ admits a left $\mathcal{E}$-mutation at $U$, and we get $\pmb{\mu}_{U,\mathcal{E}}^{-}(T) = T'$ for some tilting object $T \TldotEunder T'$ with $\add(\widetilde{T}) \subseteq \add(T')$. Therefore,  $X \in \Gen_\mathcal{E}(\widetilde{T}) \subseteq \Gen_\mathcal{E}(T')$ implies $X \in \Gen_\mathcal{E}(T')$ for some tilting object $T \TldotE T'$. Hence $ \Gen_{\mathcal{E}}(T) \subseteq \add(T) \oplus 
        \Bigl(\,\bigoplus_{T \TldotEunder T'} \Gen_{\mathcal{E}}(T') \Bigr)$. 
\end{proof}

\begin{prop}\label{prop:GenAdd}
Let $T \in \Tilt(Q)$. Then the following equalities hold:
\begin{enumerate}
    \item
    $\displaystyle 
\Gen_{\mathcal E}(T)
\;=\; \bigoplus_{T \TleqEunder T'}
\add(T')$; and,

    \item
$\displaystyle
\Sub_{\mathcal E}(T)
\;=\;
 \bigoplus_{T' \TleqEunder T }
\add(T').
$
\end{enumerate}   
\end{prop}

\begin{proof}
Since $Q$ is representation-finite, the result follows directly from Propositions \ref{prop:noMutation} and \ref{prop:InductiveStepGenSub} by induction.
\end{proof}

Before we show the proof of the last theorem, we will show how to choose some special tilting object $T' \in [T]_{\approx_{\mathcal{E}}}$ to generate $\GS_{\mathcal{E}}(T)$.

\begin{lemma}\label{UnicityGen}
Let $T_1,T_2 \in \Tilt(Q)$. If $\Gen_\mathcal{E}(T_1) = \Gen_\mathcal{E}(T_2)$, then $T_1 \cong T_2$. Dually, it holds if $\Sub_\mathcal{E}(T_1) = \Sub_\mathcal{E}(T_2)$.   
\end{lemma}

\begin{proof}
Since $\Gen_\mathcal{E}(T_1) = \Gen_\mathcal{E}(T_2)$, we have the two $\mathcal{E}$-admissible short exact sequences:
\[\begin{tikzcd}
	\xi_1 : 0 & K & \Theta_1  & T_2 &  0,
	\arrow[from=1-1, to=1-2]
	\arrow[tail, from=1-2, to=1-3]
	\arrow[two heads,from=1-3, to=1-4]
	\arrow[from=1-4, to=1-5]
    \end{tikzcd} \]
with $\Theta_1 \in \add(T_1)$, and
\[\begin{tikzcd}
	\xi_2 : 0 & K' & \Theta_2  & T_1 &  0,
	\arrow[from=1-1, to=1-2]
	\arrow[tail, from=1-2, to=1-3]
	\arrow[two heads,from=1-3, to=1-4]
	\arrow[from=1-4, to=1-5]
    \end{tikzcd} \]
with $\Theta_2 \in \add(T_2)$. Applying the functors $\Hom(T_1,-)$ to $\xi_1$ and $\Hom(T_2,-)$ to $\xi_2$, because $\Ext^1(T_1,\Theta_2)$ = $\Ext^1(T_2,\Theta_2) = 0$ and because $\rep Q $ is hereditary, we obtain $\Ext^1(T_1,T_2)$ = $\Ext^1(T_2,T_1)=0$. Since $T_1,T_2$ are maximally rigid objects, then $T_1 \cong T_2$.
\end{proof}

\begin{prop}\label{prop:MinimalT}
Let $T \in \Tilt(Q)$. Then $[T]_{\approx_{\mathcal{E}}}$ has a unique minimal element $T'$ for $\TleqE$ and $\GS_{\mathcal{E}}(T) = \Gen_\mathcal{E}(T')$. Moreover, $\GS_{\mathcal{E}}(T) = \GS^1_{\mathcal{E}}(T')$.
\end{prop}
%\Sunny{Ici, il faudrait remanier un peu la structure pour parler du lattice avant}
\begin{proof}
It is clear that $[T]_{\approx_{\mathcal{E}}}$ has minimal elements for $\TleqE$. We must only show that there is indeed a unique one. Let $T' \in [T]_{\approx_{\mathcal{E}}}$ be a minimal element. By \cref{prop:Tstable}, we already have $\Gen_\mathcal{E}(T') \subseteq \GS_{\mathcal{E}}(T')=\GS_{\mathcal{E}}(T)$.

Let us show that $\GS_{\mathcal{E}}(T') \subseteq \Gen_\mathcal{E}(T')$ by induction. First, we have that $\GS^0_{\mathcal{E}}(T') = \add(T') \subseteq \Gen_\mathcal{E}(T')$. Assume that, for a fixed $n \in \mathbb{N}$, we have $\GS^n_{\mathcal{E}}(T') \subseteq \Gen_\mathcal{E}(T')$. Consider $X \in \GS^{n+1}_{\mathcal{E}}(T')$. We have two cases to treat.

If $X \in \Gen_{\mathcal{E}}(F)$ for some $F \in \GS^{n}_{\mathcal{E}}(T')$, by induction hypothesis, $F \in \Gen_\mathcal{E}(T')$. As $\mathcal{E}$-epimorphisms are closed under composition, so is $X$. 

Now assume that $X \in \Sub_{\mathcal{E}}(F)$ for some $F \in \GS^{n}_{\mathcal{E}}(T')$. Then we have the following short $\mathcal{E}$-exact sequence
\[\begin{tikzcd}
	0 & X & F  & C &  0.
	\arrow[from=1-1, to=1-2]
	\arrow[tail, from=1-2, to=1-3]
	\arrow[two heads,from=1-3, to=1-4]
	\arrow[from=1-4, to=1-5]
\end{tikzcd} \] By induction hypothesis, $F \in \Gen_\mathcal{E}(T')$, and therefore, there exists a short $\mathcal{E}$-exact sequence
\[\begin{tikzcd}
	0 & K & \Theta  & F &  0,
	\arrow[from=1-1, to=1-2]
	\arrow[tail, from=1-2, to=1-3]
	\arrow[two heads,from=1-3, to=1-4]
	\arrow[from=1-4, to=1-5]
\end{tikzcd} \] with $\Theta \in \add(T')$. We obtain the following commutative pull-back diagram.
\[\begin{tikzcd}
         & 0 & 0 & 0 &  \\
	0 & 0 & C'  & C &  0 \\
     0 & K & \Theta & F & 0 \\
      0 & K & PB & X & 0 \\
       & 0 & 0 & 0 & 
	\arrow[from=2-1, to=2-2]
	\arrow[tail, from=2-2, to=2-3]
	\arrow[two heads,from=2-3, to=2-4]
	\arrow[from=2-4, to=2-5]
        \arrow[from=3-1, to=3-2]
	\arrow[tail, from=3-2, to=3-3]
	\arrow[two heads,from=3-3, to=3-4]
	\arrow[from=3-4, to=3-5]
        \arrow[from=4-1, to=4-2]
	\arrow[tail, from=4-2, to=4-3]
	\arrow[two heads,from=4-3, to=4-4]
	\arrow[from=4-4, to=4-5]
        \arrow[from=5-2, to=4-2]
	\arrow[equals, from=4-2, to=3-2]
	\arrow[two heads,from=3-2, to=2-2]
	\arrow[from=2-2, to=1-2]
        \arrow[from=5-3, to=4-3]
	\arrow[tail, from=4-3, to=3-3]
	\arrow[two heads,from=3-3, to=2-3]
	\arrow[from=2-3, to=1-3]
        \arrow[from=5-4, to=4-4]
	\arrow[tail, from=4-4, to=3-4]
	\arrow[two heads,from=3-4, to=2-4]
	\arrow[from=2-4, to=1-4]
\end{tikzcd} \]
The short sequences given by the three columns and those given by the last two rows are exact. By \cref{lem:3x3}, the short sequence given by the first row is also exact. Thus $C' \cong C$. Moreover, the short exact sequences given by the three rows and those given by the first and third columns are in ${\mathcal{E}}$. By \cref{lem:3x3}, the short sequence given by the middle column is also $\mathcal{E}$-exact. So $PB \in \Sub_\mathcal{E}(\Theta)$. 

Since $\Theta \in \add(T')$ and $T'$ admits no right $\mathcal{E}$-mutation (as $T'$ is minimal for $\TleqE$), then $\Sub_\mathcal{E}(\Theta) \subseteq \Sub_\mathcal{E}(T') = \add(T')$ by \cref{prop:noMutation}. Hence, $PB \in \add(T')$. Given that the last row is $\mathcal{E}$-exact, $X \in \Gen_\mathcal{E}(PB) \subseteq \Gen_\mathcal{E}(T')$.

Thus, this completes the induction proof and we get $\GS_{\mathcal{E}}(T') \subseteq \Gen_\mathcal{E}(T')$. Thus $\GS_{\mathcal{E}}(T) = \GS_{\mathcal{E}}(T') = \Gen_\mathcal{E}(T')$. The uniqueness directly follows from \cref{UnicityGen}. Recall that $\Gen_\mathcal{E}(T') \subseteq \GS_{\mathcal{E}}(T)$ and $\Sub_\mathcal{E}(T') = \add(T') \subseteq \GS_{\mathcal{E}}(T)$ from Propositions \ref{prop:noMutation} and \ref{prop:Tstable}. So $\GS^1_{\mathcal{E}}(T') \subseteq \GS_{\mathcal{E}}(T) = \Gen_\mathcal{E}(T') \subseteq \GS^1_{\mathcal{E}}(T')$.
\end{proof}

By dual statements, we get the following corollary.

\begin{cor} \label{cor:maxEreach}
Let $T \in \Tilt(Q)$. Then $[T]_{\approx_{\mathcal{E}}}$ has a unique maximal element $T''$ of $[T]_{\approx_{\mathcal{E}}}$ for $\TleqE$ such that $\GS_{\mathcal{E}}(T) = \Sub_\mathcal{E}(T')$. Moreover, $\GS_\mathcal{E}(T) = \GS_\mathcal{E}^1 (T'')$.
\end{cor}

We can now provide proof of the results we aimed for.

\begin{prop} \label{prop:samesamebutdifferent}
Let \(T,T' \in \Tilt(Q)\). Then $T \sim_{\mathcal{E}} T'$ if and only if $T \approx_{\mathcal{E}} T'$.
\end{prop}

\begin{proof}
By \cref{prop:Tstable}, we have that if $T' \approx_{\mathcal{E}} T$, then $T' \sim_{\mathcal{E}} T$. Consider $T,T' \in \Tilt(Q)$ such that $T' \sim_{\mathcal{E}} T$. By definition, $\GS_{\mathcal{E}}(T) = \GS_{\mathcal{E}}(T')$. Due to \cref{prop:MinimalT}, there exist $S \in [T]_{\approx_{\mathcal{E}}}$ and $S' \in [T']_{\approx_{\mathcal{E}}}$ such that $\GS_{\mathcal{E}}(T) = \Gen_\mathcal{E}(S)$ and $\GS_{\mathcal{E}}(T') = \Gen_\mathcal{E}(S')$. We obtain $\Gen_\mathcal{E}(S) = \Gen_\mathcal{E}(S')$ and, by \cref{UnicityGen}, $S \cong S'$. We conclude that $[T]_{\approx_{\mathcal{E}}} = [T']_{\approx_{\mathcal{E}}}$.
\end{proof}

We can finally prove one of our main results.

\begin{theorem}
\label{th:Treaching}
Let $T \in \Tilt(Q)$. Then \[\GS_{\mathcal{E}}(T) = \add\left(\bigoplus_{T' \in [T]_{\approx_{\mathcal{E}}}} T' \right)\]
\end{theorem}

\begin{proof} 
First of all, \[\add\left(\bigoplus_{T' \in [T]_{\approx_{\mathcal{E}}}} T' \right) \subseteq  \GS_{\mathcal{E}}(T)\] by \cref{prop:Tstable}. By \cref{prop:MinimalT}, there exists (a unique) $ S \in [T]_{\approx_{\mathcal{E}}}$ such that $\GS_{\mathcal{E}}(T) = \Gen_\mathcal{E}(S)$. By \cref{prop:GenAdd}, we can write \[
\GS_\mathcal{E}(T) = \Gen_{\mathcal E}(S)
\;=\; \bigoplus_{\substack{S \TleqEunder T'}}
\add(T') \subseteq \add\left(\bigoplus_{T' \in [T]_{\approx_{\mathcal{E}}}} T' \right).
\]
\end{proof}

\begin{example}
The following figure illustrates an example of Theorem \ref{th:Treaching}, as well as Corollary \ref{cor:propUniquetilting} and Proposition \ref{prop:Tstable}, for a specific equivalence class of tilting representations selected from those shown in \cref{fig:posettiltEdiamond}.
\begin{figure}[ht!]
\centering
\begin{tikzpicture}[line width=.3mm, ->, >= angle 60]

                                \node at (-0.4,2.25) { $[T_1]_{\approx_{\mathcal{E}_\diamond}}=[T_2]_{\approx_{\mathcal{E}_\diamond}}=[T_3]_{\approx_{\mathcal{E}_\diamond}}$};
                                \node at (-1.5,1.15) {$T_1$};
                                \node at (1.25,-0.1) {$T_2$};
                                \node at (-1.5,-1.35) {$T_3$};
                                \node at (5.5,1.5) {\shortstack{
    $\GS_{\mathcal{E}_\diamond}(T_1)
    = \add \left(\displaystyle T_1 \oplus T_2 \oplus T_3 \right)$
  }};
                                \draw[lava, fill=lava,opacity=0.2, loosely dotted] (-0.4,-1.4) ellipse (3.25cm and 3.25cm);

                                \node (a) at (-1.5,0){\scalebox{.5}{\begin{tikzpicture}[line width=.3mm, ->, >= angle 60]
								\node[rectangle,line width=0.2mm, double,rounded corners,fill=orange!20,draw] (2) at (0,0){\scalebox{.9}{$\llrr{2}$}};
								\node[rectangle,line width=0.2mm, double,rounded corners,fill=orange!20,draw] (12) at (1,1){\scalebox{.9}{$\llrr{1,2}$}};
								\node[rectangle,line width=0.2mm, double,rounded corners,fill=orange!20,draw] (24) at (1,-1){\scalebox{.9}{$\llrr{2,4}$}};
								\node (4) at (0,-2){\scalebox{.9}{$\llrr{4}$}};
								\node (14) at (2,0){\scalebox{.9}{$\llrr{1,4}$}};
								\node[rectangle,line width=0.2mm, double,rounded corners,fill=orange!20,draw] (23) at (2,-2){\scalebox{.9}{$\llrr{2,3}$}};
								\node (34) at (3,1){\scalebox{.9}{$\llrr{3,4}$}};
								\node (13) at (3,-1){\scalebox{.9}{$\llrr{1,3}$}};
								\node (3) at (4,0){\scalebox{.9}{$\llrr{3}$}};
								\node (1) at (4,-2){\scalebox{.9}{$\llrr{1}$}};
								\draw (2) -- (12);
								\draw (2) -- (24);
								\draw (4) -- (24);
								\draw (12) -- (14);
								\draw (24) -- (14);
								\draw (24) -- (23);
								\draw (14) -- (34);
								\draw (14) -- (13);
								\draw (23) -- (13);
								\draw (34) -- (3);
								\draw (13) -- (3);
								\draw (13) -- (1);
					\end{tikzpicture}}};
					\node (b) at (1.25,-1.25){\scalebox{.5}{\begin{tikzpicture}[line width=.3mm, ->, >= angle 60]
								\node (2) at (0,0){\scalebox{.9}{$\llrr{2}$}};
								\node[rectangle,line width=0.2mm, double,rounded corners,fill=orange!20,draw] (12) at (1,1){\scalebox{.9}{$\llrr{1,2}$}};
								\node[rectangle,line width=0.2mm, double,rounded corners,fill=orange!20,draw] (24) at (1,-1){\scalebox{.9}{$\llrr{2,4}$}};
								\node (4) at (0,-2){\scalebox{.9}{$\llrr{4}$}};
								\node[rectangle,line width=0.2mm, double,rounded corners,fill=orange!20,draw] (14) at (2,0){\scalebox{.9}{$\llrr{1,4}$}};
								\node[rectangle,line width=0.2mm, double,rounded corners,fill=orange!20,draw] (23) at (2,-2){\scalebox{.9}{$\llrr{2,3}$}};
								\node (34) at (3,1){\scalebox{.9}{$\llrr{3,4}$}};
								\node (13) at (3,-1){\scalebox{.9}{$\llrr{1,3}$}};
								\node (3) at (4,0){\scalebox{.9}{$\llrr{3}$}};
								\node (1) at (4,-2){\scalebox{.9}{$\llrr{1}$}};
								\draw (2) -- (12);
								\draw (2) -- (24);
								\draw (4) -- (24);
								\draw (12) -- (14);
								\draw (24) -- (14);
								\draw (24) -- (23);
								\draw (14) -- (34);
								\draw (14) -- (13);
								\draw (23) -- (13);
								\draw (34) -- (3);
								\draw (13) -- (3);
								\draw (13) -- (1);
					\end{tikzpicture}}};
                                \node (c) at (-1.5,-2.5){\scalebox{.5}{\begin{tikzpicture}[line width=.3mm, ->, >= angle 60]
								\node (2) at (0,0){\scalebox{.9}{$\llrr{2}$}};
								\node[rectangle,line width=0.2mm, double,rounded corners,fill=orange!20,draw] (12) at (1,1){\scalebox{.9}{$\llrr{1,2}$}};
								\node (24) at (1,-1){\scalebox{.9}{$\llrr{2,4}$}};
								\node (4) at (0,-2){\scalebox{.9}{$\llrr{4}$}};
								\node[rectangle,line width=0.2mm, double,rounded corners,fill=orange!20,draw] (14) at (2,0){\scalebox{.9}{$\llrr{1,4}$}};
								\node[rectangle,line width=0.2mm, double,rounded corners,fill=orange!20,draw] (23) at (2,-2){\scalebox{.9}{$\llrr{2,3}$}};
								\node (34) at (3,1){\scalebox{.9}{$\llrr{3,4}$}};
								\node[rectangle,line width=0.2mm, double,rounded corners,fill=orange!20,draw] (13) at (3,-1){\scalebox{.9}{$\llrr{1,3}$}};
								\node (3) at (4,0){\scalebox{.9}{$\llrr{3}$}};
								\node (1) at (4,-2){\scalebox{.9}{$\llrr{1}$}};
								\draw (2) -- (12);
								\draw (2) -- (24);
								\draw (4) -- (24);
								\draw (12) -- (14);
								\draw (24) -- (14);
								\draw (24) -- (23);
								\draw (14) -- (34);
								\draw (14) -- (13);
								\draw (23) -- (13);
								\draw (34) -- (3);
								\draw (13) -- (3);
								\draw (13) -- (1);
					\end{tikzpicture}}};

                                \node (d) at (5.5,-1.5) {\scalebox{0.9}
                                {\begin{tikzpicture}[line width=.3mm, ->, >= angle 60]
								\node[rectangle,line width=0.2mm,fill=Plum!20,draw] (2) at (0,0){\scalebox{.9}{$\llrr{2}$}};
								\node[rectangle,line width=0.2mm,fill=Plum!20,draw] (12) at (1,1){\scalebox{.9}{$\llrr{1,2}$}};
								\node[rectangle,line width=0.2mm,fill=Plum!20,draw] (24) at (1,-1){\scalebox{.9}{$\llrr{2,4}$}};
								\node (4) at (0,-2){\scalebox{.9}{$\llrr{4}$}};
								\node[rectangle,line width=0.2mm,fill=Plum!20,draw] (14) at (2,0){\scalebox{.9}{$\llrr{1,4}$}};
								\node[rectangle,line width=0.2mm,fill=Plum!20,draw] (23) at (2,-2){\scalebox{.9}{$\llrr{2,3}$}};
								\node (34) at (3,1){\scalebox{.9}{$\llrr{3,4}$}};
								\node[rectangle,line width=0.2mm,fill=Plum!20,draw] (13) at (3,-1){\scalebox{.9}{$\llrr{1,3}$}};
								\node (3) at (4,0){\scalebox{.9}{$\llrr{3}$}};
								\node (1) at (4,-2){\scalebox{.9}{$\llrr{1}$}};
								\draw (2) -- (12);
								\draw (2) -- (24);
								\draw (4) -- (24);
								\draw (12) -- (14);
								\draw (24) -- (14);
								\draw (24) -- (23);
								\draw (14) -- (34);
								\draw (14) -- (13);
								\draw (23) -- (13);
								\draw (34) -- (3);
								\draw (13) -- (3);
								\draw (13) -- (1);
                                \end{tikzpicture}}};
\end{tikzpicture}
\label{fig:EclassandGSE}
\caption{Illustration of \cref{th:Treaching} via one equivalence class in $\Tilt(Q)$ for $\approx_{\mathcal{E}_{\diamond}}$ and for $Q$ the $A_4$ type quiver in \cref{ex:TiltA4}.}
\end{figure}
\end{example}

\subsection{Canonically Jordan recoverable subcategories for type A algebras}
    \label{ss:AllCJRtypeA}

In the following, we prepare the ground for the proof of our main result, which will be an application of the results in \cref{ss:Tilting,ss:AllCJRtypeA,ss:GSandTiltmut}. To do so, we show that any maximal canonically Jordan recoverable subcategory of $\rep(Q)$ is extension-closed and contains a tilting representation.

\begin{prop} \label{prop:maxCJRextclos}
    Let $Q$ be an $A_n$ type quiver for some $n \in \mathbb{N}^*$. Then all the maximal canonically Jordan recoverable subcategories of $\rep(Q)$ are extension-closed.
\end{prop}

\begin{proof}
    Let $\mathscr{C} \subseteq \rep(Q)$ be a maximal canonically Jordan recoverable subcategory. By \cref{thm:CJRmaxcomb}, consider $\B,\E \subseteq \{1, \ldots,n\}$ such that:
    \begin{enumerate}[label=$\bullet$,itemsep=1mm]
        \item the pair $(\B,\E[1])$ is a set partition of $\{1,\ldots,n+1\}$; and, 
        \item we have $\mathscr{C} = \opC(\B,\E)$.
    \end{enumerate} To check that $\mathscr{C}$ is extension closed, by the fact that $\Ext^1$ is a bilinear functor, it is enough to prove that for any pair $(D,F)$ of indecomposable representations in $\mathscr{C}$, and for any nonsplit short exact sequence \[\begin{tikzcd}
	0 & D & E & F &  0,
	\arrow[from=1-1, to=1-2]
	\arrow[tail, from=1-2, to=1-3]
	\arrow[two heads,from=1-3, to=1-4]
	\arrow[from=1-4, to=1-5]
\end{tikzcd}\] we have $E \in \mathscr{C}$. 
    
    Given such a short exact sequence, by \cref{thm:exttypeA,thm:CJRreptheory}, we have that $E \cong E_1 \oplus E_2$ and there exist $K,L \in \mathcal{I}_n$ with $K \cap L \neq \varnothing$ such that:
    \begin{enumerate}[label=$\bullet$,itemsep=1mm]
    \item $K,L,K \cap L$ and $K \cup L$ are four different intervals in $\mathcal{I}_n$; and,
    \item either $\{D,F\} = \{X_K, X_L\}$ and $\{E_1,E_2\} = \{X_{K\cap L}, X_{K \cup L} \}$, or \\
    $\{D,F\} = \{X_{K \cap L}, X_{K \cup L}\}$ and $\{E_1,E_2\} = \{X_{K}, X_{L} \}$ .
    \end{enumerate}

    Set $K = \llrr{b_1, e_1}$ and $L = \llrr{b_2,e_2}$. Up to exchanging the role of $K$ and $L$, assume that $e_1 \leqslant e_2$. To make sure that $K,L,K \cap L$ and $K \cup L$ are four distinct intervals nonempty in $\mathcal{I}_n$, we must have either
    \begin{enumerate}[label=$(\arabic*)$, itemsep=1mm]
        \item $b_1 < b_2 \leqslant e_1 < e_2$; or,
        \item  $b_2 < b_1 \leqslant e_1 < e_2$.
    \end{enumerate}
    Assume that we are in the case $(1)$. If $\{D,F\} =\{X_K, X_L\}$, then $b_1,b_2 \in \B$ and $e_1,e_2 \in \E$. Therefore, $K \cap L = \llrr{b_2;e_1} \in \Int(\mathscr{C})$ and $K \cup L = \llrr{b_1, e_2} \in \Int(\mathscr{C})$. So $E \in \mathscr{C}$. The case where $\{D,F\} = \{X_{K\cap L},X_{K \cup L}\}$ can be treated similarly. These arguments yield the same result if we start with case $(2)$. 

    In either case, we get $E \in \mathscr{C}$, and so the desired result.
\end{proof}

\begin{remark}
    In contrast to \cref{ex:dontworkforanyE}, every maximal $\mathcal{E}_\diamond$-adapted subcategory is extension-closed. This is a direct consequence of \cref{prop:maxCJRextclos} and \cref{thm:CJRreptheory}.
\end{remark}

 \begin{prop} \label{prop:addTCJR}
    Let $Q$ be an $A_n$ type quiver. Let $T \in \rep(Q)$ be a rigid representation. Then $\add(T)$ is canonically Jordan recoverable.
\end{prop}

\begin{proof}
   It follows from \cref{thm:CJRreptheory}.
\end{proof}

\begin{prop} \label{prop:existtiltmaxCJR}
    Let $Q$ be an $A_n$ type quiver. Consider $\mathscr{C} \subset \rep(Q)$ a maximal canonically Jordan recoverable subcategory. Then $\mathscr{C}$ contains a tilting representation.
\end{prop}

\begin{proof}
    By \cref{thm:CJRmaxcomb}, consider $\B,\E \subseteq \{1,\ldots,n\}$ such that: \begin{enumerate}[label=$\bullet$,itemsep=1mm]
        \item the pair $(\B,\E[1])$ is a set partition of $\{1,\ldots,n+1\}$; and, 
        \item we have $\mathscr{C} = \opC(\B,\E)$.
    \end{enumerate}  
    Let $\pmb{\mathscr{T}}(\B, \E) \subseteq \opJ(\B,\E) = \Int(\mathscr{C})$ be the set such that: 
    \begin{enumerate}[label=$\bullet$, itemsep=1mm]
        \item for any $e \in \E$, $\llrr{1,e} \in \pmb{\mathscr{T}}(\B,\E)$; and,
        \item for any $b \in \B \setminus \{1\}$,
        \begin{enumerate}[label=$\bullet$,itemsep=1mm]
            \item if, for any $e \in E$, the interval $\llrr{b,e}$ is either on the top or at the bottom of $\llrr{1,n}$, then $\llrr{b,n} \in \pmb{\mathscr{T}}(\B,\E)$; 
            \item otherwise, by setting $e_{(b)} \in \E$ the minimal element in $\E$ such that $\llrr{b,e_{(b)}}$ is neither on the top and at the bottom of $\llrr{1,n}$, then $\llrr{b,e_{(b)}} \in \pmb{\mathscr{T}}(\B,\E)$.
        \end{enumerate}
    \end{enumerate} Set \[T = T_{(\B,\E)} = \bigoplus_{K \in \pmb{\mathscr{T}}(\B,\E)} X_{K}.\]
    Let us show that $T$ is tilting.
    
    First $\rep(Q)$ is a hereditary category. Therefore, to show that $T$ is a tilting representation, it suffices to show that:
    \begin{enumerate}[label=$\bullet$, itemsep=1mm]
        \item $\# \pmb{\mathscr{T}}(\B,\E) = n$; and,
        \item  $\Ext^1(X_K, X_L) = 0$ for all $K,L \in \pmb{\mathscr{T}}(\B,\E)$.
    \end{enumerate}
    The first point is obvious, as one can check that
    \[\#\pmb{\mathscr{T}}(\B,\E) = \#\E + \#\B - 1 = n.\]  
    
    By \cref{thm:exttypeA}, the short exact sequences
    \[\begin{tikzcd}
	0 & D & E  & F &  0,
	\arrow[from=1-1, to=1-2]
	\arrow[tail, from=1-2, to=1-3]
	\arrow[two heads,from=1-3, to=1-4]
	\arrow[from=1-4, to=1-5]
\end{tikzcd}\] where $D,F \in  \ind(Q)$, are parametrized by only two kind of pair of intervals $(K,L) \in \mathcal{I}_n$, namely:
    \begin{enumerate}[label=$(\arabic*)$,itemsep=1mm]
        \item the case where $K \cap L = \varnothing$ and $K \cup L \in \mathcal{I}_n$; or,

        \item the case where $K \cap L \neq \varnothing$, and $K,L,K\cap L$ and $K \cup L$ are four distinct intervals.
    \end{enumerate}
        The case $(1)$ is equivalent to say that $K$ and $L$ are adjacent, and as $\pmb{\mathscr{T}}(\B,\E) \subset \opJ(\B,\E)$ which is adjacency-avoinding, this cannot happen, with $D,F \in \add(T)$ two indecomposable representations.
        
        Assume the case $(2)$ holds with $D,F \in \add(T)$ two indecomposable representations. To exchange the role of $K$ and $L$, assume that $e(K) < e(L)$. Then $b(K) < b(L) \leqslant e(K) < e(L)$. By construction of $\pmb{\mathscr{T}}(\B,\E)$, both of the following assertions holds:
        \begin{enumerate}[label=$\bullet$, itemsep=1mm]
            \item $b(L) \neq 1$;
            \item if $e(L) \neq n$, then $L$ is neither on the top nor at the bottom of $\llrr{1,n}$; and,
            \item $K \cap L \notin \pmb{\mathscr{T}}(\B,\E)$; and,
            \item $K \cap L$ is either on the top or at the bottom of $\llrr{1,n}$, and, therefore, of $K \cup L$.
        \end{enumerate}
         So, in such case, we have $\{D,F\} = \{X_{K}, X_{L}\}$. However, by the last point, we cannot have a non-zero morphism between $X_K$ and $X_L$ which factors through $X_{K \cap L} \oplus X_{K \cup L}$. We sought a contradiction. 
         
        Therefore, any short exact sequence  \[\begin{tikzcd}
	0 & D & E  & F &  0,
	\arrow[from=1-1, to=1-2]
	\arrow[tail, from=1-2, to=1-3]
	\arrow[two heads,from=1-3, to=1-4]
	\arrow[from=1-4, to=1-5]
\end{tikzcd}\] where $D,F \in \add(T)$ are indecomposable, must split, and so $\Ext^1(F, D) = 0$. We proved that $T \in \mathscr{C}$ is tilting.
\end{proof}

\begin{example} \leavevmode \label{ex:A7typeTiltinginMCJR} We continue with \cref{ex:A7typeMCJR}, and we exhibit a tilting representation in this maximal canonically Jordan recoverable following the procedure given in the proof of \cref{prop:existtiltmaxCJR}. 
\end{example}
\begin{figure}[!ht]
    \centering
    \begin{tikzpicture}
                \begin{scope}[line width=.5mm,->, >= angle 60]
                    \node at (-.6,0){$Q=$};
					\node (a) at (0,0){$1$};
					\node (b) at (1,0){$2$};
					\node (c) at (2,0){$3$};
					\node (d) at (3,0){$4$};
                    \node (e) at (4,0){$5$};
                    \node (f) at (5,0){$6$};
                    \node (g) at (6,0){$7$};
					\draw (a) -- (b);
					\draw (c) -- (b);
					\draw (c) -- (d);
                    \draw (d) -- (e);
                    \draw (f) -- (e);
                    \draw (f) -- (g);
				\end{scope}
                \begin{scope}[yshift=-2cm, xshift=-.5cm, line width=.3mm, ->, >= angle 60]
					\node (2) at (0,0){\scalebox{.7}{$\llrr{2}$}};
					\node (12) at (1,1){\scalebox{.7}{$\llrr{1,2}$}};
					\node[rectangle,line width=0.2mm,fill=Plum!20,draw] (2345) at (1,-1){\scalebox{.7}{$\llrr{2,5}$}};
					\node (45) at (0,-2){\scalebox{.7}{$\llrr{4,5}$}};
                    \node[rectangle,line width=0.2mm,fill=Plum!20,draw] (5) at (-1,-3){\scalebox{.7}{$\llrr{5}$}};
                    \node[rectangle,line width=0.2mm,fill=Plum!20,draw]  (567) at (0,-4){\scalebox{.7}{$\llrr{5,7}$}};
                    \node (7) at (-1,-5){\scalebox{.7}{
							$\llrr{7}$}};
                    \node[rectangle, line width=0.2mm, double, rounded corners, fill=orange!20,draw] (56) at (1,-5){\scalebox{.7}{$\llrr{5,6}$}
                                };
                     \node (4) at (3,-5){\scalebox{.7}{$\llrr{4}$}};
                     \node[rectangle, line width=0.2mm, double, rounded corners, fill=orange!20,draw]  (23) at (5,-5){\scalebox{.7}{$\llrr{2,3}$}};
                    \node (1) at (7,-5){\scalebox{.7}{$\llrr{1}$}};
                    \node (456) at (2,-4){\scalebox{.7}{$\llrr{4,6}$}};
                    \node (234) at (4,-4){\scalebox{.7}{$\llrr{2,4}$}};
                     \node[rectangle, line width=0.2mm, double, rounded corners, fill=orange!20,draw] (123) at (6,-4){\scalebox{.7}{$\llrr{1,3}$}};
                     \node (4567) at (1,-3){\scalebox{.7}{$\llrr{4;7}$}};
                     \node[rectangle,line width=0.2mm,fill=Plum!20,draw] (23456) at (3,-3){\scalebox{.7}{$\llrr{2,6}$}};
                     \node (1234) at (5,-3){\scalebox{.7}{$\llrr{1,4}$}};
                     \node[rectangle,line width=0.2mm,fill=Plum!20,draw] (3) at (7,-3){\scalebox{.7}{$\llrr{3}$}};
					\node[rectangle, line width=0.2mm, double, rounded corners, fill=orange!20,draw] (12345) at (2,0){\scalebox{.7}{$\llrr{1,5}$}};
					\node[rectangle,line width=0.2mm,fill=Plum!20,draw] (234567) at (2,-2){\scalebox{.7}{$\llrr{2,7}$}};
                    \node[rectangle, line width=0.2mm, double, rounded corners, fill=orange!20,draw] (123456) at (4,-2){\scalebox{.7}{$\llrr{1,6}$}};
					\node (34) at (6,-2){\scalebox{.7}{$\llrr{3,4}$}};
					\node[rectangle, line width=0.2mm, double, rounded corners, fill=orange!20,draw] (1234567) at (3,-1){\scalebox{.7}{$\llrr{1,7}$}};
                    \node[rectangle,line width=0.2mm,fill=Plum!20,draw] (3456) at (5,-1){\scalebox{.7}{$\llrr{3,6}$}};
                    \node[rectangle,line width=0.2mm,fill=Plum!20,draw] (34567) at (4,0){\scalebox{.7}{$\llrr{3,7}$}};
                    \node (6) at (6,0){\scalebox{.7}{$\llrr{6}$}};
                    \node[rectangle, line width=0.2mm, double, rounded corners, fill=orange!20,draw] (345) at (3,1){\scalebox{.7}{$\llrr{3,5}$}};
                    \node (67) at (5,1){\scalebox{.7}{$\llrr{6,7}$}};
					\draw (2) -- (12);
					\draw (2) -- (2345);
					\draw (45) -- (2345);
					\draw (12) -- (12345);
					\draw (2345) -- (12345);
					\draw (2345) -- (234567);
                    \draw (12345) -- (345);

                    \draw (5) -- (45);
                    \draw (7) -- (567);
                    \draw (567) -- (4567);
                    \draw (4567) -- (234567);
                    \draw (234567) -- (1234567);
                    \draw (1234567) -- (34567);
                    \draw (34567) -- (67);
                    \draw (56) -- (456);
                    \draw (456) -- (23456);
                    \draw (23456) -- (123456);
                    \draw (123456) -- (3456);
                    \draw (3456) -- (6);
                    \draw (4) -- (234);
                    \draw (234) -- (1234);
                    \draw (1234) -- (34);
                    \draw (23) -- (123);
                    \draw (123) -- (3);
					
					\draw (5) -- (567);
                    \draw (567) -- (56);
                    \draw (45) -- (4567);
                    \draw (4567) -- (456);
                    \draw (456) -- (4);
                    \draw (234567) -- (23456);
                    \draw (23456) -- (234);
                    \draw (234) -- (23);
                    \draw (12345) -- (1234567);
                    \draw (1234567) -- (123456);
                    \draw (123456) -- (1234);
                    \draw (1234) -- (123);
                    \draw (123) -- (1);
                    \draw (345) -- (34567);
                    \draw (34567) -- (3456);
                    \draw (3456) -- (34);
                    \draw (34) -- (3);
                    \draw (67) -- (6);	
				\end{scope}
    \end{tikzpicture}
    \caption[fragile]{Example of a tilting representation $T = \oplus\ \protect\begin{tikzpicture}[baseline={(0,-.1)}]
        \node[rectangle, line width=0.2mm, double, rounded corners, fill=orange!20,minimum size=1em,draw] at (0,0){$\ \ $};
    \end{tikzpicture}$ in the maximal canonically Jordan recoverable subcategory given in \cref{fig:MaxCJR1}.}
    \label{fig:MaxCJR2}
\end{figure}

\begin{prop}
\label{propGSCJR}
 Let $n \in \mathbb{N}^*$ and $Q$ be an $A_n$ type quiver. Consider $T \in \rep(Q)$ a rigid representation. Then $\GS_{\mathcal{E_\diamond}}(T)$ is canonically Jordan recoverable.
\end{prop}

\begin{proof}
It follows from \cref{prop:EAdapt} and \cref{thm:CJRreptheory}.
\end{proof}

We can finally prove our main results.

\begin{theorem}
\label{th:mCJR=GS}
Let $Q$ be an $A_n$ type quiver for some $n \in \mathbb{N}^*$. A subcategory $\mathscr{C} \subset \rep(Q)$ is a maximal canonically Jordan recoverable subcategory if, and only if, there exists $T \in \Tilt(Q)$ such that $\mathscr{C} = \GS_{\mathcal{E_\diamond}}(T)$.
\end{theorem}

\begin{proof}
Let $\mathscr{C}$ be a maximal canonically Jordan recoverable subcategory of $\rep(Q)$. By \cref{prop:existtiltmaxCJR}, there exists a tilting representation $T \in \rep(Q)$ such that $T \in \mathscr{C}$. By \cref{thm:CJRreptheory}, the category $\mathscr{C}$ is $\mathcal{E}_\diamond$-adapted. Moreover, $\mathscr{C}$ is closed under extensions by \cref{prop:maxCJRextclos}. Therefore, we have $\GS_{\mathcal{E}_{\diamond}}(T) \subseteq \mathscr{C}$. By \cref{thm:maximal}, we get $\GS_{\mathcal{E}_\diamond}(T) = \mathscr{C}$. The converse is obvious by \cref{thm:maximal} and \cref{thm:CJRreptheory}.
\end{proof}

\begin{cor} \label{cor:CJRGSlink}
Let $Q$ be an $A_n$ type quiver for some $n \in \mathbb{N}^*$. Let $\mathscr{C} \subseteq \rep (Q)$ be a maximal canonically Jordan recoverable subcategory. Then there exists $T \in \Tilt(Q)$ such that \[\mathscr{C} = \add \left( \bigoplus_{T' \in [T]_{\approx_{\mathcal{E}_\diamond}}} T' \right). \]
\end{cor}

\begin{proof}
It follows directly from \cref{th:mCJR=GS,th:Treaching}.
\end{proof}

\section{To go further}
\label{sec:further}
\pagestyle{plain}

This section provides an overview of a few directions we are interested in investigating in the near future. The reader is invited to have a look at these various or related problems.

\subsection{Cokernel-Kernel operators}

This section introduces a closure operator on subcategories of $(\mathscr{A},\mathcal{E})$ that is slightly different from the Gen-Sub operator. More precisely, given a subcategory $\mathscr{C}$, we will construct the smallest subcategory $\mathscr{D} \supseteq \mathscr{C}$ such that:
\begin{enumerate}[label=$\bullet$,itemsep=1mm]
    \item $\mathscr{D}$ is closed under cokernels of $\mathcal{E}$-monomorphisms in $\mathscr{D}$; and,
    \item  $\mathscr{D}$ is closed under kernels of $\mathcal{E}$-epimorphisms in $\mathscr{D}$.
\end{enumerate}

The \new{$\Coker_\mathcal{E}$-operator} on subcategories of $\mathscr{A}$ is defined as follows: given a subcategory $\mathscr{C} \subseteq \mathscr{A}$, we define $\Coker_\mathcal{E}(\mathscr{C})$ as the full subcategory, closed under sums and summands, of $\mathscr{A}$ containing the cokernel of any $\mathcal{E}$-monomorphism $f : X \longrightarrow Y$ such that $X,Y \in \mathscr{C}$. The \new{$\Ker_\mathcal{E}$-operator} on subcategories of $\mathscr{A}$ is defined analogously: given a subcategory $\mathscr{C} \subseteq \mathscr{A}$, we define $\Ker_\mathcal{E}(\mathscr{C})$ as the full subcategory, closed by sums and summands, of $\mathscr{A}$ containing the kernel of any $\mathcal{E}$-epimorphism $g : Y \longrightarrow X$  such that $Y,X \in \mathscr{C}$. Given an object $M \in \mathscr{C}$, by abusing of notation, we set $\Coker_\mathcal{E}(M) = \Coker_\mathcal{E}(\add(M))$ and $\Ker_\mathcal{E}(M) = \Ker_\mathcal{E}(\add(M))$.

Given a subcategory $\mathscr{C} \subseteq \mathscr{A}$, we define a sequence of subcategories $(\CK_{\mathcal{E}}^i(\mathscr{C}))_{i\in \mathbb{N}}$ as it follows:
\begin{enumerate}[label=$\bullet$, itemsep=1mm]
    \item we set $\CK_{\mathcal{E}}^0(\mathscr{C}) = \mathscr{C}$; and,
    \item for all $i \geqslant 1$, we define $\CK^{i}_{\mathcal{E}}(\mathscr{C})$ to be the subcategory of $\mathscr{A}$ additively generated by objects in both $\Coker_{\mathcal{E}}(\CK^{i-1}_{\mathcal{E}}(\mathscr{C}))$ and $\Ker_{\mathcal{E}}(\CK^{i-1}_{\mathcal{E}}(\mathscr{C}))$.
\end{enumerate}

By the previous remark, $(\CK_{\mathcal{E}}^i(\mathscr{C}))_{i\in \mathbb{N}}$ is an increasing sequence of additive subcategories of $\mathscr{A}$. Therefore, as $\mathscr{A}$ is abelian, this sequence of subcategories admits a colimit.

\begin{definition} \label{def:CKop}
    The \new{$\CK_\mathcal{E}$-operator} on subcategories of $\mathscr{A}$ is defined as follows: given a subcategory $\mathscr{C} \subseteq \mathscr{A}$, we define $\CK_{\mathcal{E}}(\mathscr{C})$ as the colimit of the sequence of subcategories $(\CK_\mathcal{E}^i(\mathscr{C}))_{i\in \mathbb{N}}$. As before, we set $\CK_\mathcal{E}(M) = \CK_\mathcal{E}(\add(M))$ by abuse of notations.
\end{definition}

\begin{remark}
It is clear that $\CK_{\mathcal{E}}(\mathscr{C}) \subseteq \GS_{\mathcal{E}}(\mathscr{C})$. Therefore, by~\cref{prop:EAdapt}, $\CK_{\mathcal{E}}(T)$ is $\mathcal{E}$-adapted for any rigid object $T$.

\end{remark}

\begin{ex} \label{ex:A7typeCK} In \cref{fig:CKCalc1}, we calculate $\CK_{\mathcal{E}_\diamond}(\mathscr{C})$ for a subcategory $\mathscr{C}$ of $\rep(Q)$ and a fixed quiver $Q$ of $A_7$ type. This example illustrates the slight difference between $\CK_{\mathcal{E}_\diamond}(\mathscr{C})$ and $\GS_{\mathcal{E}_\diamond}(\mathscr{C})$.

\begin{figure}[!ht]
    \leavevmode\\[-0.5ex]  % forces heading to appear before content
    \centering
    \begin{tikzpicture}
				\begin{scope}[line width=.5mm,->, >= angle 60]
                    \node at (-.6,0){$Q=$};
					\node (a) at (0,0){$1$};
					\node (b) at (1,0){$2$};
					\node (c) at (2,0){$3$};
					\node (d) at (3,0){$4$};
                    \node (e) at (4,0){$5$};
                    \node (f) at (5,0){$6$};
                    \node (g) at (6,0){$7$};
					\draw (a) -- (b);
					\draw (c) -- (b);
					\draw (c) -- (d);
                    \draw (d) -- (e);
                    \draw (f) -- (e);
                    \draw (f) -- (g);
				\end{scope}
                \begin{scope}[yshift=-2cm, xshift=-.5cm, line width=.3mm, ->, >= angle 60]
					\node (2) at (0,0){\scalebox{.7}{$\llrr{2}$}};
					\node (12) at (1,1){\scalebox{.7}{$\llrr{1,2}$}};
					\node (2345) at (1,-1){\scalebox{.7}{$\llrr{2,5}$}};
					\node (45) at (0,-2){\scalebox{.7}{$\llrr{4,5}$}};
                    \node[circle,line width=0.2mm,double,fill=lava!20,draw] (5) at (-1,-3){\scalebox{.7}{$\llrr{5}$}};
                    \node[circle,line width=0.2mm,double,fill=lava!20,draw] (567) at (0,-4){\scalebox{.7}{$\llrr{5,7}$}};
                    \node (7) at (-1,-5){\scalebox{.7}{
							$\llrr{7}$}};
                    \node (56)[circle,line width=0.2mm,double,fill=lava!20,draw] at (1,-5){\scalebox{.7}{$\llrr{5,6}$}
                                };
                     \node (4) at (3,-5){\scalebox{.7}{$\llrr{4}$}};
                     \node[circle,line width=0.2mm,double,fill=lava!20,draw] (23) at (5,-5){\scalebox{.7}{$\llrr{2,3}$}};
                    \node (1) at (7,-5){\scalebox{.7}{$\llrr{1}$}};
                    \node (456) at (2,-4){\scalebox{.7}{$\llrr{4,6}$}};
                    \node (234) at (4,-4){\scalebox{.7}{$\llrr{2,4}$}};
                     \node (123) at (6,-4){\scalebox{.7}{$\llrr{1,3}$}};
                     \node (4567) at (1,-3){\scalebox{.7}{$\llrr{4;7}$}};
                     \node(23456) at (3,-3){\scalebox{.7}{$\llrr{2,6}$}};
                     \node (1234) at (5,-3){\scalebox{.7}{$\llrr{1,4}$}};
                     \node (3) at (7,-3){\scalebox{.7}{$\llrr{3}$}};
					\node[circle,line width=0.2mm, double,fill=lava!20,draw] (12345) at (2,0){\scalebox{.7}{$\llrr{1,5}$}};
					\node (234567) at (2,-2){\scalebox{.7}{$\llrr{2,7}$}};
                    \node[circle,line width=0.2mm,fill=darkgreen!20,draw] (123456) at (4,-2){\scalebox{.7}{$\llrr{1,6}$}};
					\node (34) at (6,-2){\scalebox{.7}{$\llrr{3,4}$}};
					\node[circle,line width=0.2mm,fill=darkgreen!20,draw] (1234567) at (3,-1){\scalebox{.7}{$\llrr{1,7}$}};
                    \node[circle,line width=0.2mm,fill=darkgreen!20,draw] (3456) at (5,-1){\scalebox{.7}{$\llrr{3,6}$}};
                    \node[circle,line width=0.2mm,fill=darkgreen!20,draw] (34567) at (4,0){\scalebox{.7}{$\llrr{3,7}$}};
                    \node (6) at (6,0){\scalebox{.7}{$\llrr{6}$}};
                    \node[circle,line width=0.2mm,double,fill=lava!20,draw] (345) at (3,1){\scalebox{.7}{$\llrr{3,5}$}};
                    \node (67) at (5,1){\scalebox{.7}{$\llrr{6,7}$}};
					\draw (2) -- (12);
					\draw (2) -- (2345);
					\draw (45) -- (2345);
					\draw (12) -- (12345);
					\draw (2345) -- (12345);
					\draw (2345) -- (234567);
                    \draw (12345) -- (345);

                    \draw (5) -- (45);
                    \draw (7) -- (567);
                    \draw (567) -- (4567);
                    \draw (4567) -- (234567);
                    \draw (234567) -- (1234567);
                    \draw (1234567) -- (34567);
                    \draw (34567) -- (67);
                    \draw (56) -- (456);
                    \draw (456) -- (23456);
                    \draw (23456) -- (123456);
                    \draw (123456) -- (3456);
                    \draw (3456) -- (6);
                    \draw (4) -- (234);
                    \draw (234) -- (1234);
                    \draw (1234) -- (34);
                    \draw (23) -- (123);
                    \draw (123) -- (3);
					
					\draw (5) -- (567);
                    \draw (567) -- (56);
                    \draw (45) -- (4567);
                    \draw (4567) -- (456);
                    \draw (456) -- (4);
                    \draw (234567) -- (23456);
                    \draw (23456) -- (234);
                    \draw (234) -- (23);
                    \draw (12345) -- (1234567);
                    \draw (1234567) -- (123456);
                    \draw (123456) -- (1234);
                    \draw (1234) -- (123);
                    \draw (123) -- (1);
                    \draw (345) -- (34567);
                    \draw (34567) -- (3456);
                    \draw (3456) -- (34);
                    \draw (34) -- (3);
                    \draw (67) -- (6);	
				\end{scope}
    \end{tikzpicture}
    \caption[fragile]{Calculation of $\CK_{\mathcal{E}_\diamond}(\mathscr{C})$ where $\mathscr{C} = \add \left( \protect\begin{tikzpicture}[baseline={(0,-.1)}]
        \node[circle, line width=0.2mm, double, fill=lava!20,minimum size=1em,draw] at (0,0){$ $};
    \end{tikzpicture}\right)$. We get $\CK_{\mathcal{E}_\diamond}^1(\mathscr{C}) = \add \left(\protect\begin{tikzpicture}[baseline={(0,-.1)}]
        \node[circle, line width=0.2mm, double, fill=lava!20,minimum size=1em,draw] at (0,0){$ $};
    \end{tikzpicture} + \protect\begin{tikzpicture}[baseline={(0,-.1)}]
        \node[circle, line width=0.2mm, fill=darkgreen!20,minimum size=1em,draw] at (0,0){$ $};
    \end{tikzpicture}\right)$. Here, we obtain that $\CK_{\mathcal{E}_\diamond}(\mathscr{C}) = \CK_{\mathcal{E}_\diamond}^1(\mathscr{C})$.}
    \label{fig:CKCalc1}
\end{figure}
\end{ex}

For the same subcategory $\mathscr{C}$, $\GS_{\mathcal{E}_\diamond}(\mathscr{C})$ still corresponds to the subcategory shown in Figure~\ref{fig:GSCalc1}. By adding the object $\llrr{1,3}$ to $\mathscr{C}$, thereby making it maximally rigid, we obtain $\GS_{\mathcal{E}_\diamond}(\mathscr{C}) = \CK_{\mathcal{E}_\diamond}(\mathscr{C})$. We observe that, in general, a minimal basic object $T$ for which $\GS_{\mathcal{E}}(T)$ is a maximally $\mathcal{E}$-adapted extension-closed subcategory does not necessarily have to be tilting. This suggests that studying the $\CK_\mathcal{E}$-operator could be valuable.

\begin{remark}
By minimal, we mean that no direct summand can be removed from $T$ basic without losing the property that $\GS_{\mathcal{E}}(T)$ is a maximally $\mathcal{E}$-adapted extension-closed subcategory.
\end{remark}

\begin{definition}
We recall that a subcategory $\mathscr{D}$ is called \new{$\mathcal{E}$-wide} if, for any short $\mathcal{E}$-exact sequence
\[
\begin{tikzcd}
	\xi:\quad 0 \arrow[r] & E \arrow[r, tail] & F \arrow[r, two heads] & G \arrow[r] & 0,
\end{tikzcd}
\]
whenever two of the three objects $E$, $F$, or $G$ belong to $\mathscr{D}$, then so does the third.
\end{definition}

\begin{prop}
Let $(\mathscr{A}, \mathcal{E})$ be an exact category with $\mathscr{A}$ hereditary. Then the following statements hold:
\begin{enumerate}[label=$(\alph*)$, itemsep=1mm]
    \item \label{aCKGS} if $T \in \mathscr{A}$ is rigid, then $\CK_{\mathcal{E}}(T)$ is an $\mathcal{E}$-adapted and $\mathcal{E}$-wide subcategory;
    \item \label{bCKGS} if $T \in \mathscr{A}$ is tilting, then $\CK_{\mathcal{E}}(T) = \GS_{\mathcal{E}}(T)$. 
\end{enumerate}
\end{prop}

\begin{proof}
The proof of \ref{aCKGS} follows as in \cref{prop:Extclosed}, noting that the subcategory is no longer necessarily closed under arbitrary $\mathcal{E}$-subobjects and $\mathcal{E}$-quotients, but it remains closed under taking kernels of $\mathcal{E}$-epimorphisms and cokernels of $\mathcal{E}$-monomorphisms within itself.

The proof of \cref{thm:maximal} can be adapted to the setting of $\CK_{\mathcal{E}}(T)$ for showing \ref{bCKGS}. Together with the result from \ref{aCKGS}, this implies that both $\CK_{\mathcal{E}}(T)$ and $\GS_{\mathcal{E}}(T)$  are maximal $\mathcal{E}$-adapted extension-closed subcategories with $\CK_{\mathcal{E}}(T) \subseteq \GS_{\mathcal{E}}(T)$, which yields the desired equality.
\end{proof}

\begin{conj} \label{conj:TiltasmingenCK}
Let $(\mathscr{A},\mathcal{E})$ be an exact category with $\mathscr{A}$ hereditary. Then, $T \in \mathscr{A}$ is minimal such that $\CK_{\mathcal{E}}(T)$ forms a maximally $\mathcal{E}$-adapted extension-closed subcategory if and only if $T$ is tilting.  
\end{conj}

\subsection{Congruences on the tilting lattice} \label{ss:tiltposet} The equivalence relation $\approx_{\mathcal{E}}$ seems to hide other interesting properties in the lattice $(\Tilt(Q), \Tleq)$. Considering the equivalence relation $\approx_{\mathcal{E}}$, we can ask ourselves if it defines well-behaved lattice quotients, usually called \emph{lattice congruences}. 
\begin{definition} \label{def:latticecongr}
    Let $(L, \leqslant, \vee, \wedge)$. A \new{congruence} of $L$ is an equivalence relation $\sim$ on $L$ such that, for any $a,b,x,y \in L$, if $a \sim x$ and $b \sim y$, then both $a \vee b \sim x \vee y$ and $a \wedge b \sim x \wedge y$. The \new{quotient lattice} of $L$ on $\sim$ is the lattice $(L / \sim, \leqslant, \vee, \wedge)$ induced by $L$.
\end{definition}
N. Reading gives a criterion on equivalence relations that are congruences on a lattice without using join and meet operations \cite{R04}.

For any $n \in \mathbb{N}$, denote by $\mathcal{B}_n$ the Boolean lattice of order $n$. In \cref{fig:posettiltEdiamond}, we can check that $\approx_{\mathcal{E}_\diamond}$ defines a congruence on $(\Tilt(Q), \Tleq)$, and the lattice congruence obtained is isomorphic to the boolean lattice $\mathcal{B}_3$. In light of this example, we made a stronger conjecture that encapsulates the general behavior of $\approx_{\mathcal{E}_\diamond}$.

\begin{conj} \label{conj:latticequotient} Let $Q$ be an $A_n$ type quiver for some $n \in \mathbb{N}^*$. The equivalence relation $\approx_{\mathcal{E}_\diamond}$ defines a congruence on $(\Tilt(Q), \Tleq)$, and the obtained lattice congruence quotient is isomorphic to the Boolean lattice $\mathcal{B}_{n-1}$  
\end{conj}

\begin{example} \label{ex:latticequotientEdiamond}
  In \cref{fig:posettiltEdiamond}, the equivalence relation $\approx_{\mathcal{E}_{\diamond}}$ gives a lattice congruence, given by identifying gathered tilting representations, and the quotient is isomorphic to $\mathcal{B}_3$.  
\end{example}

A rough plan to prove this conjecture is as follows. First, we show that $\approx_{\mathcal{E}_\diamond}$ defines a congruence on $(\Tilt(Q), \Tleq)$ by using Reading's criterion. Then, we can easily count the number of equivalence classes of $\approx_{\mathcal{E}_\diamond}$ in $\Tilt(Q)$.

\begin{lemma} \label{lem:numberofeqclassEdiamond}
    Let $Q$ be an $A_n$ type quiver for some $n \in \mathbb{N}^*$. There are exactly $2^{n-1}$ equivalence classes for $\approx_{\mathcal{E}_\diamond}$ in $\Tilt(Q)$.
\end{lemma}

\begin{proof}
    By \cref{thm:CJRmaxcomb,th:mCJR=GS}, those equivalence classes are parametrized by pairs $(\B,\E)$ such that $\{\B,\E[1]\}$ is a bipartition of $\{1,\ldots,n+1\}$. Then there is a one-to-one correspondence $\phi$ from subsets of $\{2,\ldots,n\}$ to such pairs given by \[\forall \A \subseteq \{2,\ldots,n\},\ \phi(\A) = (\{1\} \cup \A,   \{2,\ldots, n+1\}) \setminus \A.\]
    We got the desired result.
\end{proof}

Finally, we can construct the lattice isomorphism explicitly. Note that the order of subsets of $\{2, \ldots, n\}$ induced by $(\Tilt(Q) / \approx_{\mathcal{E}_\diamond}, \Tleq)$ does not coincide with the usual inclusion order on those subsets. 

\begin{remark} \label{rem:trueforlinear}
   If $Q$ is a linearly-oriented $A_n$ type quiver, then $(\Tilt(Q), \Tleq)$ is isomorphic to the classical order on \emph{binary trees} with $n$ internal nodes, and $\approx_{\mathcal{E}_\diamond}$ identifies trees by their \emph{canopy} \cite{V09}. In this special case, the conjecture was proved by J.-L. Loday and M. O. Racia  \cite{LR98,L04}. 
\end{remark}

We aim to prove this conjecture soon. We can also try to characterize exact structures $\mathcal{E}$ on a type $ADE$ quiver $Q$ such that $\approx_\mathcal{E}$ defines a congruence on $\Tilt(Q)$. It could provide a deeper understanding of a subfamily of well-known lattice congruences within the Cambrian lattice.

\section*{Acknowledgements}

The authors extend their gratitude to the organizers of the \emph{CHARMS Summer School}, where our discussions on this project began, namely J. Besson, P. Bodin, E.D. Børve, R. Canesin, M. Garcia, E. Gupta, M. Schoonheere, and V. Soto. They would like to express their particular appreciation to M. Garcia and M. Schoonheere for the numerous further discussions they had on this project. They thank T. Brüstle and Y. Palu for their helpful advice. They also acknowledge \textit{CHARMS program grant (ANR-19CE40-0017-02)} for their funding support. 

B.D. would like to acknowledge the \textit{Engineering and Physical Sciences Research Council (EP/W007509/1)} for their partial funding support. 

S.R. is supported by the B2X FRQNT grant.

%\nocite{*}
\bibliography{Article1}
\bibliographystyle{alpha}

\end{document}